\documentclass[a4paper, 10pt, twoside, leqno]{article}

\usepackage{amsmath,amsthm,amssymb}
\usepackage[T1]{fontenc}
\usepackage[utf8x]{inputenc}
\usepackage{mathrsfs,esint,bbm}
\usepackage{upgreek}
\usepackage{fixltx2e} 
\usepackage[english]{babel}
\usepackage{fancyhdr}
\usepackage[all,cmtip]{xy}
\usepackage{authblk}

\usepackage{enumitem}

\newtheorem{theor}{Theorem}
\newtheorem{cor}[theor]{Corollary}
\newtheorem{lemma}[theor]{Lemma}
\newtheorem{prop}[theor]{Proposition}

\theoremstyle{definition}
\newtheorem{defn}{Definition}

\theoremstyle{remark}

\newtheorem{remark}{Remark}

\newcommand{\N}{\mathbb{N}}        
\newcommand{\R}{\mathbb{R}}        
\newcommand{\C}{\mathbb{C}}        
\newcommand{\PR}{\mathbb{RP}^2}
\renewcommand{\SS}{\mathscr{S}}
\newcommand{\Sl}{\mathscr{S}_{\mathrm{line}}}
\newcommand{\Spt}{\mathscr{S}_{\mathrm{pts}}}
\renewcommand{\H}{\mathscr{H}}
\newcommand{\n}{\mathbf{n}}
\newcommand{\m}{\mathbf{m}}
\newcommand{\p}{\mathbf{p}}
\newcommand{\Mat}{\mathrm{M}_3(\mathbb{R})}
\newcommand{\NN}{\mathscr{N}}
\newcommand{\MM}{\mathscr{M}}
\newcommand{\RR}{\varrho}
\newcommand{\Cs}{\mathscr{C}}

\newcommand{\Sz}{\mathbf{S}_0}
\newcommand{\Qe }{Q_\varepsilon}

\DeclareMathOperator{\infess}{ess\,inf} 
 
\DeclareMathOperator{\Id}{Id}
\DeclareMathOperator{\tr}{tr}
\DeclareMathOperator{\dist}{dist}
\DeclareMathOperator{\sign}{sign}

\renewcommand{\d}{\mathrm{d}}
\newcommand{\D}{\mathrm{D}}
\renewcommand{\u}{\mathbf{u}}
\newcommand{\w}{\mathbf{w}}

\newcommand{\X}{\mathbf{X}}
\renewcommand{\r}{\mathbf{r}}
\newcommand{\e}{\mathbf{e}}

\newcommand{\abs}[1]{\left| #1 \right|}
\newcommand{\norm}[1]{\left\| #1 \right\|}

\newcommand{\one}{\mathbbm{1}}
\newcommand{\mres}{\mathbin{\vrule height 1.6ex depth 0pt width 0.13ex\vrule height 0.13ex depth 0pt width 1.3ex}}
\newcommand{\csubset}{\subset\!\subset}

\renewcommand{\v}{\mathbf{v}}

\let\oldS\S
\newcommand{\Stipografico}{\oldS}
\renewcommand{\S}{\mathbb{S}}

\usepackage[usenames,dvipsnames]{pstricks}
\usepackage{epsfig}
\usepackage{pst-grad} 
\usepackage{pst-plot} 

\usepackage{color}

\begin{document}

\title
{\textbf{Line defects in the small elastic constant limit \\ of a three-dimensional Landau-de Gennes model}}

\author{Giacomo CANEVARI}
\affil{Mathematical Institute, University of Oxford, \\
Andrew Wiles Building, Radcliffe Observatory Quarter,\\
Woodstock Road, Oxford~OX2~6GG, United Kingdom \\
\emph{\textbf{E-mail address}: \texttt{canevari@maths.ox.ac.uk}}}
\renewcommand\Affilfont{\itshape}

\date{\today}

\pagestyle{fancy}
\fancyhead{} 
\fancyfoot{} 
\fancyhead[CE]{\textsc{\lowercase{Giacomo Canevari}}}
\fancyhead[LE]{\thepage}
\fancyhead[CO]{\textsc{\lowercase{Line defects in the limit of a {\small 3}D Landau-de Gennes model}}}
\fancyhead[RO]{\thepage}
\renewcommand{\headrulewidth}{0pt}
\renewcommand{\footrulewidth}{0pt}

\thispagestyle{plain}
\maketitle

\begin{abstract} 
 We consider the Landau-de Gennes variational model for nematic liquid crystals, in three-di\-men\-sio\-nal domains.
 More precisely, we study the asymptotic behaviour of minimizers as the elastic constant tends to zero, under the assumption that
 minimizers are uniformly bounded and their energy blows up as the logarithm of the elastic constant.
 We show that there exists a closed set~$\Sl$ of finite length, 
 such that minimizers converge to a locally harmonic map away from~$\Sl$.
 Moreover, $\Sl$ restricted to the interior of the domain is a locally finite union of straight line segments.
 We provide sufficient conditions, depending on the domain and the boundary data, under which our main results apply.
 We also discuss some examples.
 
\end{abstract}

\section{Introduction}
\label{sect:intro}

\subsection{Variational theories for nematic liquid crystals}

A nematic liquid crystal is matter in an intermediate state between liquid and crystalline solid.
Molecules can flow and their centers of mass are randomly distributed, but the molecular axes tend to self-align locally.
As a result, the material is anisotropic with respect to optic and electromagnetic properties.
In the so-called uniaxial nematics, the molecules are often rod-shaped and, although they may carry a permanent dipole, there are as many dipoles `up' as there are `down'.
Therefore, the material symmetry group contains the rotations around the molecular axis and the reflection symmetry which exchanges the two orientations of the axis.
The long-range orientational order due to the self-alignmennt of the molecules is broken in some places, called \emph{defects}.
The word nematic itself refers to the line defects (see Friedel,~\cite{Friedel}):
\begin{quote}
\emph{I am going to use the term nematic ($\nu{\acute\eta}\mu\alpha$, thread)
to describe the forms, bodies, phases, etc. of the second type\ldots~because of the linear discontinuities, which are twisted like threads, 
and which are one of their most prominent characteristics.}
\end{quote}
In addition to line defects, also called \emph{disclinations}, nematic media exhibit ``hedgehog-like'' point singularities.
According to the topological theory of ordered media (see e.g.~\cite{Mermin, ToulouseKleman, VolovikMineev}), both kinds of defects are described by the homotopy groups of a manifold, 
which parametrizes the possible local configurations of the material.

Several models for uniaxial nematic liquid crystals have drawn the attention of the mathematical community.
The most popular continuum theories which are based on a finite-dimensional order parameter space are probably the Oseen-Frank, the Ericksen and the Landau-de Gennes theories.
In the Oseen-Frank theory~\cite{Frank-LC}, the material is modeled by a unit vector field~$\n = \n(x) \in\S^2$, which represents the preferred direction of molecular alignment.
The elastic energy, in the simplest setting, reduces to the Dirichlet functional
\begin{equation} \label{dirichlet_energy}
 E(\n) := \frac12 \int_\Omega \abs{\nabla \n}^2 ,
\end{equation}
where~$\Omega\subseteq\R^3$ is the physical domain.
In this case, least-energy configurations are but harmonic maps~$\n\colon\Omega\to\S^2$.
As such, minimizers have been widely studied in the literature (the reader is referred to e.g.~\cite{HeleinWood} for a general review of this subject).
Schoen and Uhlenbeck~\cite{SchoenUhlenbeck} proved that minimizers are smooth away from a discrete set of points singularities.
Brezis, Coron and Lieb~\cite{BrezisCoronLieb} investigated the precise shape of minimizers around a point defect~$x_0$, and proved that
\begin{equation} \label{hedgehog}
 \n(x) \simeq \pm R\frac{x - x_0}{|x - x_0|}  \qquad \textrm{for } |x - x_0| \ll 1,
\end{equation}
where~$R$ is a rotation.
These ``hedgehog-like'' point defects are associated with a non-trivial homotopy class of maps $\n\colon \partial B_r(x_0) \to \S^2$, i.e. a non-trivial element of~$\pi_2(\S^2)$.
Interesting results are also available for the full Oseen-Frank energy, which consists of various terms accounting for splay, twist and bend deformations.
Hardt, Kinderlehrer and Lin~\cite{HKL} proved the existence of minimizers and partial regularity, i.e. regularity out of an exceptional set whose Hausdorff dimension is strictly less than $1$.
As for the local behaviour of minimizers around the defects, the picture is not as clear as for the Dirichlet energy~\eqref{dirichlet_energy},
but at least the stability of ``hedgehog-like'' singularities such as~\eqref{hedgehog} as been completely analyzed (see~\cite{KinderlehrerOuWalkington} and the references therein).
However, the partial regularity result of~\cite{HKL} implies that the Oseen-Frank theory cannot account for line defects.

Ericksen theory is less restrictive, because it allows variable orientational order.
Indeed, the configurations are described by a pair~$(s, \, \n)\in \R\times\S^2$, where~$\n$ is the preferred direction of molecular alignment and the scalar~$s$ measures the degree of ordering.
In this theory, defects are identified by the condition~$s=0$, which correspond to complete disordered states.
Under suitable assumptions, minimizers can exhibit line singularities and even planar discontinuities (see~\cite[Theorem~7.2]{Lin1991b}).
Explicit examples were studied by Ambrosio and Virga~\cite{AmbrosioVirga} and Mizel, Roccato and Virga~\cite{MizelRoccatoVirga}.
However, the Ericksen theory --- as the Oseen-Frank theory --- excludes configurations which might have physical reality.
Ericksen himself was aware of this, since he presented his theory as a ``kind of compromise''~\cite[p.~98]{Ericksen-LC} between physical intuition and mathematical simplicity.
Indeed, both the Oseen-Frank and the Ericksen theory do not take into account the material symmetry, that is, the configurations represented by~$\n$ and~$-\n$ are physically indistinguishable. 
Moreover, these theories postulate that, at each point of the medium, there is at most one preferred direction of molecular orientation.
Configurations for which such a preferred direction exists are called \emph{uniaxial}, because they have one axis of rotational symmetry.
If no preferred direction exists, the configuration is called \emph{isotropic} (in the Ericksen theory, this corresponds to $s = 0$).

The Landau-de Gennes theory~\cite{deGennes} allows for a rather complete description of the local behaviour of the medium, because it accounts for 
\emph{biaxial}\footnote{Here ``uniaxial'' and ``biaxial'' refer to \emph{arrangements} of molecules, not to the molecules themselves which are always assumed to be uniaxial.} configurations as well.
A state is called biaxial when it has no axis of rotational symmetry, but three orthogonal axes of reflection symmetry instead
(see~\cite{MottramNewton} for more details).
What makes the Landau-de Gennes theory so rich is the order parameter space.
Configurations are described by matrices (the so-called $Q$-tensors), 
which can be interpreted as renormalized second-order moments of the microscopic distribution of molecules with respect to the orientation.

In this paper, we aim at describing the generation of line defects for nematics in three-dimensional domains from a variational point of view, within the Landau-de Gennes theory.
Two main simplifying assumptions are postulated here.
First, we neglect the effect of external electromagnetic fields.
To induce non-trivial behaviour in minimizers, we couple the problem with non-homogeneous Dirichlet boundary conditions (strong anchoring).
Second, we adopt the one-constant approximation, that is we drop out several terms in the expression of the elastic energy, and we are left with the gradient-squared term only.
These assumptions, which drastically reduce the technicality of the problem, are common in the mathematical literature on this subject
(see e.g.~\cite{DiFratta, GartlandMkaddem, HenaoMajumdar, INSZ-Hedgehog, Lamy, MajumdarZarnescu}).
For the two-dimensional case, the analysis of the analogous problem is presented in~\cite{pirla, GolovatyMontero}.

\subsection{The Landau-de Gennes functional}

As we mentioned before, the local configurations of the medium are described by $Q$-tensors, i.e. elements of
\[
 \Sz := \left\{Q\in\Mat \colon Q^{\mathsf{T}} = Q, \ \tr Q = 0 \right\} .
\]
This is a real linear space, of dimension five, which we endow with the scalar product $Q \cdot P := Q_{ij}P_{ij}$ (Einstein's convention is assumed).
This choice of the configurations space can be justified as follows.
At a microscopic scale, the distribution of molecules around a given point~$x\in \Omega$, as a function of orientation, can be represented by a probability measure~$\mu_x$ on the unit sphere~$\S^2$.
The measure~$\mu_x$ satisfies to the condition $\mu_x(B) = \mu_x(-B)$ for all $B\in \mathscr B(\S^2)$, which accounts for the head-to-tail symmetry of the molecules.
Then, the simplest meaningful way to condense the information conveyed by~$\mu_x$ is to consider the second-order moment
\[
 Q = \int_{\S^2} \left( \n^{\otimes 2} - \frac13 \Id \right) \d \mu_x(\n) .
\]
We denote by~$\n^{\otimes 2}$ the matrix defined by~$(\n^{\otimes 2})_{i,j} := \n_i\n_j$, for each~$i, \, j\in\{1, \, 2, \, 3\}$.
The quantity~$Q$ is renormalized, so that the isotropic state~$\mu_x = \H^2\mres\S^2$ corresponds to~$Q = 0$.
As a result, $Q$~is a symmetric traceless matrix. (The interested reader is referred e.g. to~\cite{MottramNewton} for further details.)

The (simplified) Landau-de Gennes functional reads
\begin{equation} \label{energy} \tag{LG$_\varepsilon$}
 E_\varepsilon(Q) := \int_\Omega \left\{\frac12 \abs{\nabla Q}^2 + \frac{1}{\varepsilon^2} f(Q) \right\} ,
\end{equation}
where~$Q\colon \Omega \to \Sz$ is the configuration of the medium, located in a bounded container~$\Omega\subseteq\R^3$.
The function~$f$ is the quartic Landau-de Gennes potential, defined by
\begin{equation} \label{f}
  f(Q) = k - \frac a2 \tr Q^2 - \frac b3 \tr Q^3 + \frac c4 \left(\tr Q^2\right)^2 \qquad \textrm{for } Q\in\Sz .
\end{equation}
This expression for~$f$ has been derived by a formal expansion in powers of~$Q$.
All the terms are invariant by rotations so that~$f$ is independent of the coordinate frame.
This potential allows for multiple local minima, with a first-order isotropic-nematic phase transition (see~\cite{deGennes, Virga}).
The positive parameters~$a$,~$b$ and~$c$ depend on the material and the temperature~(which is assumed to be uniform),
whereas $k$ is just an additive constant, which plays no role in the minimization problem.
The potential~$f$ is bounded from below, so we determine uniquely the value of~$k$ by requiring~$\inf f = 0$.
The parameter~$\varepsilon^2$ is a material-dependent elastic constant, typically very small. 
For each $0 < \varepsilon < 1$, we assign a boundary datum $g_\varepsilon \in H^1(\partial\Omega, \, \Sz)$ and we restrict our attention to minimizers~$Q_\varepsilon$ of~\eqref{energy} in the class
\[
 H^1_{g_\varepsilon}(\Omega, \, \Sz) := \left\{ Q\in H^1(\Omega, \, \Sz)\colon Q = g_\varepsilon \textrm{ on } \partial\Omega \textrm{ in the sense of traces}\right\} .
\]

When $\varepsilon$~is small, the term~$\varepsilon^{-2}f(Q)$ in~\eqref{energy} forces minimizers to take their values close to~$\NN := f^{-1}(0)$.
This set can be characterized as follows (see~\cite[Proposition~9]{MajumdarZarnescu}):
 \begin{equation} \label{N}
  \NN = \left\{ s_* \left(\n^{\otimes 2} - \frac13\Id \right) \colon \n\in \S^2 \right\} ,
 \end{equation}
where the constant~$s_*$ is defined by
\begin{equation*} 
  s_* = s_*(a, \, b, \, c) :=\frac{1}{4c}\left(b + \sqrt{b^2 + 24ac}\right).
\end{equation*}
Thus,~$\NN$ is a smooth submanifold of~$\Sz$, diffeomorphic to the real projective plane~$\PR$, called \emph{vacuum manifold}.
The topology of $\NN$ plays an important role, for a map~$\Omega\to\NN$ may encounter topological obstructions to regularity.
Sources of obstruction are the homotopy groups~$\pi_1(\NN)\simeq \mathbb Z/2\mathbb Z$ and~$\pi_2(\NN) \simeq\mathbb Z$, which are associated with line and point singularities, respectively.
There is a remarkable difference with the Oseen-Frank model at this level, for $\S^2$ is a simply connected manifold, so topological obstructions result from~$\pi_2(\S^2)$ only.
Despite this fact, a strong connection between the Oseen-Frank and Landau-de Gennes theories was established by Majumdar and Zarnescu.
In their paper~\cite{MajumdarZarnescu}, they addressed the asymptotic analysis of minimizers of~\eqref{energy}, in three-dimensional domains.
Their results imply that, when~$\Omega$,~$\partial\Omega$ are simply connected and~$g_\varepsilon = g\in C^1(\partial \Omega, \, \NN)$,
minimizers~$\Qe$ of~\eqref{energy} converge in~$H^1(\Omega, \, \Sz)$ to a map of the form
\[
 Q_0(x) = s_* \left(\n_0^{\otimes 2}(x) - \frac13 \Id \right)
\]
where $\n_0\in H^1(\Omega, \, \S^2)$ is a minimizer of~\eqref{dirichlet_energy}. The convergence is locally uniform, away from singularities of~${Q}_0$.
Also in this case,~line defects do not appear in the limiting map, although point defects analogous to~\eqref{hedgehog} might occur.
Indeed, their assumptions on the domain and boundary datum are strong enough to guarantee the uniform energy bound
\begin{equation} \label{energy bd}
 E_\varepsilon(\Qe) \leq C
\end{equation}
for an $\varepsilon$-independent constant $C$, and obtain $H^1$-compactness.
In this paper, we work in the logarithmic energy regime
\begin{equation} \label{logenergy}
 E_\varepsilon(\Qe) \leq C \left(\abs{\log\varepsilon} + 1 \right),
\end{equation}
which is compatible with singularities of codimension two, in the small~$\varepsilon$ limit.

There are analogies between the functional~\eqref{energy} and the Ginzburg-Landau energy for superconductivity, which reduces to
\begin{equation} \label{GL}
 E^{\mathrm{GL}}_\varepsilon(u) := \int_\Omega \left\{\frac12 \abs{\nabla u}^2 + \frac{1}{4\varepsilon^2} \left( 1 - |u|^2\right)^2 \right\}
\end{equation}
when no external field is applied. Here the unknown is a complex-valued function $u$.
There is a rich literature about the asymptotic behaviour, as $\varepsilon\to 0$, of critical points satisfying a logarithmic energy bound such as~\eqref{logenergy}.
It is well-known that, under appropriate assumptions, critical points converge to maps with topology-driven singularities of codimension two.
In two-dimensional domains, the theory has been developed after Bethuel, Brezis and H\'elein's work \cite{BBH}.
In the three-dimensional case, the asymptotic analysis of minimizers was performed by Lin and Rivi\`ere~\cite{LinRiviere},
and extended to non-minimizing critical points by Bethuel, Brezis and Orlandi~\cite{BethuelBrezisOrlandi}.
Later, Jerrard and Soner~\cite{JerrardSoner-GL} and Alberti, Baldo, Orlandi~\cite{AlbertiBaldoOrlandi} proved independently that $|\log\varepsilon|^{-1}E_\varepsilon^{\mathrm{GL}}$ $\Gamma$-converges, when $\varepsilon\to 0$, to a functional on integral currents of codimension two.
This functional essentially measures the length of defect lines, weighted by some quantity that accounts for the topology of the defect.

\subsection{Main results}

For each fixed $\varepsilon > 0$, a classical argument of Calculus of Variations shows that minimizers of~\eqref{energy}
exist as soon as $g_\varepsilon\in H^{1/2}(\partial\Omega, \, \Sz)$ and are regular in the interior of the domain.
Our main result deals with their asymptotic behaviour as~$\varepsilon\to 0$.

\begin{theor} \label{th:convergence}
 Let $\Omega$ be a bounded, Lipschitz domain.
 Assume that there exists a positive constant~$M$ such that, for any $0 < \varepsilon < 1$, there hold
 \begin{equation} \label{hp:H} \tag{H}
  E_\varepsilon(\Qe) \leq M \left(\abs{\log\varepsilon} + 1 \right) \qquad \textrm{and} \qquad \norm{\Qe}_{L^\infty(\Omega)} \leq M.
 \end{equation}
 Then, there exist a subsequence $\varepsilon_n\searrow 0$, a closed set $\Sl\subseteq\overline\Omega$ and a map~${Q}_0\in H^1_{\mathrm{loc}}(\Omega\setminus\Sl, \, \NN)$ such that the following holds.
 \begin{enumerate}[label = (\roman*), ref = (\roman*)]
  \item \label{th:first} $\Sl\cap\Omega$ is a countably $\H^1$-rectifiable set, and $\H^1(\Sl\cap\Omega) < +\infty$.
  \item ${Q}_{\varepsilon_n} \to {Q}_0$ strongly in~$H^1_{\mathrm{loc}}(\Omega\setminus \Sl, \, \NN)$.
  \item ${Q}_0$ is locally minimizing harmonic in $\Omega\setminus\Sl$, that is for every ball $B \csubset \Omega\setminus\Sl$ and any $P\in H^1(B, \, \NN)$
  which satisfies $P = Q_0$ on~$\partial B$ we have
  \[
   \frac12 \int_{B} \abs{\nabla {Q}_0}^2 \leq \frac12 \int_{B} \abs{\nabla P}^2 .
  \]
  \item \label{th:last} There exists a locally finite set $\Spt\subseteq\Omega\setminus\Sl$ such that 
  ${Q}_0$ is smooth on $\Omega\setminus (\Sl \cup \Spt)$ and $\Qe \to {Q}_0$ locally uniformly in~$\Omega\setminus (\Sl \cup \Spt)$.
 \end{enumerate}
\end{theor}

By saying that~$\Sl$ is countably $\H^1$-rectifiable we mean that there exists a decomposition
\[
 \Sl = \bigcup_{j\in\N} \SS_j ,
\]
where~$\H^1(\SS_0) = 0$ and, for each~$j\geq 1$, the set~$\SS_j$ is the image of a Lipschitz function $\R\to \R^3$.
In addition to the singular set~$\Sl$ of dimension one, the limiting map~${Q}_0$ may have a set of point singularities~$\Spt$.
This is consistent with the regularity results for minimizing harmonic maps~\cite{GiaquintaGiusti, SchoenUhlenbeck}.
Later on, we will discuss examples where~$\Sl$ and~$\Spt$ are non-empty.

Theorem~\ref{th:convergence} is local in nature.
In particular, boundary conditions play no particular role in the proof of this result, although they need to be imposed to induce non-trivial behaviour of minimizers.
Theorem~\ref{th:convergence} can be adapted to the analysis near the boundary of the domain, under additional assumptions on the boundary datum.
The necessary modifications are sketched in Section~\ref{subsect:boundary}.

The singular set~$\Sl$ is defined as the concentration set for the energy densities of minimizers. 
In other words, thanks to~\eqref{hp:H} we find a subsequence~$\varepsilon_n\searrow 0$ and a measure
$\mu_0\in\mathscr{M}_{\mathrm{b}}(\overline\Omega) := C(\overline\Omega)^\prime$ such that
\[
 \left\{\frac12 \abs{\nabla Q_{\varepsilon_n}}^2 + \frac{1}{\varepsilon^2_n} f(Q_{\varepsilon_n})\right\}\frac{\d x}{|\log\varepsilon_n|} \rightharpoonup^* \mu_0 
 \qquad \textrm{in } \mathscr{M}_{\mathrm{b}}(\overline\Omega) ,
\]
then we define~$\Sl :=\mathrm{supp\,}\mu_0$.
A more precise description of the limit measure~$\mu_0$ is given by the following result. We set 
\[
 \kappa_* := \frac{\pi}{2} s_*^2.
\]
As we will see in Section~\ref{subsect:JerrardSandier}, this number quantifies the energy cost associated with a topological defect of codimension two.

\begin{prop} \label{prop:intro-S}
 The measure~$\mu_0\mres\Omega$ is naturally associated with a stationary varifold, and
 \[
  \lim_{r\to 0} \frac{\mu_0(\overline{B}_r(x))}{2r} = \kappa_* \qquad \textrm{for } \mu_0\textrm{-a.e. } x\in\Omega .
 \]
 For any open set~$K\csubset\Omega$, $\Sl\cap\overline{K}$ is the union of a finite number of closed straight line segments~$L_1, \, \ldots, \, L_p$.
 After possible subdivision, assume that for each~$i\neq j$, either $L_i$ and~$L_j$ are disjoint or they intersect at a common endpoint.
 Then, the following properties hold.
 \begin{enumerate}[label = (\roman*), ref = (\roman*)]
  \item \label{item:S-nontrivial} If~$D\csubset K$ is a closed disk which intersects~$\Sl$ at a single point~$x_0$ and~$x_0$ is not an endpoint for any~$L_i$, 
  then the homotopy class of~$Q_0$ restricted to~$\partial D$ is non-trivial.
  
  \item \label{item:S-even} Suppose that~$x_0\in K$ is an endpoint of exactly~$q$ segments~$L_{i_1}, \, \ldots, \, L_{i_q}$. Then~$q$ is even.
 \end{enumerate}
\end{prop}

The definition of stationary varifold is given in~\cite[Chapter~4]{Simon-GMT}.
Varifolds are a generalization of differentiable manifolds, introduced by Almgren~\cite{Almgren66} in the context of Calulus of Variations, and
stationary varifolds can be though as a weak notion of minimal manifolds.
Proposition~\ref{prop:intro-S} relies heavily on the structure theorem for stationary varifolds of dimension one~\cite[Theorem p.~89]{AllardAlmgren}.

Inside the domain, the singular set is a locally finite union of line segments.
Branching points are not excluded by this result, but only an even number of branches can originate from each point.
(However, we expect that branching points should not arise --- see the concluding remarks, Section~\ref{subsect:conclusion}.)
Moreover, the set~$\Sl$ is a topological singularity, i.e. it is associated with a non-trivial homotopy class of the map~$Q_0$.
Therefore, from the physical point of view $\Sl$ corresponds to the ``thin disclination lines'' of index~$\pm 1/2$ (see, e.g.~\cite{ChandrasekharRanganath}).
The ``thick disclination lines'' of index~$\pm 1$ are not included in~$\Sl$ because the order parameter~$Q_0$ can be defined continuously throughout their cores,
thanks to the ``escape in the third dimension'' proposed by Cladis and Kl\'eman~\cite{CladisKleman}.
Proposition~\ref{prop:intro-S} also excludes disclination loops in the interior of the domain, although
loops of radius larger than some critical value~$R_c$ are expected to occur (see~\cite[p.~519]{ChandrasekharRanganath}).
However, in the limit as~$\varepsilon\to 0$ we have that~$R_c\to+\infty$, therefore any defect loop which is contractible in~$\Omega$ should become unstable, shrink and eventually disappear.
On the other hand, disclination loops may occur at the boundary of the domain.
In Section~\ref{subsect:torus}, we show by an example that the singular set~$\Sl$ may touch~$\partial\Omega$, even if the boundary datum is smooth.
In this case, the conclusion of Proposition~\ref{prop:intro-S} does not hold any more. 

\medskip
We provide sufficient conditions for the estimate~\eqref{hp:H} to hold, in terms of the domain and the boundary data.
Here is our first condition.
\begin{enumerate}[label=\textup{}{(H\textsubscript{\arabic*})}, ref={H\textsubscript{\arabic*}}]
 \item \label{hp:H0.5} $\Omega$ is a bounded, smooth domain and~$\{g_\varepsilon\}_{0 < \varepsilon < 1}$ is a bounded family in~$H^{1/2}(\partial\Omega, \, \NN)$.
\end{enumerate}
The uniform~$H^{1/2}$-bound is satisfied if, for instance, $g_\varepsilon = g\colon\partial\Omega\to\NN$ has a finite number of disclinations.
This means, there exists a finite set~$\Sigma\subseteq\partial\Omega$ such that~$g$ is smooth on~$\partial\Omega\setminus\Sigma$
and, for each~$x_0\in\Sigma$, we can write
\begin{equation} \label{discl_bord}
 g(x) = s_* \left\{ \bigg(\mathbf{\tau}_1\cos\left(k\theta(x)\right) + \mathbf{\tau}_2\sin\left(k\theta(x)\right) \bigg)^{\otimes 2} - \frac13 \Id \right\} +
 \textrm{smooth terms of order }\rho(x) 
\end{equation}
as~$x\to x_0$. Here~$k\in\frac12\mathbb{Z}$,~$(\rho(x), \, \theta(x))$ are geodesic polar coordinates centered at~$x_0$ 
and~$(\mathbf{\tau}_1, \, \mathbf{\tau}_2)$ is an orthonormal pair in~$\R^3$.
\begin{prop} \label{prop:intro-H1/2}
 Condition~\eqref{hp:H0.5} implies~\eqref{hp:H}.
\end{prop}

Alternatively, one can assume
\begin{enumerate}[label=\textup{}{(H\textsubscript{\arabic*})}, ref={H\textsubscript{\arabic*}}, resume]
 \item \label{hp:domain}
 $\Omega\subseteq\R^3$ is a bounded Lipschitz~domain, and it is bilipschitz equivalent to a handlebody (i.e. a $3$-ball with a finite number of handles attached). 
 \item \label{hp:bd data} There exists $M_0 > 0$ such that, for any $0 < \varepsilon < 1$, we have $g_\varepsilon\in (H^1\cap L^\infty)(\partial\Omega, \, \Sz)$ and
 \begin{equation*} 
   E_\varepsilon(g_\varepsilon, \, \partial\Omega) \leq M_0 \left(\abs{\log\varepsilon} + 1\right) , \qquad \norm{g_\varepsilon}_{L^\infty(\partial\Omega)} \leq M_0 .
 \end{equation*}
 \end{enumerate}
As an example of sequence satisfying~\eqref{hp:bd data}, one can take smooth approximations of a map~$g\colon\partial\Omega\to\NN$ of the form~\eqref{discl_bord}.
For instance, we can take
\begin{equation} \label{discl_bord_reg}
 g_\varepsilon(x) := \eta_\varepsilon(\rho(x)) g(x)
\end{equation}
where~$\eta_\varepsilon\in C^\infty[0, \, +\infty)$ is such that
\[
 \eta_\varepsilon(0) = \eta_\varepsilon^\prime(0) = 0, \qquad \eta_\varepsilon(\rho) = 1 \textrm{ if } \rho \geq \varepsilon, 
 \qquad 0 \leq \eta_\varepsilon \leq 1, \qquad \abs{\eta_\varepsilon^\prime} \leq C \varepsilon^{-1} .
\]
\begin{prop} \label{prop:intro-H}
 If~\eqref{hp:domain} and~\eqref{hp:bd data} are satisfied, then~\eqref{hp:H} holds.
\end{prop}

\begin{remark} 
 Hypothesis \eqref{hp:domain} is \emph{not} the same as asking $\Omega$ to be a bounded Lipschitz domain with connected boundary. 
 Let $K\subseteq\S^3$ be a (open) tubular neighborhood of a trefoil knot. 
 Then $K$ is a solid torus, i.e. $K$ is diffeomorphic to $\S^1\times B^2_1$, but $\S^3\setminus K$ is \emph{not} a solid torus.
 In fact, $\S^3\setminus K$ is not even a handlebody, because
 \[
  \pi_1(\S^2 \setminus K) = \textrm{ the knot group of the trefoil knot } = \left\langle x, \, y \, | \, x^2 = y^3 \right\rangle 
 \]
 whereas the fundamental group of any handlebody is free.
 By composing with a stereographic projection, one constructs a smooth domain $\overline\Omega\subseteq\R^3$ diffeomorphic to $\S^3\setminus K$.
 In particular, $\partial\Omega$ is a torus but $\Omega$ does not satisfies~\eqref{hp:domain}.
\end{remark}

Given an arbitrary domain, one can construct examples where line defects occur.
\begin{prop} \label{prop:intro-log}
 For each bounded domain~$\Omega\subseteq\R^3$ of class~$C^1$, there exists a family of boundary data~$\{g_\varepsilon\}_{0<\varepsilon <1}$
 satisfying~\eqref{hp:bd data} and a number~$\alpha >0$ such that
 \[
  E_\varepsilon(Q) \geq \alpha \left(\abs{\log\varepsilon} - 1 \right)
 \]
 for any~$Q\in H^1_{g_\varepsilon}(\Omega, \, \Sz)$ and any~$0 < \varepsilon < 1$.
 Moreover,~$\Sl$ is non-empty.
\end{prop}
The functions~$g_\varepsilon$ are smooth approximations of a map~$\partial\Omega\to\NN$, which has point singularities of the form~\eqref{discl_bord_reg}.

\begin{figure}[t] 
 \centering
 \includegraphics[height =.26\textheight, keepaspectratio = true]{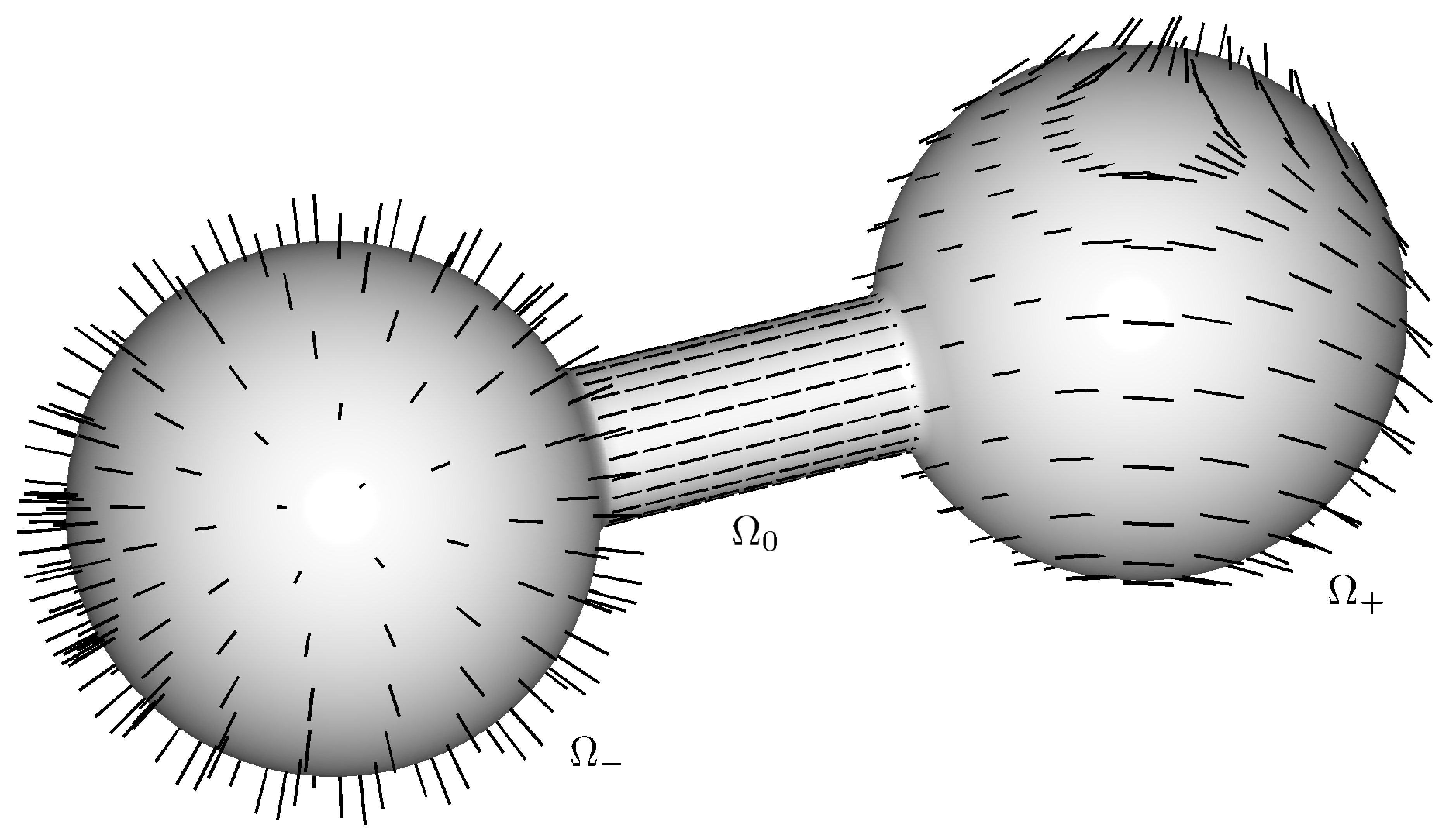}
 \caption{The domain considered in Section~\ref{sect:SX}: two unit balls ($\Omega_-$ on the left, $\Omega_+$ on the right) joined by a cylinder~$\Omega_0$ of length~$2L$ and radius~$r$.
 The (unoriented) director field associated to the boundary datum is also represented.
 The boundary datum restricted to the boundary of~$\Omega_-$, $\Omega_+$ defines non-trivial homotopy classes in~$\pi_2(\NN)$, $\pi_1(\NN)$ respectively.}
 \label{fig:ballcylinder}
\end{figure}

Finally, we consider an example where both~$\Sl$ and~$\Spt$ are non-empty.
The domain consists of two balls of radius~$1$, joined by a cylinder of radius~$r\in (0, \, 1/2)$ and length~$2 L$.
The boundary datum, which is defined in Section~\ref{sect:SX}, is uniaxial and has two point defects.
In Figure~\ref{fig:ballcylinder}, we represent the behaviour of the boundary datum or, more precisely, 
the direction of the eigenspace corresponding to the leading eigenvalue (that is, the average orientation of the molecules at each point).
This map defines non-trivial homotopy classes both in~$\pi_1(\NN)$ and in~$\pi_2(\NN)$.

\begin{prop} \label{prop:intro-SX}
 There exists a positive number~$L^*$ such that, if~$L \geq L^*$, then~$\Sl$ is non-empty and there exists a point~$x_0\in\Spt$ such that~$\dist(x_0, \, \Sl) \geq L/2$.
\end{prop}

In other words, if the cylinder is long enough then the limit configuration has line defects and at least one point defect, which is far from the line defects.
Although the boundary datum defines a non-trivial class in~$\pi_2(\NN)$, topological arguments alone are not enough to conclude that~$\Spt\neq\emptyset$,
for there exist maps~$\Omega\to\NN$ which satisfy the boundary condition and~$\Sl\neq\emptyset$,~$\Spt=\emptyset$ (see Remark~\ref{remark:X empty}).
Proposition~\ref{prop:intro-SX} is inspired by Hardt and Lin's paper~\cite{HardtLin}, 
where the existence of minimizing harmonic maps with non-topologically induced singularities is proved.
However, there is an additional difficulty here, namely minimizers are not uniformly bounded in~$H^1$ as~$\varepsilon\to 0$.
We take care of this issue by adapting some ideas of the proof of Theorem~\ref{th:convergence}.

\medskip
Let us spend a few word on the proof of our main result, Theorem~\ref{th:convergence}.
The core of the argument is a concentration property for the energy, which can be stated as follows.

\begin{prop} \label{prop:desiderata}
 Assume that the condition~\eqref{hp:H} holds.
 For any~$0 < \theta < 1$ there exist positive numbers~$\eta$,~$\epsilon_0$ and~$C$ such that,
 for any~$x_0\in\Omega$, $R >0$ satisfying $\overline B_R(x_0) \subseteq\Omega$ and any~$0 < \varepsilon \leq \epsilon_0 R$, if 
 \begin{equation} \label{desid: hp}
  E_\varepsilon(\Qe, \, B_R(x_0)) \leq \eta R \log\frac{R}{\varepsilon}
 \end{equation}
 then
 \begin{equation*} \label{desid: eta-log energy}
   E_\varepsilon(\Qe, \, B_{\theta R}(x_0)) \leq C R .
 \end{equation*}
\end{prop}

Proposition~\ref{prop:desiderata} implies that either the energy on a ball blows up at least logarithmically, or it is bounded on a smaller ball.
Combining this fact with covering arguments, one proves that the energy concentrates on a set $\Sl$ of finite length.
Then, the asymptotic behaviour of minimizers away from $\Sl$ can be studied using well-established techniques, e.g. arguing as in~\cite{MajumdarZarnescu}.

\medskip
Roughly speaking, the proof of Proposition~\ref{prop:desiderata} goes as follows.
Condition~\eqref{desid: hp} implies that the problem can be reformulated in terms of~$\NN$-valued maps.
Indeed, by an average argument we find~$r\in (\theta R, \, R)$ such that the energy of~$Q_\varepsilon$ on the sphere of radius~$r$ is controlled by~$C\eta|\log\varepsilon|$.
On the other hand, the energy per unit length associated with a topological defect line is of the order of~$\kappa_* |\log\varepsilon|$ 
(see the estimates by Jerrard and Sandier~\cite{Jerrard, Sandier}).
Therefore, if~$\eta$ is small compared to~$\kappa_*$, the sphere of radius~$r$ intersects no topological defect line of~$Q_\varepsilon$.
This makes it possible to approximate~$Q_\varepsilon$ with an~$\NN$-valued map~$P_\varepsilon$, defined on a sphere of radius~$r^\prime$ close to~$r$.
Now, since the sphere is simply connected, $P_\varepsilon$ can be lifted to a~$\S^2$-valued map, i.e. one can write
\[
 P_\varepsilon(x) = s_* \left( \n_\varepsilon^{\otimes 2}(x) - \frac13 \Id \right) \qquad \textrm{for } x\in \partial B_{r^\prime}(x_0) 
\]
for a smooth vector field~$\n_\varepsilon\colon\partial B_{r^\prime}(x_0)\to \S^2$.
Thus, we have reduced to a problem which is formulated in terms of vector fields, and we can apply the methods by Hardt, Kinderlehrer and Lin~\cite[Lemma~2.3]{HKL} to obtain boundedness of the energy.
In other words, Condition~\eqref{desid: hp} enable us to reduce the asymptotic analysis of the Landau-de Gennes problem to the analysis of the Oseen-Frank problem. 
Extension results are needed in several steps of this proof, 
for instance to interpolate between~$Q_\varepsilon$ and~$P_\varepsilon$ in order to construct an admissible comparison map.
Various results in this direction are discussed in detail in Section~\ref{sect:extension}.
In particular, we prove variants of Luckhaus' lemma~\cite[Lemma~1]{Luckhaus-PartialReg} which are fit for our purposes.

\begin{remark}
Condition~\eqref{hp:H} is \emph{not} sufficient to obtain compactness for minimizers of the Ginzburg-Landau energy.
Indeed, Brezis and Mironescu~\cite{BrezisMironescu} constructed a sequence\footnote{Throughout the paper, 
the word ``sequence'' will be used to denote family of functions indexed by a continuous parameter as well.} of minimizers~$u_\varepsilon\in H^1(B^2_1, \, \C)$
such that
\[
 E^{\mathrm{GL}}_\varepsilon(u_\varepsilon, \, B_1^2) \ll \abs{\log\varepsilon} \qquad \textrm{and} \qquad \abs{u_\varepsilon} \leq 1,
\]
yet~$\{u_\varepsilon\}_{0 < \varepsilon < 1}$ does not have subsequences converging a.e.~on sets of positive measure.
Brezis and Mironescu's example relies on \emph{oscillations of the phase}. 
Indeed, the~$u_\varepsilon$'s can be lifted to $\R$-valued functions~$\varphi_\varepsilon$'s 
(that is~$u_\varepsilon \approx \exp(i\varphi_\varepsilon)$), but the latter are not uniformly bounded in~$L^\infty$.
This phenomenon does not occur in our case, because the $\NN$-valued map~$P_\varepsilon$ is lifted to a unit vector field.
The finiteness of the fundamental group $\pi_1(\NN)$ yields the compactness of the universal covering of~$\NN$, hence better compactness properties for the minimizers.
\end{remark}

\subsection{Concluding remarks and open questions}
\label{subsect:conclusion}

Several questions about minimizers of the Landau-de Gennes functional on three-dimensional domains remain open.
A first question concerns the behaviour of the singular set~$\Sl$.
Since~$\Sl$ is obtained as a limit of minimizers, one would expect that it inherits from~$\Qe$ minimizing properties.
It is natural to conjecture that~$\Sl$ is a relative cycle, whose homology class is determined by the domain and the boundary datum, 
and that~$\Sl$ has minimal length in its homology class.
For instance, if the domain is convex and the boundary data has a finite number of point singularities~$x_1, \, \ldots, \, x_p$ of the form~\eqref{discl_bord}, 
then~$\Sl$ should be a union of non-intersecting straight lines connecting the~$x_i$'s in pairs, in such a way that the total length of~$\Sl$ is minimal.
(Notice that, by topological arguments, the number~$p$ must be even.)
However, if the domain is non-convex, then a part of~$\Sl$ may lie on the boundary.

It would be interesting to study the structure of minimizers~$\Qe$ in the core of line defects.
For instance, \emph{does the core of line defects contain biaxial phases?}
Contreras and Lamy~\cite{ContrerasLamy} and Majumdar, Pisante and Henao~\cite{MajumdarPisanteHenao} proved that the core of point singularities, 
in dimension three, contains biaxial phases when the temperature is low enough.
Their proofs use a uniform energy bound such as~\eqref{energy bd}, so they do not apply directly to singularities of codimension two.
However, the analysis of point defects on two-dimensional domains (see e.g.~\cite{pirla, DiFratta, INSZ-Instability}) 
suggests that line defects may also contain biaxial phases, when the temperature is low.
A related issue is the analysis of \emph{singularity profiles}.
Let~$x_0\in\Sl$ and let~$\Pi$ be an orthogonal plane to~$\Sl$, passing through the point~$x_0$. Set
\[
  P_{\varepsilon, x_0}(y) := Q_\varepsilon(x_0 + \varepsilon y) \qquad \textrm{for } y\in\Pi .
\]
This defines a bounded sequence in~$L^\infty(\Pi, \, \Sz)$, such that
\[
  \norm{\nabla P_{\varepsilon, x_0}}_{L^2(K)} = \norm{\nabla Q_\varepsilon}_{L^2(x_0 + \varepsilon K)} \leq C(K) \qquad \textrm{for every } K\csubset\Pi .
\]
Therefore, up to a subsequence we have~$P_{\varepsilon, x_0}\rightharpoonup P_{x_0}$ in~$H^1_{\mathrm{loc}}(\Pi, \, \Sz)$.
The map~$P_{x_0}$ contains the information on the fine structure of the defect core. What can be said about~$P_{x_0}$?

In another direction, investigating the asymptotic behaviour of a more general 
class of functionals in the logarithmic energy regime is a challenging issue.
For instance, one may consider functionals with more elastic energy terms and/or choose different potentials, such as the sextic potential
\[
  f(Q) := -\frac{a_1}{2} \tr Q^2 - \frac{a_2}{3} \tr Q^3 + \frac{a_3}{4} \left(\tr Q^2\right)^2
  + \frac{a_4}{5} \left(\tr Q^2\right) \left(\tr Q^3\right) + \frac{a_5}{6} \left(\tr Q^2\right)^3 + \frac{a_5^\prime}{6} \left(\tr Q^3\right)^2
\]
(see~\cite{DavisGartland, Gramsbergen}) or the singular potential proposed by Ball and Majumdar~\cite{BallMajumdar}.
From this point of view, it is interesting to remark that the proof of Proposition~\ref{prop:desiderata} is quite robust, 
as it is based on variational arguments alone and does not use the structure of the Euler-Lagrange system.
Therefore, the proof of Proposition~\ref{prop:desiderata} could be adapted to other choices of the elastic energy density,
provided that they are quadratic in the gradient, and other potentials~$f$, possibly infinite-valued outside some convex set,
provided that they satisfy some non-degeneracy conditions around their set of minimizers (see Lemma~\ref{lemma:f below}).
This would yield local $H^1$-compactness, away from a singular set of dimension one, for minimizers of more general Landau-de Gennes functionals.
However, proving stronger compactness for minimizers (e.g., with respect to the uniform norm)
and understanding the structure of the singular set~$\Sl$ for the general Landau-de Gennes energy are completely open questions.


\medskip
The paper is organized as follows.
Section~\ref{sect:preliminary} deals with general facts about the space of $Q$-tensors and Landau-de Gennes minimizers.
In particular, lower estimates for the energy of $Q$-tensor valued maps are established in Subsection~\ref{subsect:JerrardSandier}, by adapting Jerrard's and Sandier's arguments.
Section~\ref{sect:extension} contains several extension results, which are a fundamental tool for the proof of our main theorem.
Section~\ref{sect:convergence} aims at proving Theorem~\ref{th:convergence}, and in particular it contains the proof of Proposition~\ref{prop:desiderata} (Section~\ref{subsect:desiderata}).
The asymptotic analysis away from the singular lines is carried out in Section~\ref{subsect:H1 bounds}, whereas the singular set~$\Sl$ is defined and studied in Section~\ref{subsect:S}.
In Section~\ref{subsect:boundary}, we work out the analysis of minimizers near the boundary of the domain.
Section~\ref{sect:S} deals with the proof of Proposition~\ref{prop:intro-S}.
We first show the stationarity of~$\mu_0$ (Section~\ref{subsect:stationary}); then, with the help of an auxiliary problem (Section~\ref{subsect:cylinder}),
we compute the density of~$\mu_0$ and conclude the proof (Section~\ref{subsect:constant_density}).
In Section~\ref{subsect:torus}, we construct an example where~$\mu_0$ concentrates at the boundary of the domain.
Section~\ref{sect:prop_H} deals with the proofs of Propositions~\ref{prop:intro-H1/2} (Section~\ref{subsect:datoH1/2}), 
\ref{prop:intro-H} and~\ref{prop:intro-log} (Section~\ref{subsect:datoH1}).
Finally, in Section~\ref{sect:SX} we prove Proposition~\ref{prop:intro-SX} by constructing an example where the limit configuration~$Q_0$ has both lines and point singularities.

\numberwithin{equation}{section}
\numberwithin{defn}{section}
\numberwithin{theor}{section}
\numberwithin{remark}{section}

\section{Preliminary results}
\label{sect:preliminary}

Throughout the paper, we use the following notation.
We denote by~$B^k_r(x)$ (or, occasionally, $B^k(x, \, r)$) the $k$-dimensional open ball of radius~$r$ and center~$x$, and by~$\overline{B}^k_r(x)$ the corresponding closed ball.
When $k = 3$, we omit the superscript and write~$B_r(x)$ instead of~$B_r^3(x)$.
When $x = 0$, we write~$B_r^k$ or~$B_r$. Balls in the matrix space~$\Sz$ will be denoted~$B^{\Sz}_r(Q)$ or~$B^{\Sz}_r$.
For any~$Q\in H^1(\Omega, \, \Sz)$ and any~$k$-submanifold $U\subseteq\Omega$, we set
\[
 e_\varepsilon(Q) := \frac12 \abs{\nabla Q}^2 + \frac{1}{\varepsilon^2} f(Q) , \qquad 
 E_\varepsilon(Q, \, U) := \int_U e_\varepsilon(Q) \, \d \H^k .
\]
The function~$e_\varepsilon(Q)$ will be called the energy density of~$Q$.
We also set~$E_\varepsilon(Q, \, \emptyset) := 0$ for any map~$Q$.
Additional notation will be set later on.

\subsection{Properties of~$\Sz$ and~$f$}
\label{subsect:Sz}

We discuss general facts about $Q$-tensors, which are useful in order to to have an insight into the structure of the target space~$\Sz$.
The starting point of our analysis is the following representation formula.

\begin{lemma} \label{lemma:representation}
 For all fixed $Q\in\Sz\setminus \{0\}$, there exist two numbers $s \in (0, \, +\infty)$, $r\in [0, \, 1]$ and an orthonormal pair of vectors $(\n, \, \m)$ in $\R^3$ such that
 \[
  Q = s \left\{\n^{\otimes 2} - \frac 13 \Id + r \left( \m^{\otimes 2} - \frac13 \Id \right) \right\} .
 \]
 Given $Q$, the parameters $s = s(Q)$, $r = r(Q)$ are uniquely determined.
 The functions $Q \mapsto s(Q)$ and $Q \mapsto r(Q)$ are continuous on $\Sz\setminus\{0\}$, and are positively homogeneous of degree~$1$ and~$0$, respectively.
\end{lemma}

Slightly different forms of this formula are often found in the literature (e.g.~\cite[Proposition~1]{MajumdarZarnescu}).
The proof is a straightforward computation sketched in~\cite[Lemma 3.2]{pirla}, so we omit it here.

\begin{remark} \label{remark: s,r}
 The parameters~$s(Q)$,~$r(Q)$ are determined by the eigenvalues $\lambda_1(Q) \geq \lambda_2(Q) \geq \lambda_3(Q)$ according to this formula:
 \begin{equation*} 
  s(Q) = 2\lambda_1(Q) + \lambda_2(Q), \qquad r(Q) = \frac{\lambda_1(Q) + 2\lambda_2(Q)}{2\lambda_1(Q) + \lambda_2(Q)} .
 \end{equation*}
\end{remark}

Following~\cite[Proposition~15]{MajumdarZarnescu}, the vacuum manifold~$\NN:= f^{-1}(0)$ can be characterized as follows:
\begin{equation} \label{characterizationN}
  Q\in\NN \qquad \textrm{if and only if } s(Q) = s_* \textrm{ and } r(Q) = 0,
\end{equation}
where 
\begin{equation*} 
  s_* :=\frac{1}{4c}\left(b + \sqrt{b^2 + 24ac}\right).
\end{equation*}
There is another set which is important for our analysis, namely
\[
 \Cs := \bigg\{Q \in \Sz\colon \lambda_1(Q) = \lambda_2(Q) \bigg\} ,
\]
i.e.~$\Cs$ is the set of matrices whose leading eigenvalue has multiplicity $> 1$. 
This is a closed subset of~$\Cs$, and it is cone (i.e., $\mu Q\in \Cs$ for any $Q\in\Cs$, $\mu\in\R^+$). 
By Remark~\ref{remark: s,r}, we see that
\[
 Q\in\Cs \qquad \textrm{if and only if } s(Q) = 0 \textrm{ or } r(Q) = 1.
\]
Then, applying Lemma~\ref{lemma:representation} and the identity $\Id = \n^{\otimes 2} + \m^{\otimes 2} + \p^{\otimes 2}$ (where~$(\n, \, \m, \, \p)$ is any orthonormal, positive basis of~$\R^3$),
we see that~$Q\in\Cs$ if and only if there exist~$s\geq 0$ and~$\p\in\S^2$ such that
\begin{equation} \label{oblate}
 Q = - s \left( \p^{\otimes 2} - \frac 13 \Id \right).
\end{equation}
Therefore, $\Cs$ is the cone over the projective plane~$\PR$.
The importance of~$\Cs$ is explained by the following fact.
Away from~$\Cs$, it is possible to define locally a continuous map~$Q \mapsto\n(Q)$, which selects a unit eigenvector~$\n(Q)$ associated with~$\lambda_1(Q)$
(see e.g.~\cite[Section~9.1, Equation~(9.1.41)]{Atkinson}).
In particular, the map~$\RR\colon \Sz\setminus\Cs\to\NN$ defined by
\begin{equation} \label{retraction}
 \RR(Q) := s_* \left( \n(Q) - \frac13\Id\right) \qquad \textrm{for } Q\in\Sz\setminus\Cs
\end{equation}
is continuous.
It was proven in \cite[Lemma~3.10]{pirla} that~$\RR$ gives a retraction by deformation of~$\Sz\setminus\Cs$ onto the vacuum manifold~$\NN$.

\begin{lemma} \label{lemma:R C1}
 The retraction~$\RR$ is of class~$C^1$ on~$\Sz\setminus\Cs$. 
 Moreover,~$\RR$ coincides with the nearest-point projection onto~$\NN$, that is
 \begin{equation} \label{nearest-point projection}
  \abs{Q - \RR(Q)} \leq \abs{Q - P} 
 \end{equation}
 holds for any~$Q\in\Sz\setminus\Cs$ and any~$P\in\NN$, with strict inequality if~$P\neq\RR(Q)$.
\end{lemma}
\begin{proof}
Fix a matrix $Q_0\in\Sz\setminus\Cs$. 
By definition of~$\Cs$, the leading eigenvalue~$\lambda_1(Q)$ is simple.
Then, classical differentiability results for the eigenvectors (see e.g.~\cite[Section~9.1]{Atkinson}) imply
that the map~$Q\mapsto\n(Q)$ which appears in~\eqref{retraction} is locally of class~$C^1$, in a neighborhood of~$Q_0$.
As a consequence, $\RR$ is of class~$C^1$ in a neighborhood of~$Q_0$, and since $Q_0$ is arbitrary we have $\RR\in C^1(\Sz\setminus\Cs)$.

To show that~$\RR$ is the nearest point projection onto~$\NN$, we pick an arbitrary~$Q\in \Sz\setminus\Cs$ and~$P\in\NN$. 
By applying Lemma~\ref{lemma:representation} and~\eqref{characterizationN}, we write
\[
 Q = s \left(\n^{\otimes 2} - \frac13 \Id\right)  + sr \left(\m^{\otimes 2} - \frac13 \Id\right) \qquad \textrm{and} \qquad
 P = s_* \left(\p^{\otimes 2} - \frac 13 \Id \right)
\]
for some numbers~$s > 0$ and~$0 \leq r < 1$, some orthonormal pair~$(\n, \, \m)$ and some unit vector~$\p$.
We compute that
\begin{equation*} \label{retract-npp1}
 \begin{split}
  \abs{Q - P}^2 &= \abs{s\n^{\otimes 2} + sr\m^{\otimes 2} - s_*\p^{\otimes 2} - \frac13 (s + sr - s_*)\Id}^2 \\
  &= C(s, \, r, \, s_*) - 2s_*s \bigg\{(\n\cdot \p)^2 + r (\m\cdot \p)^2\bigg\}
 \end{split}
\end{equation*}
where $C(s, \, r, \, s_*)$ is a constant which only depends on~$s$, $r$ and~$s_*$.
For the last equality, we have used the identities $\u^{\otimes 2}\cdot\v^{\otimes 2} = (\u\cdot\v)^2$ and $\u^{\otimes 2}\cdot\Id = |\u|^2$, which hold for any vectors~$\u$, $\v$.
Given~$s$, $r$, $\n$ and~$\m$, we minimize the right-hand side with respect to~$\p$, subject to the constraint
\[
 (\n\cdot\p)^2 + (\m\cdot\p)^2 \leq 1.
\]
Since~$r < 1$, one easily sees that the minimum is achieved if and only if~$\p = \pm\n$, that is~$P = \RR(Q)$.
\end{proof}

Another function will be involved in the analysis of Section~\ref{subsect:JerrardSandier}.
Let~$\phi\colon\Sz\to\R$ be given by
\begin{equation} \label{phi_lambda}
 \phi(Q) = s_*^{-1} \left(\lambda_1(Q) - \lambda_2(Q)\right),
\end{equation}
where~$\lambda_1(Q)$ and~$\lambda_2(Q)$ are the largest and the second-largest eigenvalue of~$Q$. 
It is clear that~$\phi\geq 0$, and~$\phi(Q) = 0$ if and only if~$Q\in\Cs$.
Moreover, by applying Remark~\ref{remark: s,r} we have
\[
  \phi(Q) := s_*^{-1}s(Q) (1 - r(Q)) \qquad \textrm{for any } Q \in \Sz\setminus\{0\},
\]
therefore we have~$\phi(Q) = 1$ if~$Q\in\NN$, thanks to~\eqref{characterizationN}.

\begin{lemma} \label{lemma:phi}
 The function $\phi$ is Lipschitz continuous on~$\Sz$, of class $C^1$ on~$\Sz\setminus\Cs$ and satisfies
 \[
  \sqrt 2 s_*^{-1} \leq \abs{\D \phi(Q)} \leq 2 s_*^{-1} \qquad \textrm{for any } Q\in\Sz\setminus\Cs .
 \]
\end{lemma}
\begin{proof}
Thanks to standard regularity results for the eigenvalues (see e.g.~\cite[Equation~(9.1.32)]{Atkinson}),
we immediately deduce that~$\phi$ is locally Lipschitz continuous on~$\Sz$ and of class~$C^1$ on $\Sz\setminus\Cs$.
Let~$(\n, \, \m, \, \p)$ be an orthonormal set of eigenvectors relative to $(\lambda_1, \, \lambda_2, \, \lambda_3)$ respectively. Then, for any~$Q\in\Sz\setminus\Cs$ there holds
\[
 s_*\abs{\D \phi(Q)} = \max_{B\in\Sz, \ |B| = 1} \abs{\frac{\partial\phi}{\partial B}(Q)} = \max_{B\in \Sz, \ |B| = 1} \abs{\n \cdot B \n - \m \cdot B \m} 
\]
(the last identity follows by differentiating~\eqref{phi_lambda}, with the help of~\cite[Equation~(9.1.32)]{Atkinson} again). This implies~$\abs{\D \phi(Q)} \leq 2$. Now, set
\[
 B_0 := \frac{1}{\sqrt 2} \left( \n^{\otimes 2} - \m^{\otimes 2}\right) \in \Sz .
\]
Since $|\n^{\otimes 2}| = |\m^{\otimes 2}| = 1$ and~$\n^{\otimes2} \cdot \m^{\otimes 2} = 0$, it is straightforward to check that~$|B_0| = 1$, so
\[
 s_*\abs{\D \phi(Q)} \geq  \abs{\n \cdot B_0 \n - \m \cdot B_0 \m} = \frac{1}{\sqrt 2} \left( |\n|^2 + |\m|^2 \right) = \sqrt 2 . \qedhere
\]
\end{proof}

We conclude our discussion on the structure of the target space~$\Sz$ by proving a couple of properties of the Landau-de Gennes potential~$f$, defined by~\eqref{f}.
\begin{lemma} \label{lemma:f below}
 There exists a constant $\gamma = \gamma(a, \, b, \, c) > 0$ with the following properties. For any~$Q\in\Sz$, there holds
 \begin{equation} \label{f below} \tag{F\textsubscript{0}}
     f(Q) \geq \gamma \left(1 - \phi(Q)\right)^2 .
 \end{equation}
 For any~$Q_0\in\NN$ and any matrix~$P\in\Sz$ which is orthogonal to~$T_{Q_0}\NN$, we have
 \begin{equation} \label{f_non_deg} \tag{F\textsubscript{1}}
   \frac{\d^2}{\d t^2}_{|t = 0} \, f(Q_0 + t P) \geq \gamma |P|^2 .
 \end{equation}
 As a consequence of~\eqref{f_non_deg}, there exist~$\gamma^\prime$ and~$\delta_0 > 0$ such that, if~$Q\in \Sz$ satisfies $\dist(Q, \, \NN) \leq \delta_0$, then
 \begin{equation} \label{f dist2} \tag{F\textsubscript{2}}
   f(Q) \geq \gamma^\prime \dist^2(Q, \, \NN) 
 \end{equation}
 and
 \begin{equation} \label{f loc_convex} \tag{F\textsubscript{3}}
   f\left( tQ + (1 - t)\RR(Q) \right) \leq \gamma^\prime t^2 f(Q)
 \end{equation}
 for any~$0 \leq t \leq 1$.
\end{lemma} 

\begin{proof}[Proof of~\eqref{f below}]
Using the representation formula of Lemma~\ref{lemma:representation}, we can compute $\tr Q^2$ and $\tr Q^3$ as functions of $s := s(Q)$, $t := s(Q)r(Q)$. This yields
\[
 f(Q) = k - \frac{a}{3} \left(s^2 - st + t^2\right) - \frac{b}{27} \left(2s^3 - 3s^2t + 3st^2 - 2t^3 \right) + \frac{c}{9} \left(s^2 - st + t^2\right)^2 =: \tilde f (s, \, t) .
\]
We know that $(s_*, \, 0)$ is the unique minimizer of~$\tilde f$ (see e.g.~\cite[Proposition~15]{MajumdarZarnescu}), so $\D^2 \tilde f (s_*, \, 0) \geq 0$.
Moreover, it is straightforward to compute that
\[
\begin{split}
 \det \D^2 \tilde f(s_*, \, 0) > 0
 \end{split}
\]
thus $\D^2 \tilde f(s_*, \, 0) > 0$.
As a consequence, there exist two numbers $\delta > 0$ and $C > 0$ such that
\begin{equation} \label{f below 1}
 \tilde f (s, \, sr) \geq C(s_* - s)^2 + C s^2r^2 \qquad \textrm{if } (s - s_*)^2 + s^2r^2 \leq \delta .
\end{equation}
The left-hand side in this inequality is a polynomial of order four with leading term~$\frac{c}{9}(s^2 - st + t^2)^2 \geq\frac{c}{36} (s^2 + t^2)^2$, whereas the right-hand side is a polynomial of order two.
Therefore, there exists a positive number~$M$ such that
\begin{equation} \label{f below 2}
 \tilde f (s, \, sr) \geq C(s_* - s)^2 + C s^2r^2 \qquad \textrm{if } (s - s_*)^2 + s^2r^2 \geq M.
\end{equation}
Finally, we have $\tilde f(s, \, t) > 0$ for any~$(s, \, t)\neq (s_*, \, 0)$, so there exists a positive constant~$C^\prime$ such that
\begin{equation} \label{f below 3}
 \tilde f (s, \, sr) \geq C^\prime \qquad \textrm{if } \delta < (s - s_*)^2 + s^2r^2 \leq M.
\end{equation}
Combining \eqref{f below 1}, \eqref{f below 2} and~\eqref{f below 3}, and modifying the value of $C$ if necessary, for any $Q\in\Sz$, $s = s(Q)$, $r = r(Q)$ we obtain
\[
\begin{split}
 \tilde f(s, \, sr) \geq C (s_* - s)^2 + C s^2 r^2 \geq \frac{C s_*^2}{2} \left(1 - \frac{s}{s_*} + \frac{sr}{s_*}\right)^2 
                    = \frac{C s_*^2}{2} \left(1 - \phi(Q)\right)^2 . \qedhere
 \end{split}
\]
\end{proof}

\begin{proof}[Proof of~\eqref{f_non_deg}]
Since the group~$\mathrm{SO}(3)$ acts transitively on the manifold~$\NN$ and the potential~$f$ is preserved by the action, 
it suffices to check~\eqref{f_non_deg} in case
\begin{equation} \label{Qzero}
 Q_0 = s_* \left(\mathbf{e}^{\otimes 2} - \frac13 \Id \right).
\end{equation}
Indeed, for any $Q\in\Sz$ there exists~$\n\in \S^2$ such that
\[
 \RR(Q) = s_* \left(\n^{\otimes 2} - \frac13 \Id \right) ,
\]
and there exists a matrix $R\in\mathrm{SO}(3)$ such that $R\n = \mathbf{e}_3$.
As is easily checked, the function $\xi_R\colon Q \mapsto RQ R^{\mathsf{T}}$ maps isometrically $\Sz$ onto itself. Then,~\eqref{nearest-point projection} implies that~$\xi_R$ commutes with $\RR$, so
\[
 \RR(\xi_R(Q)) = \xi_R(\RR(Q)) = s_* \left(R \n (R \n)^{\mathsf{T}} - \frac13 \Id \right) = s_* \left(\mathbf{e}_3 \mathbf{e}_3^{\mathsf{T}} - \frac13 \Id \right).
\]
On the other hand,~$f$ is invariant by composition with $\xi_R$ (i.e.~$f\circ\xi_R = f$) because it is a function of the scalar invariants of~$Q$.
Therefore, we assume WLOG that~$Q_0$ is given by~\eqref{Qzero}.

In~\cite[Lemma~3.5]{pirla}, it is shown that a matrix $P\in\Sz$ is orthogonal to $T_{{Q}_0}\NN$ if and only if it can be written in the form
\[
  P = \left(\begin{matrix}
            -\dfrac13 (s_* + x_0) + x_2 & x_1                         & 0  \\
            x_1                         & -\dfrac13 (s_* + x_0) - x_2 & 0 \\
            0                           & 0                           & \dfrac23 (s_* + x_0)
           \end{matrix} \right)
\]
for some $x = (x_0, \, x_1, \, x_2)\in\R^3$.
Then, one can write~$f(Q_0 + tP)$ as a function of $t$, $x$ and compute the second derivatives.
The proof of~\eqref{f_non_deg} is reduced to a straightforward computation, which we omit here.
\end{proof}

\begin{proof}[Proof of~\eqref{f dist2}--\eqref{f loc_convex}]
 These properties follow from~\eqref{f_non_deg}, by a Taylor expansion of~$f$ around~$Q_0$.
\end{proof}

\subsection{Energy estimates in $2$-dimensional domains}
\label{subsect:JerrardSandier}

In the analysis of the Ginzburg-Landau functional, a very useful tool are the estimates proved by Jerrard~\cite{Jerrard} and Sandier~\cite{Sandier}.
These estimates provide a lower bound for the energy of complex-valued maps defined on a two-dimensional disk, depending on the topological properties of the boundary datum.
More precisely, if~$u\in H^1(B^2_1, \, \C)$ satisfies~$|u(x)|= 1$ for a.e.~$x\in \partial B^2_1$ (plus some technical assumptions) then
\begin{equation} \label{GL-jerrardsandier}
 E^{\mathrm{GL}}_\varepsilon(u, \, B^2_1) \geq \pi\abs{d} \abs{\log\varepsilon} - C ,
\end{equation}
where~$E^{\mathrm{GL}}_\varepsilon$ is the Ginzburg-Landau energy, defined by~\eqref{GL}, and~$d$ denotes the topological degree of~$u/|u|$ on~$\partial B^2_1$, i.e. its winding number.
The aim of this subsection is to generalize this result to tensor-valued maps and the Landau-de Gennes energy.

Since we work in the~$H^1$-setting, we have to take care of a technical detail.
Set~$A := B^2_1\setminus B^2_{1/2}$. Let~$Q\in H^1(B^2_1, \, \Sz)$ be a given map, which satisfies
 \begin{equation} \label{no defect jerrard}
  \phi_0(Q, \, A) := \underset{A}{\infess} \, \phi\circ Q > 0.
 \end{equation}
In case~$Q$ is continuous, Condition~\eqref{no defect jerrard} is equivalent to
\begin{equation*} 
 Q(x) \notin \Cs \qquad \textrm{for every } x\in\overline{A} .
\end{equation*}
For a.e.~$r\in (1/2, \, 1)$ the restriction of~$Q$ to~$\partial B^2_r$ is an~$H^1$-map (due to Fubini theorem)
and hence, by Sobolev injection, it is a continuous map which satisfies~$Q(x)\notin \Cs$ for every~$x\in \partial B^2_r$.
Therefore,~$\RR \circ Q$ is well defined and continuous on~$\partial B^2_r$.
Moreover, its homotopy class is independent of~$r$.
If~$\RR \circ Q$ is continuous, then~$\RR \circ Q$ itself provides a homotopy between~$\RR \circ Q_{|\partial B^2_{r_1}}$ and~$\RR \circ Q_{|\partial B^2_{r_2}}$, for any~$r_1$ and~$r_2$.
Otherwise, by convolution (as in~\cite[Proposition p.~267]{SchoenUhlenbeck2}) one constructs a smooth
approximation~$(\RR \circ Q)_\delta\colon A\to\NN$ such that $(\RR \circ Q)_\delta \to \RR \circ Q$ in~$H^1(A, \, \Sz)$ when~$\delta\to 0$. 
By Fubini theorem and Sobolev injection, we have $(\RR \circ Q)_\delta \to \RR \circ Q$ uniformly on~$\partial B^2_r$ for a.e.~$r\in (1/2, \, 1)$.
Therefore, the maps~$\RR \circ Q_{|\partial B^2_r}$ belong to the same homotopy class, for a.e.~$r$.
By abuse of notation, this homotopy class will be called ``homotopy class of~$\RR\circ Q$ restricted to the boundary'' or also ``homotopy class of the boundary datum''.

\begin{prop} \label{prop:lower bound}
 There exist positive constants~$M$ and $\kappa_*$, depending only on~$f$, with the following property.
 Let $0 < \varepsilon < 1$ and~$Q\in H^1(B^2_1, \, \Sz)$ be given. 
 Assume that~$Q$ satisfies~\eqref{no defect jerrard} and the homotopy class of $\RR \circ Q_{|\partial B^2_1}$ is non-trivial. Then
 \[
  E_\varepsilon(Q, \, B_1^2) \geq \kappa_* \phi_0^2(Q, \, A) \abs{\log\varepsilon} - M .
 \]
\end{prop}

The energetic cost associated with topological defects is quantified by a number~$\kappa_*$, defined by~\eqref{kappa*} and explicitly computed in Lemma~\ref{lemma:kappa*}:
\[
 \kappa_* = \frac{\pi}{2} s_*^2 .
\]
This number plays the same role as the quantity~$\pi|d|$ in~\eqref{GL-jerrardsandier}.
The quantity~$\phi_0^2(Q, \, A)$ at the right-hand side has been introduced for technical reasons.
Notice that~$\phi = 1$ on~$\NN$, so~$\phi_0(Q, \, A) = 1$ if~$Q_{|A}$ takes values in~$\NN$.

Before dealing with the proof of Proposition~\ref{prop:lower bound}, we state an immediate consequence.
\begin{cor} \label{cor:lower bound}
 Let $\varepsilon$, $R$ be two numbers such that~$0 < \varepsilon < R/2$.
 Given a map~$Q\in H^1(B^2_R, \, \Sz)$, suppose that the restriction to the boundary belongs to~$H^1(\partial B^2_R, \, \Sz)$ and that
 \[
  \phi_0(Q, \, \partial B^2_R) := \underset{\partial B^2_R}{\infess} \, \phi\circ Q > 0.
 \]
 If the homotopy class of $\RR \circ{Q}_{|\partial B^2_R}$ is non-trivial, then
 \[
  E_\varepsilon(Q, \, B_R^2) + (\log2) R \, E_\varepsilon(Q, \, \partial B_R^2) \geq \kappa_* \phi_0^2(Q, \, \partial B^2_R) \log\frac{R}{\varepsilon} - M .
 \]
\end{cor}
In particular, if~$Q$ satisfies $Q=g$ on $\partial B^2_R$ for some non-trivial~$g\in H^1(\partial B^2_R, \, \NN)$, then Corollary~\ref{cor:lower bound} implies
\[
 E_\varepsilon(Q, \, B_R^2) \geq \kappa_* \log\frac{R}{\varepsilon} - M
\]
for a constant~$M = M(R, \, g)$. (Compare this estimate with~\cite[Theorem~3.1]{Jerrard},~\cite[Theorem~1]{Sandier}, and~\cite[Proposition~6.1]{Chiron}.)
\begin{proof}[Proof of Corollary~\ref{cor:lower bound}]
We apply Proposition~\ref{prop:lower bound} to~$\epsilon := 2\varepsilon/R$ and the map~$\tilde Q \in H^1(B_1^2, \, \Sz)$ defined by
\[
 \tilde Q(x) := \begin{cases}
                 Q\left(\dfrac{Rx}{|x|}\right) & \textrm{if } x\in A := B^2_1 \setminus B^2_{1/2} \\
                 Q\left(2Rx\right)             & \textrm{if } x\in B^2_{1/2} .
                \end{cases}
\]
Notice that~$\phi_0(\tilde Q, \, A) = \phi_0(Q, \, \partial B^2_R)$. Then, by a change of variable, we deduce
\[
\begin{split}
 \kappa_* \phi^*(Q) \log\frac{R}{\varepsilon} - C \leq E_\epsilon(\tilde Q, \, B^2_1) &\leq E_\epsilon(Q, \, B^2_{1/2}) + \int_{1/2}^1 E_\epsilon(\tilde Q, \, \partial B^2_r) \, \d r\\
 &= E_\varepsilon(Q, \, B^2_R) + \int_{1/2}^1 \frac{R}{r} E_{2\varepsilon/r}(Q, \, \partial B^2_R) \, \d r \\
 &\leq E_\varepsilon(Q, \, B^2_R) + (\log 2) R \, E_\varepsilon(Q, \, \partial B^2_R). \qedhere
\end{split}
\]
\end{proof}

A generalization of the Jerrard-Sandier estimate~\eqref{GL-jerrardsandier} has already been proved by Chiron, in his PhD thesis~\cite{Chiron}.
Given a smooth, compact manifold without boundary, Chiron considered maps into the \emph{cone over}~$\NN$, that is
\[
 X_{\NN} := \left((0, \, +\infty)\times\NN\right) \cup \{0\} \ni u = \left(|u|, \, u/|u|\right)
\]
(with a metric defined accordingly). 
He obtained an estimate analogous to~\eqref{GL-jerrardsandier}.
In case~$\NN = \S^1$, one has~$X_{\S^1} \simeq \C$ and the standard estimate~\eqref{GL-jerrardsandier} is recovered. Given
a map $u\colon U\subseteq\R^k\to X_{\NN}$, a key step in Chiron's arguments is to decompose the gradient of $u$ in terms of modulus and phase, that is
\begin{equation} \label{cone_struct}
 \abs{\nabla u}^2 = \abs{\nabla|u|}^2 + \abs{u}^2 \abs{\nabla\left(u/|u|\right)}^2 \qquad \textrm{a.e. on } U .
\end{equation}

Chiron's result does not apply to tensor-valued maps, because the space~$\Sz$ do not coincide with the cone over~$\NN$ 
(the latter only contains uniaxial matrices, whereas~$\Sz$ also contains biaxial matrices).
However, one can prove an estimate in the same spirit as~\eqref{cone_struct}, namely, 
the gradient of a map $\Omega\to\Sz$ is controlled from below by the gradients of~$\phi\circ Q$ and~$\RR \circ Q$.

\begin{lemma} \label{lemma:energy retraction}
 Let $U\subseteq\R^k$ be a domain and let $Q\in C^1(U, \, \Sz)$. There holds
 \begin{equation*} 
  \abs{\nabla Q}^2 \geq \frac{s_*^2}{3} \abs{\nabla\left(\phi\circ Q\right)}^2 + \left(\phi\circ Q\right)^2 \abs{\nabla \left(\RR \circ Q\right)}^2 \qquad \H^k \textrm{-a.e. on } U,
 \end{equation*}
 where we have set $(\phi\circ Q)|\nabla (\RR \circ Q)|(x) := 0$ if~$Q(x)\in \Cs$.
\end{lemma}

\begin{proof}
First of all, notice that $\RR\circ Q$ is well-defined on the set where~$Q\notin\Cs$, or equivalently, the set where~$\phi\circ Q >0$.
Therefore, the right-hand side always makes sense.
Because of our choice of the norm, we have
\begin{equation*} 
 \abs{\nabla \psi}^2 = \sum_{i = 1}^k \abs{\partial_{x_i} \psi}^2
\end{equation*}
for any scalar or tensor-valued map $\psi$. 
Thus, it suffices to prove the lemma when $\nabla$ is replaced by the partial derivative operator $\partial_{x_i}$, then sum over $i = 1, \, \ldots, \, k$.
In view of this remark, we assume WLOG that $k = 1$. 

Since $\phi$ is Lipschitz continuous (Lemma~\ref{lemma:phi}), we know that $\phi\circ Q\in W^{1, \infty}_{\mathrm{loc}}(U)$.
Moreover, $\phi\circ Q=0$ on~$Q^{-1}(\Cs)$.
Therefore, we have $(\phi\circ Q)^\prime = 0$ a.e. on~$Q^{-1}(\Cs)$ and the lemma is trivially satisfied a.e. on~$Q^{-1}(\Cs)$.

For the rest of the proof, we fix a point $x\in U \setminus Q^{-1}(\Cs)$ so that $\phi\circ Q$, $\RR\circ Q$ are of class $C^1$ in a neighborhood of~$x$,
and the leading eigenvalue of $Q(x)$ has multiplicity one.
We are going to distinguish a few cases, depending on whether the others eigenvalues of~$Q(x)$ have multiplicity one as well.
Suppose first that $r(Q(x)) > 0$: in this case, all the eigenvalues of $Q(x)$ are simple.
Using Lemma~\ref{lemma:representation} and the results in \cite{Atkinson}, the map~$Q$ can be locally written as
\begin{equation} 
 Q = s\left(\n^{\otimes 2} - \frac13 \Id\right) + sr \left(\m^{\otimes 2} - \frac13 \Id\right) ,
\end{equation}
where $s$, $r$, $\n$, $\m$ are $C^1$ functions defined in a neighborhood of $x$, satisfying the constraints
\[
 s > 0, \qquad 0 < r < 1, \qquad \abs{\n} = \abs{\m} = 1 , \qquad \n \cdot \m = 0.
\]
Then,
$\RR \circ Q$ is of class $C^1$ in a neighborhood of $x$, and we can compute $|Q^\prime|$, $|(\RR \circ Q)^\prime|$ in terms of~$s$, $r$, $\n$, $\m$ and their derivatives.
Setting $t:=sr$, a straightforward computation gives
\begin{equation*} \label{norm P prime}
 s_*^2{\left(\phi\circ Q\right)^\prime}^2 = {s^\prime}^2 - 2s^\prime t^\prime + {t^\prime}^2 , \qquad \abs{(\RR \circ Q)^\prime}^2 = 2s_*^2 \abs{\n^\prime}^2
\end{equation*}
and
\begin{equation} \label{Q prime conto}
\begin{split}
 \abs{Q^\prime}^2 &= \frac23 \left( {s^\prime}^2 - s^\prime t^\prime + {t^\prime}^2 \right) + 2s^2 \abs{\n^\prime}^2 +2t^2 \abs{\m^\prime}^2 +4st (\n^\prime \cdot \m) (\n \cdot \m^\prime) \\
                  &\geq \frac{s_*^2}{3} {\left(\phi\circ Q\right)^\prime}^2 + 2s^2 \left( \abs{\n^\prime}^2 + r^2 \abs{\m^\prime}^2 + 2 r (\n^\prime \cdot \m) (\n \cdot \m^\prime)\right)
 \end{split}
\end{equation}
Let $\p:=\n\times \m$, so that $(\n, \, \m, \, \p)$ is an orthonormal, positive frame in $\R^3$.
By differentiating the orthogonality conditions for $(\n, \, \m, \, \p)$, we obtain
\[
 \n^\prime = \alpha \m + \beta \p , \qquad \m^\prime = -\alpha \n + \gamma \p
\]
for some smooth, real-valued functions $\alpha$, $\beta$, $\gamma$.
Then, from \eqref{Q prime conto} and \eqref{norm P prime} we deduce
\[
\begin{split}
 \abs{Q^\prime}^2 - \frac{s_*^2}{3} {\left(\phi\circ Q\right)^\prime}^2 & \geq 2 s^2 \left( \alpha^2 + \beta^2 + r^2 (\alpha^2 + \gamma^2) - 2r \alpha^2 \right) \\
                  & \geq 2 s^2 (1 - r)^2 (\alpha^2 + \beta^2) \\
                  & = s_*^{-2}{s^2 (1 - r)^2} \abs{(\RR \circ Q)^\prime}^2 = \left(\phi\circ Q\right)^2 \abs{(\RR \circ Q)^\prime}^2 ,
 \end{split}
\]
so the lemma holds at the point $x$.

If $r(Q) = 0$ in a neighborhood of $x$ then the function~$\m$ might not be well-defined.
However, the previous computation still make sense because $t = sr$ vanishes in a neighborhood of~$x$,
and from~\eqref{norm P prime},~\eqref{Q prime conto} we deduce that the lemma holds at~$x$.
We still have to consider one case, namely, $r(Q(x)) = 0$ but $r(Q)$ does not vanish identically in a neighborhood of $x$.
In this case, there exists a sequence $x_k \to x$ such that $r(Q(x_k)) > 0$ for each $k\in\N$.
By the previous discussion the lemma holds at each~$x_k$, and the functions~$\phi\circ Q$,
~$(\RR\circ Q)^\prime$ are continuous (by Lemmas~\ref{lemma:phi} and~\ref{lemma:R C1}). 
Passing to the limit as~$k\to +\infty$, we conclude that the lemma is satisfied at~$x$ as well.
\end{proof}


The regularity of~$Q$ in Lemma~\ref{lemma:energy retraction} can be relaxed. 

\begin{cor} \label{cor:energy retraction}
 The map $\tau\colon\Sz\to\Sz$ given by
 \[
  \tau \colon Q \mapsto \begin{cases}
				s_* \phi(Q) \RR(Q)  & \textrm{if } Q \in \Sz\setminus \Cs \\
				0			   & \textrm{if } Q \in\Cs
                        \end{cases}
 \]
  is Lipschitz-continuous.
 Moreover, for any~$Q\in H^1(U, \, \Sz)$ there holds~$\tau\circ Q\in H^1(U, \, \Sz)$ and
 \begin{equation} \label{pointwise gradient 2}
   \frac14 \abs{\nabla\left(\tau\circ Q\right)}^2 \leq \frac{s_*^2}{3} \abs{\nabla\left(\phi\circ Q\right)}^2 + \left(\phi\circ Q\right)^2 \abs{\nabla (\RR \circ Q)}^2 \leq  \abs{\nabla Q}^2
   \qquad \mathscr H^k\textrm{-a.e. on } U.
 \end{equation}
\end{cor}
\begin{proof}
By differentiating~$\tau$ and applying Lemma~\eqref{lemma:energy retraction} to the map~$Q = \Id_{\Sz}$, we obtain
\[
 \frac14 \abs{\D \tau}^2 \leq \frac{s_*^2}{3} \abs{\D\phi}^2 + \phi^2 \abs{\D\RR}^2 \leq C \qquad \textrm{on } \Sz\setminus \Cs .
\]
Using this uniform bound, together with~$\tau\in C(\Sz, \, \Sz)$ and~$\tau_{|\Cs} = 0$, it is not hard to conclude that $\tau$ has bounded derivative in the sense of distributions,
therefore~$\tau$ is a Lipschitz function and the lower bound in~\eqref{pointwise gradient 2} holds.
The upper bound follows easily from Lemma~\eqref{lemma:energy retraction}, by a density argument.
\end{proof}

Following an idea of Chiron~\cite{Chiron}, we can associate with each homotopy class of maps~$\S^1\to\NN$ a positive number which measures the energy cost of that class.
Since~$\NN$ is a real projective plane, quantifying the energy cost of non-trivial maps is simple, because there is a unique homotopy class of such maps.
Define
\begin{equation} \label{kappa*}
 \kappa_* := \inf\left\{ \frac12 \int_{\S^1} \abs{P^\prime(\theta)}^2 \, \d\theta \colon P\in H^1(\S^1, \, \NN) \textrm{ is homotopically non-trivial }\right\} .
\end{equation}
Thanks to the compact embedding $H^1(\S^1, \, \NN)\hookrightarrow C^0(\S^1, \, \NN)$, a standard argument shows that the infimum is achieved.
The Euler-Lagrange condition for~\eqref{kappa*} implies that minimizers are geodesics in~$\NN$.
Moreover, we have the following property.

\begin{lemma} \label{lemma:kappa*}
 A minimizer for \eqref{kappa*} is given by
 \[
  P_0(\theta) := s_* \left( \n_*(\theta)^{\otimes 2} - \frac13 \Id\right) \qquad \textrm{for  } 0 \leq \theta \leq 2\pi ,
 \]
 where $\n_*(\theta) := \left(\cos(\theta/2), \, \sin(\theta/2), \, 0\right)^{\mathsf{T}}$. In particular, there holds
 \[
  \kappa_* = \frac{\pi}{2} s_*^2.
 \]
\end{lemma}
\begin{proof}[Sketch of the proof]
The lemma has been proved, e.g., in~\cite[Lemma~3.6, Corollary~3.7]{pirla}, but we sketch the proof for the convenience fo the reader.
Let~$\psi\colon\S^2\to\NN$ be the universal covering of~$\NN$, that is
\begin{equation} \label{covering}
  \psi(\n) := s_* \left( \n^{\otimes 2} - \frac13 \Id \right) \qquad \textrm{for } \n\in \S^2 .
\end{equation}
For any~$\n\in\S^2$ and any tangent vector~$\v\in T_\n \S^2$,  one computes that
\[
 |\d\psi(\n)\v|^2 = 2 s_*^2 |\v|^2,
\]
that is, the pull-back metric induced by~$\psi$ coincides with the first fundamental form of~$\NN$, up to a constant factor.
Therefore, the Levi-Civita connections associated with the two metrics are the same,
because the Christoffel symbols are homogeneous functions, of degree zero, of the coefficients of the metric.
As a consequence, a loop~$P$ is a geodesic in~$\NN$ if and only if it can be written as $P = \psi\circ\n$, 
where $\n\colon [0, \, 2\pi]\to \S^2$ is a geodesic path in $\S^2$ such that~$\n(2\pi) = \pm\n(0)$.
Moreover, $P$ has a non-trivial homotopy class if and only if~$\n(2\pi) = -\n(0)$.
Therefore, $P := \psi\circ\n$ is a minimizing geodesic in the non-trivial class if and only if~$\n$ is half of a great circle in $\S^2$ parametrized by arc-length, and the lemma follows.
\end{proof}

By adapting Sandier's arguments in~\cite{Sandier}, we can bound from below the energy of~$\NN$-valued maps, in terms of the quantity~\eqref{kappa*}.
We use the following notation: for any $V \csubset\R^2$, we define the radius of $V$ as
\begin{equation*} 
 \mathrm{rad}(V) := \inf\left\{ \sum_{i = 1}^n r_i \colon \overline{V} \subseteq \bigcup_{i = 1}^n B(a_i, \, r_i)\right\} .
\end{equation*}
We clearly have~$\mathrm{rad}(V) \leq \mathrm{diam}(\overline{V})$ and, since for bounded sets there holds~$\mathrm{diam}(\overline{V}) = \mathrm{diam}(\partial V)$, we obtain that 
\begin{equation} \label{rad-diam}
 \mathrm{rad}(V) \leq \mathrm{diam}(\partial V).
\end{equation}

\begin{lemma} \label{lemma:lower bound-sand}
 Let~$V$ be a subdomain of~$B^2_1$ and let~$\rho > 0$ be such that $\dist(V, \, \partial B^2_1) \geq 2\rho$.
 For any $P\in H^1(B^2_1 \setminus V, \, \NN)$ such that ${P}_{|\partial B^2_1}$ is homotopically non-trivial, there holds
 \[
  \frac12 \int_{B^2_1\setminus V} \abs{\nabla P}^2 \, \d \H^2 \geq \kappa_* \log\frac{\rho}{\mathrm{rad}(V)} .
 \]
\end{lemma}
\begin{proof}[Sketch of the proof]
 Suppose, at first, that $V = B^2_r$ with~$0 < r < 1$ and $u$ is smooth.
 Then, computing in polar coordinates, we obtain
 \[
 \begin{split}
  \frac12 \int_{B_1^2\setminus B^2_r} \abs{\nabla P}^2 \, \d \H^2 
   &= \frac12 \int_r^1  \int_{\S^1} \left( \rho \abs{\frac{\d P}{\d \rho}}^2 + \frac{1}{\rho} \abs{\frac{\d P}{\d \theta}}^2\right) \, \d\theta \, \d\rho\\
   &\stackrel{\eqref{kappa*}}{\geq} \kappa_* \int_r^1 \frac{\d \rho}{\rho} = \kappa_* \log\frac{1}{r}
 \end{split}
 \]
 so the lemma is satisfied for any~$0 < \rho \leq 1$.
 By a density argument, the same estimate holds for any $P\in H^1(B^2_1\setminus B_r^2, \, \NN)$. 
 For a general $V$, the lemma can be proved arguing exactly as in~\cite[Proposition p.~385]{Sandier}. 
 (Assuming additional $W^{1, \, \infty}$-bounds on~$P$, the lemma could also be deduced by the arguments of~\cite[Theorems~3.1 and~4.1]{Jerrard}.)
\end{proof}

Finally, we can prove the main result of this subsection.

\begin{proof}[Proof of Proposition~\ref{prop:lower bound}]
We argue as in~\cite[Theorem~6.1]{Chiron} and~\cite[Proposition~3.11]{pirla}.
As a first step, we suppose that~$Q$ is smooth. Reminding that~$A := B^2_1\setminus B^2_{1/2}$, we have
\[
 \phi_0 := \phi_0(Q, \, A) = \underset{\overline A}{\min} \, \phi\circ Q \stackrel{\eqref{no defect jerrard}}{>} 0.
\]
Moreover, there must be
\begin{equation} \label{min phi}
 \underset{\overline{B}_1^2}{\min} \, \phi\circ Q = 0 ,
\end{equation}
otherwise~$\RR \circ Q$ would be a well-defined, continuous map~$\overline{B}^2_1\to\NN$ and the boundary datum would be topologically trivial. 
For each $\lambda > 0$, we set
\[
 \Omega_\lambda := \left\{x\in B_1^2 \colon \phi\circ Q(x) > \lambda \right\} , \quad \omega_\lambda := \left\{x\in B_1^2 \colon \phi\circ Q(x) < \lambda\right\} , \quad
 \Gamma_\lambda := \partial\Omega_\lambda \setminus \partial\Omega = \partial \omega_\lambda .
\]
Notice that $\Omega_\lambda$, $\omega_\lambda$, and $\Gamma_\lambda$ are non empty for every~$\lambda\in (0, \, \phi_0)$, due to~\eqref{min phi}.
We also set
\[
 \Theta(\lambda) := \int_{\Omega_\lambda} \abs{\nabla \left( \RR \circ Q \right)}^2 \, \d \H^2 , \qquad 
 \nu(\lambda) := \int_{\Gamma_\lambda} \abs{\nabla\left(\phi\circ Q\right)} \, \d\H^1 .
\]
Lemma~\ref{lemma:energy retraction} entails
\[
 \int_{B^2_1} \abs{\nabla Q}^2 \geq \int_{B^2_1} \left\{ \frac12 \abs{\nabla \left( \phi \circ Q \right)}^2 +
  \left(\phi\circ Q\right)^2 \abs{\nabla \left( \RR \circ Q \right)}^2 \right\} \, \d\H^2
\]
and, applying the coarea formula, we deduce
\begin{equation} \label{lower LG 1}
 E_\varepsilon(Q) \geq \frac12 \int_0^{\phi_0} \left\{ \int_{\Gamma_\lambda}\left( \frac12 \abs{\nabla \left(\phi\circ Q\right) } 
 + \frac{2 f(Q)}{\varepsilon^2 \abs{\nabla \left(\phi\circ Q\right)}}\right) \d\H^1
 - 2\lambda^2 \Theta^\prime(\lambda) \right\} \, \d\lambda \, .
\end{equation}
Thanks to Sard lemma, a.e.~$\lambda\in (0, \, \phi_0)$ is a proper regular value of~$\phi\circ Q$, so dividing by~$|\nabla(\phi\circ Q)|$ makes sense.
Let us estimate the terms in the right-hand side of \eqref{lower LG 1}, starting from the second one.
Lemma~\ref{lemma:f below}, ~\eqref{f below} implies that
\[
 f(Q) \geq C \left(1 - \lambda\right)^2  \qquad \textrm{on }\Gamma_\lambda .
\]
Therefore, with the help of the H\"older inequality we deduce
\begin{equation} \label{lower LG 3}
\int_{\Gamma_\lambda} \frac{2 f(Q)}{\varepsilon^2) \abs{\nabla \left(\phi\circ Q\right)}} \, \d \H^1 \geq
\frac{C\left(1 - \lambda\right)^2}{\varepsilon^2} \int_{\Gamma_\lambda} \frac{1}{\abs{\nabla \left(\phi\circ Q\right)}} \, \d\H^1 
\geq \frac{C\left(1 - \lambda\right)^2 \H^1(\Gamma_\lambda)^2}{\varepsilon^2 \nu(\lambda)} .
\end{equation}
Moreover, we have
\[
 \H^1(\Gamma_\lambda) \geq 2 \, \mathrm{diam}(\Gamma_\lambda) \stackrel{\eqref{rad-diam}}{\geq} 2\,\mathrm{rad}(\omega_\lambda) .
\]
Combining this with \eqref{lower LG 1} and \eqref{lower LG 3}, we find
\begin{equation} \label{lower GL 2}
\begin{split}
 E_\varepsilon(Q) &\geq \frac12 \int_0^{\phi_0} \left\{ \frac12 \nu(\lambda) + \frac{C\left(1 - \lambda\right)^2 \mathrm{rad}^2(\omega_\lambda)}{\varepsilon^2 \nu(\lambda)} \right\} \, \d \lambda -
 \int_0^{\phi_0} \lambda^2 \Theta^\prime(\lambda) \, \d\lambda  \\
  &\geq \int_0^{\phi_0} \frac{C}{\varepsilon} \abs{1 - \lambda} \mathrm{rad}(\omega_\lambda) \, \d\lambda - \int_0^{\phi_0} \lambda^2 \Theta^\prime(\lambda) \, \d\lambda .
\end{split}
\end{equation}
The second line follows by the elementary inequality $a + b \geq 2 \sqrt{ab}$.
As for the last term, we integrate by parts. For all $\lambda_0 > 0$, we have
\[
 - \int_{\lambda_0}^{\phi_0} \lambda^2 \Theta^\prime(\lambda) \, \d\lambda = 2 \int_{\lambda_0}^{\phi_0} \lambda \Theta(\lambda) \, \d\lambda 
 + {\lambda_0}^2 \Theta({\lambda_0}) \geq 2\int_{\lambda_0}^{\phi_0} \lambda \Theta(\lambda) \, \d\lambda
\]
and, letting ${\lambda_0} \to 0$, by monotone convergence ($\Theta \geq 0$, $-\Theta^\prime\geq 0$) we conclude that
\[
 - \int_0^{\phi_0} \lambda^2 \Theta^\prime(\lambda) \, \d\lambda \geq  2\int_0^{\phi_0} \lambda \Theta(\lambda) \, \d\lambda .
\]
Now, for any~$\lambda\in (0, \, \phi_0)$ we have~$\omega_\lambda \subseteq B^2_{1/2}$, so~$\dist(\omega_\lambda, \, \partial B^2_1) \geq 1/2$.
Therefore, by applying Lemma~\ref{lemma:lower bound-sand} we obtain
\[
 \Theta(\lambda) \geq -\kappa_* \log\left(\mathrm{rad}(\omega_\lambda)\right) - \kappa_* \log 4.
\]
Thus,~\eqref{lower GL 2} implies
\[
 F_t(Q) \geq \int_0^{\phi_0} \left\{ \frac{C}{\varepsilon} \abs{1 - \lambda} \mathrm{rad}(\omega_\lambda) - 2\kappa_*\lambda \log\left(\mathrm{rad}(\omega_\lambda)\right) \right\} \, \d\lambda - C .
\]
An easy analysis shows that the function $r\in (0, \, +\infty)\mapsto C\varepsilon^{-1} |1 - \lambda| r - 2\kappa_*\lambda \log r$ has a unique minimizer~$r_*$, which is readily computed.
As a consequence, we obtain the lower bound
\[
\begin{split}
 F_t(Q) &\geq \int_0^{\phi_0} \left\{  2\kappa_*\lambda -  2\kappa_*\lambda \log\frac{C\varepsilon\kappa_*\lambda}{\abs{1 - \lambda}} \right\} \, \d\lambda - C \\
           &=  -2\kappa_* \int_0^{\phi_0} \left\{ \lambda \log\varepsilon - \lambda + \lambda \log\frac{C\kappa_* \lambda}{\abs{1 - \lambda}} \right\} \, \d\lambda - C 
\end{split}
\]
All the terms are locally integrable functions of~$\lambda$, so the proposition is proved in case~$Q$ is smooth.

Given any~$Q$ in~$H^1$, we can reduce to previous case by means of a density argument, inspired by~\cite[Proposition p.~267]{SchoenUhlenbeck2}.
For~$\delta > 0$, let~$\chi^\delta$ be a standard mollification kernel and set~$Q^\delta := Q * \chi^\delta$.
(In order to define the convolution at the boundary of~$\Omega$, we extend~$Q$ by standard reflection on a neighborhood of the domain.)
Then,~$\{Q^\delta\}_{\delta > 0}$ is a sequence of smooth maps, which converge to~$Q$ strongly in~$H^1$ and, by Sobolev embedding, in~$L^4$.
This implies~$E_\varepsilon(Q^\delta) \to E_\varepsilon(Q)$ as~$\delta\to 0$. Moreover, for any~$x\in A$ we have
\begin{equation} \label{SU2}
 \begin{split}
 \dist\left(\phi\circ Q^\delta (x), \, [\phi_0, \, +\infty)\right)
 &\leq \fint_{B^2_\delta(x)} \abs{\phi\circ Q^\delta (x) - \phi\circ Q(y)} \, \d\H^2(y) \\
 &\leq C \fint_{B^2_\delta(x)} \abs{Q^\delta(x) - Q(y)} \, \d\H^2(y) ,
 \end{split}
\end{equation}
where the last inequality follows by the Lipschitz continuity of~$\phi$ (Lemma~\ref{lemma:phi}).
Now, we can adapt the Poincar\'e-Wirtinger inequality
and combine it with the H\"older inequality to obtain
\[
 \int_{B^2_\delta(x)} \abs{Q^\delta(x) - Q(y)} \, \d\H^2(y) \leq C \delta \int_{B^2_\delta(x)} \abs{\nabla Q} \, \d\H^2 
 \leq C \delta^2 \left( \int_{B^2_\delta(x)} \abs{\nabla Q}^2 \, \d\H^2 \right)^{1/2} .
\]
This fact, combined with~\eqref{SU2}, implies
\[
 \dist\left(\phi\circ Q^\delta (x), \, [\phi_0, \, +\infty)\right) 
 \leq C \left( \int_{B^2_\delta(x)} \abs{\nabla Q}^2 \, \d\H^2 \right)^{1/2}
 \to 0 \qquad \textrm{uniformly in } x\in A \textrm{ as } \delta\to 0
\]
so, in particular, $\phi_0(Q^\delta, \, A)\to \phi_0(Q, \, A)$ as~$\delta\to 0$.
Then, since the proposition holds for each~$Q^\delta$, by passing to the limit as~$\delta\to 0$ we see that it also holds for~$Q$.
\end{proof}

\subsection{Basic properties of minimizers}
\label{subsect:minimizers}

We conclude the preliminary section by recalling recall some basic facts about minimizers of~\eqref{energy}.

 \begin{lemma} \label{lemma:Linfty}
  Minimizers~$\Qe$ of~\eqref{energy} exist and are of class $C^\infty$ in the interior of~$\Omega$.
  Moreover, for any $U\csubset \Omega$ they satisfy
  \[
   \varepsilon\norm{\nabla Q_\varepsilon}_{L^\infty(U)} \leq C(U) .
  \]
 \end{lemma}
 
 \begin{proof}[Sketch of the proof]
  The existence of minimizers follows by standard method in Calculus of Variations.
  Minimizers solve the Euler-Lagrange system
  \begin{equation} \label{EL}
   -\varepsilon^2\Delta \Qe - a \Qe - b \Qe^2 + \frac b3 \Id \abs{\Qe}^2  + c \abs{\Qe}^2 Q_{\varepsilon} = 0
  \end{equation}
  on $\Omega$, in the sense of distributions. The term $\Id |\Qe^2|$ is a Lagrange multiplier, associated with the tracelessness constraint.
  The elliptic regularity theory, combined with the uniform $L^\infty$-bound of Assumption~\eqref{hp:H}, implies that each component~$Q_{\varepsilon, ij}$ is of class~$C^\infty$ in the interior of the domain.
  The $W^{1,\infty}(U)$-bound follows by interpolation results, see~\cite[Lemma~A.1,~A.2]{BBH-degree_zero}.
 \end{proof}
 
\begin{lemma}[Stress-energy identity] \label{lemma:stress-energy}
 For any~$i\in\{1, \, 2, \ 3\}$, the minimizers satisfy
\[
 \frac{\partial}{\partial x_j}\left(e_\varepsilon(\Qe)\delta_{ij} - \frac{\partial \Qe}{\partial  x_i} \cdot \frac{\partial \Qe}{\partial x_j} \right) = 0 \qquad \textrm{in } \Omega
\]
in the sense of distributions.
\end{lemma}
\begin{proof}
Since~$\Qe$ is of class $C^\infty$ in the interior of the domain by Lemma~\ref{lemma:Linfty}, we can differentiate the products and use the chain rule.
Setting $\partial_i := \partial / \partial x_i$ for the sake of brevity, for each~$i$ we have
\[
 \begin{split}
  \partial_j \big( e_\varepsilon(\Qe)\delta_{ij} & - \partial_i \Qe\cdot\partial_j \Qe \big) \\
  &= \partial_i\partial_k \Qe\cdot \partial_k \Qe + \frac{1}{\varepsilon^2} \frac{\partial f(\Qe)}{\partial Q_{pq}}  \partial_i Q_{\varepsilon, pq} 
  - \partial_i\partial_j \Qe \cdot \partial_j \Qe - \partial_i \Qe \cdot \partial_j\partial_j \Qe \\
  &\stackrel{\eqref{EL}}{=} \partial_k \partial_k \Qe \cdot \partial_i \Qe - \frac{b}{3} \abs{\Qe}^2 \Id \cdot \partial_i \Qe - \partial_i \Qe \cdot \partial_j\partial_j \Qe = 0
 \end{split}
\]
where we have used that~$\Id \cdot \partial_i \Qe = 0$, because~$\Qe$ is traceless.
\end{proof}
 
 \begin{lemma}[Pohozaev identity] \label{lemma:pohozaev}
 Let~$G\subset\Omega$ be a Lipschitz subdomain and $x_0$ a point in~$G$. Then, there holds
 \[
  \begin{split}
   E_\varepsilon(Q_\varepsilon, \, G) &+ \frac12 \int_{\partial G} \nu(x)\cdot(x - x_0) \abs{\frac{\partial Q_\varepsilon}{\partial\nu}}^2 \,\d\H^2 \\
   &= \int_{\partial G} \nu(x)\cdot(x - x_0) e_\varepsilon(Q_\varepsilon) \,\d\H^2 - \int_{\partial G} \left(\nabla Q_\varepsilon\right) \nu(x) \cdot \left(\nabla Q_\varepsilon\right) P_{\partial G} (x - x_0) \,\d\H^2,
  \end{split}
 \]
 where~$\nu(x)$ is the outward normal to~$\partial G$ at~$x$ and~$P_{\partial G}(x - x_0) := (\Id - \nu^{\otimes 2}(x) ) (x - x_0)$ is the tangential component of~$x - x_0$.
\end{lemma}

This identity can be proved arguing exactly as in~\cite[Theorem III.2]{BBH} (the reader is also referred to~\cite[Lemma~2]{MajumdarZarnescu} for more details).
The Pohozaev identity has a very important consequence, which is obtained by taking~$G = B_r(x_0)$ (see e.g.~\cite[Lemma~2]{MajumdarZarnescu} for a proof).

\begin{lemma}[Monotonicity formula] \label{lemma:monotonicity}
 Let $x_0\in\Omega$, and let $0 < r_1 < r_2 < \dist(x_0, \, \partial\Omega)$. Then
 \[
  r_1^{-1} E_\varepsilon(\Qe, \, B_{r_1}(x_0)) \leq r_2^{-1} E_\varepsilon(\Qe, \, B_{r_2}(x_0)) .
 \]
\end{lemma}

Here is another useful consequence of the Pohozaev identity, whose proof is straightforward.
\begin{lemma} \label{lemma:3Dextension}
 Assume that~$G\subseteq\Omega$ is star-shaped, i.e. there exists~$x_0\in G$ such that~$\nu(x)\cdot(x - x_0) \geq 0$ for any~$x\in\partial G$. Then
 \[
  E_\varepsilon(Q_\varepsilon, \, G) \leq 3 \,\mathrm{diam}(G) \, E_\varepsilon(Q_\varepsilon, \, \partial G).
 \]
\end{lemma}

\section{Extension properties}
\label{sect:extension}

\subsection{Extension of~$\S^2$- and~$\NN$-valued maps}
\label{subsect:S2 extension}

In some of our arguments, we will encounter extension problems for~$\NN$-valued maps.
This means, given~$g\colon \partial B^k_r \to \NN$ (for~$k\in\N$, $k\geq 2$ and $r > 0$) we look for a map~$Q\colon B^k_r \to \NN$ satisfying $Q = g$ on~$\partial B^k_r$, with a control on the energy of~$Q$.
When the datum $g$ is smooth enough and no topological obstruction occur, this problem can be reformulated in terms of $\S^2$-valued maps.
Indeed, if the homotopy class of~$g$ is trivial then~$g$ can be \emph{lifted}, 
i.e. there exists a map $\n\colon \partial B^k_r \to \S^2$, as regular as~$g$, such that the diagram
\[
\xymatrix{
  & \S^2 \ar[d]^{\psi} \\
  \partial B^k_r \ar[ur]^{\n} \ar[r]^{g} & \NN}
\]
commutes. Here~$\psi$ is the universal covering map of~$\NN$, given by~\eqref{covering}.
In other words, the function~$\n$ satisfies
\begin{equation} \label{lifting}
 g(x) = (\psi\circ\n) (x) \qquad \textrm{for (almost) every } x\in \partial B^k_r .
\end{equation}
Physically speaking, the vector field~$\n$ determines an orientation for the boundary datum~$g$, therefore a map~$g$ which admits a lifting is said to be \emph{orientable}. 
Since~$\S^2$ is a simply connected manifold, $\S^2$-valued maps are easier to deal with than~$\NN$-valued map, and extension results are known.
Therefore, one can find an $\S^2$-valued extension~$\w$ of~$\n$, then apply~$\psi$ to define an extension~$P:=\psi\circ\w$ of~$g$.
Thus, one proves extension results for $\NN$-valued maps, which will be crucial in the proof of~Proposition~\ref{prop:desiderata}.

\begin{lemma} \label{lemma:extension1}
 There exists a constant~$C > 0$ such that, for any $r > 0$, $k\geq 3$ and any $g\in H^1(\partial B^k_r, \, \NN)$, 
 there exists $P\in H^1(B^k_r, \, \NN)$ which satisfies $P = g$ on~$\partial B^k_r$ and
 \[
  \norm{\nabla P}_{L^2(B^k_r)}^2 \leq C r^{k/2 - 1/2} \norm{\nabla_\top g}_{L^2(\partial B^k_r)} .
 \]
\end{lemma}
 
In Lemma~\ref{lemma:extension1}, the two sides of the inequality have different homogeneities in~$P$,~$g$.
This fact is of main importance, for the arguments of Section~\ref{sect:convergence} rely crucially on it.

In case~$k = 2$, it makes sense to consider the homotopy class of a $H^1$-boundary datum~$g$, due to Sobolev embedding.
If the homotopy class is trivial, we have the following

\begin{lemma} \label{lemma:extension2}
 There exists a constant $C> 0$ such that, for any $r > 0$ and any $g\in H^1(\partial B^2_r, \, \NN)$ with trivial homotopy class,
 there exists $P\in H^1(B^2_r, \, \NN)$ satisfying $P = g$  on~$\partial B^2_r$ and
 \[
  \norm{\nabla P}_{L^2(B^2_r)}^2 \leq C r \norm{\nabla_\top g}_{L^2(\partial B^2_r)}^2.
 \]
\end{lemma}
 
If the homotopy class of the boundary datum~$g$ is non-trivial, then there is no extension~$P\in H^1_g(B^2_r, \, \NN)$.
However, we can still find an extension~$P\in H^1(B^2_r, \, \Sz)$ whose energy satisfies a logarithmic upper bound.
 
\begin{lemma} \label{lemma:extension3}
 There exists a constant $C > 0$ such that, for any $0 < \varepsilon < 1$, $r>0$ and any~$g\in H^1(\partial {B^2_r}, \, \NN)$ with non-trivial homotopy class,
 there exists $P_\varepsilon\in H^1({B^2_r}, \, \Sz)$ such that $P_\varepsilon = g$ ~$\partial B^2_1$ and
 \[
  E_\varepsilon(P_\varepsilon, \, {B^2_r}) \leq \kappa_* \log\frac{r}{\varepsilon} + C \left( r\norm{\nabla_\top g}_{L^2(\partial B^2_r)}^2 + 1\right) .
 \]
\end{lemma}

We also prove an extension results on a cylinder, in dimension three. 
Given positive numbers~$L$ and~$r$, set~$\Lambda:=B^2_r\times [-L, \, L]$ and~$\Gamma:= \partial B^2_r\times [-L, \, L]$.
Let $g\in H^1(\Gamma, \, \NN)$ be a boundary datum, which is only defined on the lateral surface of the cylinder.
By Fubini theorem and Sobolev embedding, the restriction of~$g$ to~$\partial B^2_r\times\{z\}$ has a well-defined homotopy class, for a.e.~$z\in [-L, \, L]$.
Moreover, arguing by density (as in Section~\ref{subsect:JerrardSandier}) we see that this class is indipendent of~$z$.
We call it the homotopy class of~$g$.

\begin{lemma} \label{lemma:extension4}
 For any $0 < \varepsilon < 1$ and any~$g\in H^1(\Gamma, \, \NN)$ with non-trivial homotopy class,
 there exists $P_\varepsilon\in H^1(\Lambda, \, \Sz)$ such that ${P_\varepsilon} = g$ on~$\Gamma$,
 \[
 E_\varepsilon(P_\varepsilon, \, \Lambda) \leq \kappa_* L \log\frac{r}{\varepsilon} + CL\left(\frac{L}{r} + \frac{r}{L}\right) \norm{\nabla_\top g}^2_{L^2(\Gamma)} + CL
 \]
 and
 \[
  E_\varepsilon(P_\varepsilon, \, B^2_r\times\{z\}) \leq \kappa_* \log\frac{r}{\varepsilon} + C\left(\frac{L}{r} + \frac{r}{L}\right) \norm{\nabla_\top g}^2_{L^2(\Gamma)} + C
  \qquad \textrm{for } z\in\{-L, \, L\}.
 \]
\end{lemma}
In both the inequalities, the prefactors of the $H^1$-seminorm of~$g$ are probably not optimal, but the leading order terms are sharp (see Corollary~\ref{cor:lower bound}).

A useful technique to construct extensions of~$\S^2$-valued maps has been proposed by Hardt, Kinderlehrer and Lin~\cite{HKL}.
Their method combines $\R^3$-valued harmonic extensions with an average argument, in order to find a suitable re-projection~$\R^3\to\S^2$.

\begin{lemma}[Hardt, Kinderlehrer and Lin,~\cite{HKL}] \label{lemma:HKLproj}
 For all $\n\in H^1(\partial B^k_r, \, \S^2)$, there exists an extension~$\w\in H^1(B^k_r, \, \S^2)$ which satisfy~$\w_{|\partial B^k_r} = \n$, 
  \begin{equation} \label{harmonic2}
  \norm{\nabla \w}_{L^2(B^k_r)}^2 \leq C_k r^{k/2 - 1/2} \norm{\nabla_\top \n}_{L^2(\partial B^k_r)}
 \end{equation}
 and
  \begin{equation} \label{harmonic1}
  \norm{\nabla \w}_{L^2(B^k_r)}^2 \leq C_k r \norm{\nabla_\top \n}_{L^2(\partial B^k_r)}^2 .
 \end{equation}
\end{lemma}
\begin{proof}[Sketch of the proof]
The existence of an extension~$\w$ which satisfies~\eqref{harmonic2} has been proved by Hardt, Kinderleherer and Lin (see~\cite[proof of Lemma~2.3, Equation~(2.3)]{HKL}).
Although the proof has been given in the case~$k = 3$, a careful reading shows that the same argument applies word by word to any~$k\geq 2$.
The map~$\w$ also satisfies~\eqref{harmonic1}: this follows from~\cite[proof of Lemma~2.3, second and sixth equation at p.~556]{HKL}.
\end{proof}

We state now a lifting property for Sobolev maps.
This subject has been studied extensively, among others, by Bethuel and Zheng~\cite{BethuelZheng}, Bourgain, Brezis and Mironescu~\cite{BourgainBrezisMironescu}, Bethuel and Chiron~\cite{BethuelChiron}, Ball and Zarnescu~\cite{BallZarnescu} 
(in particular, in the latter a problem closely related to the $Q$-tensor theory is considered).

\begin{lemma} \label{lemma:lifting}
 Let $\MM$ be a smooth, simply connected surface (possibly with boundary).
 Then, any map~$g\in H^1(\MM, \, \NN)$ has a lifting, i.e. there exists~$\n\in H^1(\MM, \, \S^2)$ which satisfies~\eqref{lifting}. Moreover,
 \begin{equation} \label{equal grad}
  \abs{\nabla g}^2 = 2s_*^2 \abs{\nabla \n}^2 \qquad \H^2\textrm{-a.e. on } \MM.
 \end{equation}
 If $\MM$ has a boundary then~$\n_{|\partial \MM}$ is a lifting of~$g_{|\partial \MM}$, 
 and if~$g_{|\partial \MM}\in H^1(\partial \MM, \, \NN)$ then~$\n_{|\partial \MM}\in H^1(\partial \MM, \, \S^2)$.
\end{lemma}

\begin{proof}[Sketch of the proof]
The identity~\eqref{equal grad} follows directly by~\eqref{lifting}, by a straightforward computation.
The existence of a lifting is a well-known topological fact, when~$g$ is of class~$C^1$.
In case $g\in H^1$ and~$\MM$ is a bounded, smooth domain in $\R^2$, the existence of a lifting has been proved by Ball and Zarnescu~\cite[Theorem~2]{BallZarnescu}.
The proof, which is based on the density of smooth maps in~$H^1(\MM, \, \NN)$ (see~\cite{SchoenUhlenbeck2}), carries over to more general manifolds~$\MM$.
In case~$\MM$ has a boundary and~$g_{|\partial\MM}\in H^1$, one can adapt the density argument and find a lifting~$\n$ such that~${\n}_{|\partial\MM}\in H^1$. 
If~$\tilde{\n}$ is any other lifting of~$g$, then~$\n\cdot\tilde{\n}$ is an~$H^1$-map~$\MM\to\{1, \, -1\}$ 
and so, by a slicing argument, either~$\tilde{\n} = \n$ a.e. or~$\tilde{\n} = -\n$ a.e (see \cite[Proposition~2]{BallZarnescu}).
In particular, any lifting~$\tilde{\n}$ of~$g$ satisfies~${\tilde{\n}}_{|\partial\MM}\in H^1$.
\end{proof}

Combining Lemmas~\ref{lemma:HKLproj} and~\ref{lemma:lifting}, we obtain easily Lemmas~\ref{lemma:extension1} and~\ref{lemma:extension2}.

 \begin{proof}[Proof of Lemmas~\ref{lemma:extension1} and~\ref{lemma:extension2}]
  Consider Lemma~\ref{lemma:extension1} first.
  Let $\n\in H^1(\partial B^k_r, \, \S^2)$ be a lifting of~$g$, whose existence is guaranteed by Lemma~\ref{lemma:lifting},
  and let $\w\in H^1(B^k_r, \, \S^2)$ be the extension given by Lemma~\ref{lemma:HKLproj}.
  Then, the map defined by
  \[
   P(x) := s_*\left(\w^{\otimes 2}(x) - \frac13 \Id\right) \qquad \textrm{for } \H^k \textrm{-a.e. } x\in B^k_r
  \]
  has the desired properties.
  The proof of Lemma~\ref{lemma:extension2} is analogous.
 \end{proof}

\begin{proof}[Proof of Lemma~\ref{lemma:extension3}]
 By a scaling argument, we can assume WLOG that~$r = 1$.
%
Let~$h(x) := P_0(2x)$ for $x\in\partial B^2_{1/2}$, where~$P_0$ is given by Lemma~\ref{lemma:kappa*}, and let~$w_\varepsilon\colon B^2_{1/2}\to\Sz$ be given by
\[
 w_\varepsilon(x) := \eta_\varepsilon(\abs{x}) h\left(\frac{x}{\abs{x}}\right) \qquad \textrm{for } x\in B^2_{1/2} ,
\]
where
\begin{equation} \label{eta-epsilon}
 \eta_\varepsilon(\rho) := \begin{cases}
                         1 & \textrm{if } \rho \geq \varepsilon \\
                         \varepsilon^{-1} \rho & \textrm{if } 0 \leq \rho < \varepsilon .
                        \end{cases}
\end{equation}
Then, $w_\varepsilon$ belongs to $H^1(B^2_{1/2}, \, \Sz)$ and
\begin{equation} \label{ext h 1}
 E_\varepsilon({w_\varepsilon}, \, B^2_{1/2}) \leq \kappa_* \abs{\log\varepsilon} + C.
\end{equation}
Indeed,
\[
 \abs{\nabla w_\varepsilon}^2 = \abs{\frac{\d w_\varepsilon}{\d \rho}}^2 + \frac{1}{\rho^2} \abs{\nabla_\top w_\varepsilon}^2 
 \begin{cases}
       \leq C \varepsilon^{-1}           & \textrm{where } \rho \leq \varepsilon  \\
       = {\rho^{-2}}{\abs{\nabla_\top h}^2} & \textrm{where } \rho \geq \varepsilon ,
 \end{cases}
\]
and~$w_\varepsilon(x)\in\NN$ if $|x|\geq \varepsilon$. Therefore, we have
\[
 \begin{split}
  E_\varepsilon(w_\varepsilon, \, B^2_{1/2}) &\leq \frac12 \int_\varepsilon^{1/2} \frac{\d \rho}{\rho} \int_{\S^1} \abs{\nabla_\top h}^2 \, \d\H^1 + E_\varepsilon(w_\varepsilon, \, B^2_{\varepsilon}) \\
      &\leq \frac12 \left( \abs{\log\varepsilon} - \log 2\right) \int_{\S^1} \abs{\nabla_\top h}^2 \, \d\H^1 + C,
 \end{split}
\]
whence~\eqref{ext h 1} follows.

To complete the proof of the lemma, we only need to interpolate between~$g$ and~$h$ by a function defined on the annulus~$D := {B^2_1}\setminus {B^2_{1/2}}$.
Up to a bilipschitz equivalence,~$D$ can be thought as the unit square~$(0, \, 1)^2$ 
with an equivalence relation identifying two opposite sides of the boundary, as shown in Figure~\ref{fig:anelloquadrato}.
We assign the boundary datum~$g$ on the bottom side, and~$h$ on the top side.
Since $\NN$ is path-connected, we find a smooth path $c\colon [0, \, 1]\to \NN$ connecting $g(0, \, 0)$ to $h(0, \, 1)$.
By assigning $c$ as a boundary datum on the lateral sides of the square, we have defined an $H^1$-map $\partial [0, \, 1]^2\to \NN$, homotopic to~$g * c * h * \tilde c$.
(Here, the symbol~$*$ stands for composition of paths, and $\tilde c$ is the reverse path of $c$.)
Since the square is bilipschitz equivalent to a disk, it is possible to apply Lemma~\ref{lemma:extension2} and find $\tilde v\in H^1([0, \, 1]^2, \, \NN)$ such that
\begin{equation} \label{ext h 2}
 \int_{[0, \, 1]^2} \abs{\nabla {\tilde v}}^2 \, \d \H^2 \leq C \left( \norm{\nabla g}^2_{L^2(\partial {B^2_1})} + \norm{\nabla h}^2_{L^2(\partial B^2_{1/2})} + \norm{c^\prime}^2_{L^2(0, \, 1)} \right) .
\end{equation}
Passing to the quotient $[0, \, 1]^2 \to D$, we obtain a map $v\in H^1(D, \, \Sz)$.
Now, the function~$P\colon B^2_1\to\Sz$ defined by~$P := v$ on~$D$ and~$P := w_\varepsilon$ on~$B^2_{1/2}$ satisfies the lemma.
Indeed, the energy of~$P$ is bounded by~\eqref{ext h 1} and~\eqref{ext h 2}, and the $H^1$-norms of~$h$ and~$c$ are controlled by a constant depending only on~$\NN$.
\end{proof}

\begin{figure}[t] 
 \centering
 \includegraphics[height =.26\textheight, keepaspectratio = true]{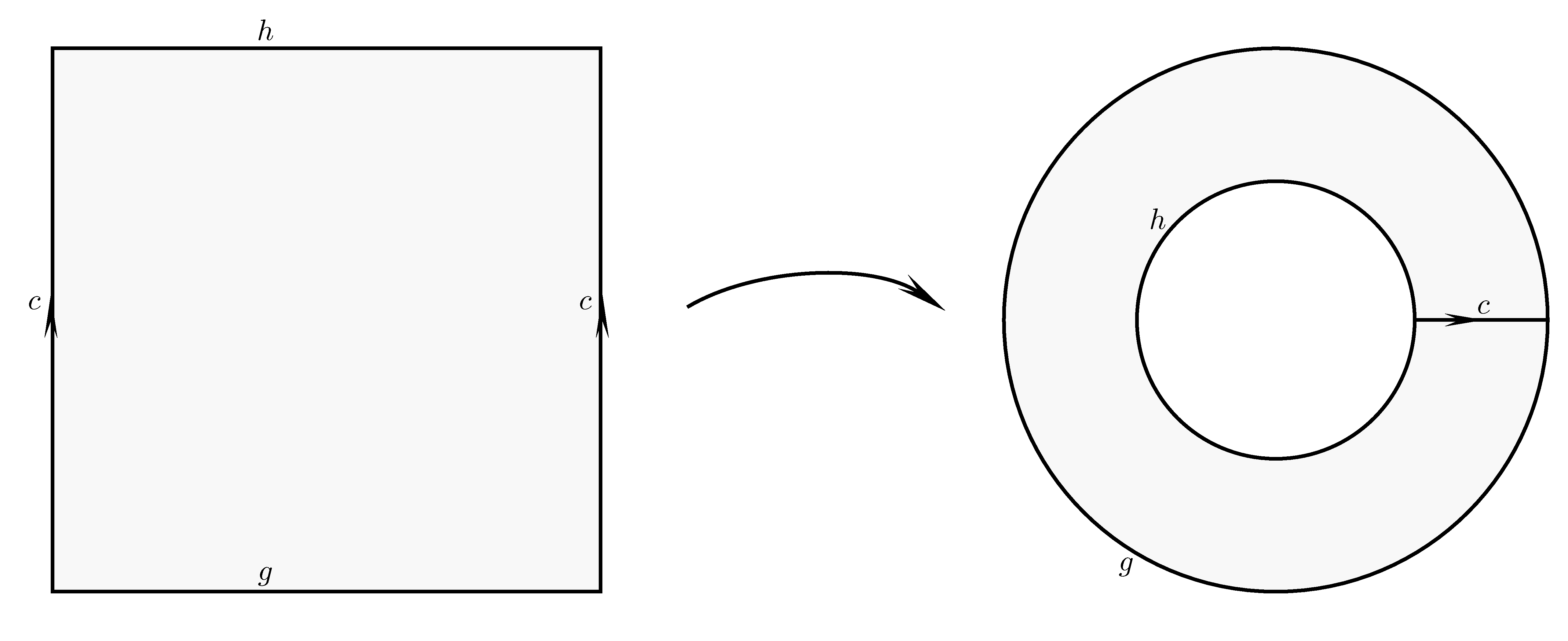}
 \caption{A square can be mapped into an annulus, by identifying a pair of opposite sides.}
 \label{fig:anelloquadrato}
\end{figure}

\begin{proof}[Proof of Lemma~\ref{lemma:extension4}]
 By an average argument, we find~$z_0\in [-L/4, \, L/4]$ such that
 \begin{equation} \label{extension4,1}
  \norm{\nabla_\top g}^2_{L^2(\partial B^2_r\times\{z_0\})} \leq \frac{8}{L} \norm{\nabla_\top g}^2_{L^2(\Gamma)}.
 \end{equation}
 To avoid notation, we assume WLOG that~$z_0 = 0$. 
 We construct an $\NN$-valued map~$\tilde P$, defined over the ``cylindrical annulus'' $(B^2_r\setminus B^2_{r/2})\times[-L, \, L]$, such that 
 \[
  \tilde P \left(\rho e^{i\theta}, \, z\right) = \begin{cases}
                                                  g(r e^{i\theta}, \, z) \qquad \textrm{for } \rho = r \\
                                                  g(r e^{i\theta}, \, 0) \qquad \textrm{for } \rho = r/2,
                                                 \end{cases}
 \]
 where~$(\rho, \, \theta, \, z)\in [0, \, r]\times[0, \, 2\pi]\times [-L, \, L]$ are the cylindrical coordinates.
 Then, since~$\tilde P$ restricted to~$\partial B^2_{r/2}\times[-L, \, L]$ is independent of the $z$-variable,
 we can apply Lemma~\ref{lemma:extension3} to construct an extension in the inner cylinder~$B^2_{r/2}\times[-L, \, L]$.
 
 The map~$\tilde P$ is defined as follows:
 \[
  \tilde P(\rho e^{i\theta}, \, z) := \begin{cases}
                                       g\left(re^{i\theta}, \, z^\prime(\rho, \, z)\right) & \textrm{if } \rho_0(\rho, \, z) \leq \rho \leq r \textrm{ and } |z| \leq L \\
                                       g(re^{i\theta}, \, 0) & \textrm{if } r/2 \leq \rho <\rho_0(\rho, \, z) \textrm{ and } |z| \leq L,
                                      \end{cases}
 \]
 where
 \[
  z^\prime(\rho, \, z) := \dfrac{2L}{r}\sign(z)(\rho - r) + z, \qquad \rho_0(\rho, \, z) := r - \dfrac{r}{2L}|z|.
 \]
 Note that~$|z^\prime(\rho, \, z)| \leq L$ if~$\rho_0(\rho, \, z) \leq \rho \leq r$ and~$|z|\leq L$.
 We compute the $H^1$-seminorm of~$\tilde P$. For simplicity, we restrict our attention to the upper half-cylinder. We have
 \[
 \begin{split}
   &\|\nabla\tilde P\|^2_{L^2((B^2_r\setminus B^2_{r/2})\times [0, \, L])} \\
   &\qquad =\int_0^{2\pi} \int_0^L\int_{\rho_0(\rho, \, z)}^r \left\{ \left(\frac{4L^2}{r^2}+1\right) \rho\abs{\partial_z g}^2\left(re^{i\theta}, \, z^\prime(\rho, \, z)\right)
   + \rho^{-1}\abs{\partial_\theta g}^2\left(re^{i\theta}, \, z^\prime(\rho, \, z)\right) \right\} \, \d\rho \, \d z \, \d\theta \\
   &\qquad\qquad\qquad + \int_0^{2\pi} \int_0^L\int_{r/2}^{\rho_0(\rho, \, z)} \rho^{-1}\abs{\partial_\theta g}^2 \left(re^{i\theta}, \, 0\right)\, \d\rho \, \d z \, \d\theta \\
   &\qquad \leq \frac{r}{2L} \int_0^{2\pi} \int_0^L \int_{0}^z \left\{ \left(\frac{4L^2}{r^2}+1\right) r \abs{\partial_z g}^2
   + 2r^{-1}\abs{\partial_\theta g}^2\right\} (re^{i\theta}, \, \xi)  \, \d\xi \, \d z \, \d\theta \\
   &\qquad\qquad\qquad + (\log 2)L \int_0^{2\pi} \abs{\partial_\theta g}^2 \left(re^{i\theta}, \, 0\right)\, \d\theta \\
   &\qquad \leq r \left(\frac{4L^2}{r^2}+1\right) \int_0^{2\pi} \int_{0}^L \left\{ r\abs{\partial_z g}^2
   + r^{-1}\abs{\partial_\theta g}^2\right\} (re^{i\theta}, \, \xi)  \, \d\xi \, \d\theta
   + (\log 2)L \int_0^{2\pi} \abs{\partial_\theta g}^2 \left(re^{i\theta}, \, 0\right)\, \d\theta.
 \end{split}
 \]
 An analogous estimates holds on the lower half-cylinder. Therefore,
 \[
  \|\nabla\tilde P\|^2_{L^2((B^2_r\setminus B^2_{r/2})\times [-L, \, L])} \leq \left(\frac{4L^2}{r} + r\right) \norm{\nabla_\top g}^2_{L^2(\Gamma)}
  + (\log 2)r L \norm{\nabla_\top g}^2_{L^2(\partial B^2_r \times\{0\})}
 \]
 and so, due to~\eqref{extension4,1}, we have
 \begin{equation} \label{extension4,2}
  \|\nabla\tilde P\|^2_{L^2((B^2_r\setminus B^2_{r/2})\times [-L, \, L])} \leq C \left(\frac{L^2}{r} + r\right) \norm{\nabla_\top g}^2_{L^2(\Gamma)}.
 \end{equation}
 By applying Lemma~\ref{lemma:extension3} (and~\eqref{extension4,1}) to~$g(\cdot, \, 0)$, we find an extension~$\tilde P_\varepsilon\in H^1(B^2_{r/2}, \, \Sz)$ which satisfies
 \begin{equation} \label{extension4,3}
  E_\varepsilon(\tilde P_\varepsilon, \, B^2_{r/2}) \leq \kappa_*\log\frac{r}{\varepsilon} + \frac{Cr}{L} \norm{\nabla_\top g}^2_{L^2(\Gamma)} + C .
 \end{equation}
 Define the map~$P$ by letting $P(\rho e^{i\theta}, \, z) := \tilde P(\rho e^{i\theta}, \, z)$ if $r/2 < \rho \leq r$
 and~$P_\varepsilon(\rho e^{i\theta}, \, z) := \tilde P_\varepsilon(\rho e^{i\theta})$ if $\rho\leq r/2$.
 By integrating~\eqref{extension4,2} with respect to~$z\in[-L, \, L]$, and combining the resulting inequality with~\eqref{extension4,2}, 
 we give an upper bound for the energy of~$P_\varepsilon$ on~$\Lambda$.
 Moreover, there holds
 \[
 \begin{split}
   &\|\nabla\tilde P\|^2_{L^2((B^2_r\setminus B^2_{r/2})\times\{L\})} \\
   &\qquad =\int_0^{2\pi}\int_{r/2}^r \left\{ \left(\frac{4L^2}{r^2}+1\right) \rho\abs{\partial_z g}^2\left(re^{i\theta}, \, z^\prime(\rho, \, L)\right)
   + \rho^{-1}\abs{\partial_\theta g}^2\left(re^{i\theta}, \, z^\prime(\rho, \, L)\right) \right\} \, \d\rho \, \d\theta \\
   &\qquad \leq \frac{r}{2L} \int_0^{2\pi} \int_{0}^L \left\{ \left(\frac{4L^2}{r^2}+1\right) r \abs{\partial_z g}^2
   + 2r^{-1}\abs{\partial_\theta g}^2 \right\} (re^{i\theta}, \, \xi) \, \d\xi \, \d\theta \\
   &\qquad \leq \frac{r}{L} \left(\frac{4L^2}{r^2}+1\right) \norm{\nabla_\top g}^2_{L^2(\Gamma)} .
 \end{split}
 \]
 This inequality combined with~\eqref{extension4,3}, gives an upper bound for the energy of~$P_\varepsilon$ on~$B^2_r\times\{L\}$;
 a similar inequality holds on~$B^2_r\times\{-L\}$.
 This concludes the proof.
\end{proof}

\subsection{Luckhaus' lemma and its variants}
\label{subsect:luckhaus}

When dealing with the asymptotic analysis for minimizers~$\Qe$ of~\eqref{energy}, we will be confronted with the following issue.
Suppose that~$r > 0$ and~$B_r\subseteq\Omega$. Given a map~$P_\varepsilon\colon B_r\to\Sz$, we wish to compare the energy of~$P_\varepsilon$ with the energy of a minimizer~$Q_\varepsilon$.
However, it may happen that~$P_\varepsilon\neq\Qe$ on~$\partial B_r$, so $P_\varepsilon$ is not an admissible comparison map.
To correct this, we need to construct a function which interpolates between $P_\varepsilon$ and~$\Qe$ over a thin spherical shell.

In general terms, the problem may be stated as follows.
Fix a parameter~$0 < \epsilon < 1$ (in the application, $\epsilon$ will depend on both~$\varepsilon$ and~$r$).
Let~$u_\epsilon$,~$v_\epsilon \colon \partial B_1 \to \Sz$ be two given maps in~$H^1$.
We look for a spherical shell~$A_\epsilon := B_1 \setminus B_{1 - h(\epsilon)}$ of (small) thickness~$h(\epsilon) > 0$ and a map $\varphi_\epsilon\colon A_\epsilon \to \Sz$, such that
\begin{equation} \label{phi bd value}
 \varphi_\epsilon(x) = u_\epsilon(x) \quad \textrm{and} \quad \varphi_\epsilon(x - h(\epsilon)x) = v_\epsilon(x) \qquad \textrm{for }\H^2\textrm{-a.e. } x\in \partial B_1 
\end{equation}
and the energy~$E_\varepsilon(\varphi_\epsilon, \, A_\epsilon)$ satisfies a suitable bound.
Moreover, in some circumstances only the function~$u_\epsilon$ is prescribed,
and we need to find \emph{both} a map~$v_\epsilon\colon \partial B_1\to \NN$ 
and the interpolating function~$\varphi_\epsilon$.

Luckhaus proved an interesting interpolation lemma (see \cite[Lemma~1]{Luckhaus-PartialReg}), which turned out to be useful in several applications.
When the two maps $u_\epsilon$, $v_\epsilon$ take values in the manifold~$\NN$, Luckhaus' lemma gives an extension~$\varphi_\epsilon$ satisfying~\eqref{phi bd value} and
\[
 \sup_{x\in A_\epsilon}\dist(\varphi_\epsilon(x), \, \NN) + \int_{A_\epsilon} \abs{\nabla \varphi_\epsilon}^2
 \leq C(u_\epsilon, \, v_\epsilon, \, h(\epsilon)) .
\]
For the convenience of the reader, and for future reference, we recall Luckhaus' lemma.
Since the potential~$\epsilon^{-2}f$ is not taken into account here, we drop the subscript~$\epsilon$ in the notation.

\begin{lemma}[Luckhaus,~\cite{Luckhaus-PartialReg}] \label{lemma:interpolation1}
 For any $\beta\in(1/2, \, 1)$, there exists a constant $C > 0$ with this property. 
 For any fixed numbers $0 < \lambda \leq 1/2$, $0 < \sigma< 1$ and any~$u, \, v\in H^1(\partial B_1, \, \NN)$, set
 \begin{equation*} 
  K := \int_{\partial B_1} \left\{ \abs{\nabla u}^2 + \abs{\nabla v}^2 + \frac{\abs{u - v}^2}{\sigma^2}\right\} \,\d\H^2.
 \end{equation*}
 Then, there exists a function $\varphi\in H^1(B_1 \setminus B_{1 - \lambda}, \, \Sz)$ satisfying~\eqref{phi bd value},
 \begin{equation*} 
 \dist(\varphi(x), \, \NN ) \leq C \sigma^{1 - \beta} \lambda^{-1/2} K^{1/2}
 \end{equation*}
 for a.e. $x\in B_1\setminus B_{1 - \lambda}$ and
 \begin{equation*}
  \int_{B_1 \setminus B_{1 - \lambda}} \abs{\nabla \varphi}^2 \leq C \lambda \left(1 + \sigma^2 \lambda^{-2} \right) K .
 \end{equation*}
\end{lemma}

The idea of the proof is illustrated in Figure~\ref{fig:luckhaus}.
One constructs a grid on the sphere~$\partial B_1$ with suitable properties.
The map~$\varphi$ is defined by linear interpolation between~$u$ and~$v$ on the boundary of the cells.
Inside each cell, $\varphi$ is defined by a homogeneous extension.
By choosing carefully the grid on~$\partial B_1$, and using Sobolev embeddings, 
one can bound the $L^\infty$-distance between $u$ and~$v$ on the boundary of the cells, in terms of~$K$.
This yields bounds both on~$\dist(\varphi(x), \, \NN)$ and on the gradient of~$\varphi$.

\begin{figure}[bt] 
\centering
 \parbox{.40\textwidth}{
  \includegraphics[height=.28\textheight]{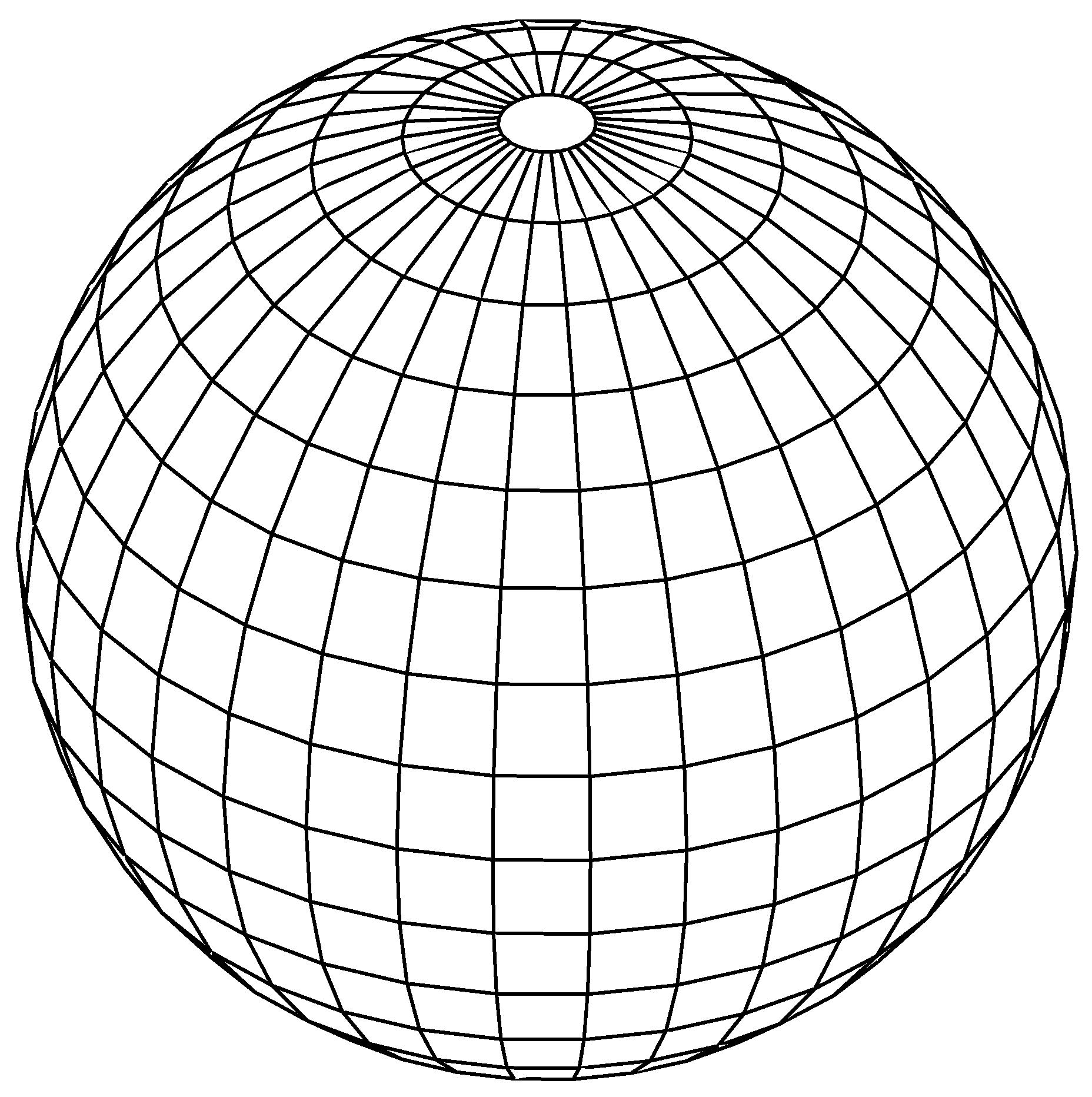}
  }%
 \qquad
 \begin{minipage}{.47\textwidth}%
  \includegraphics[height=.28\textheight]{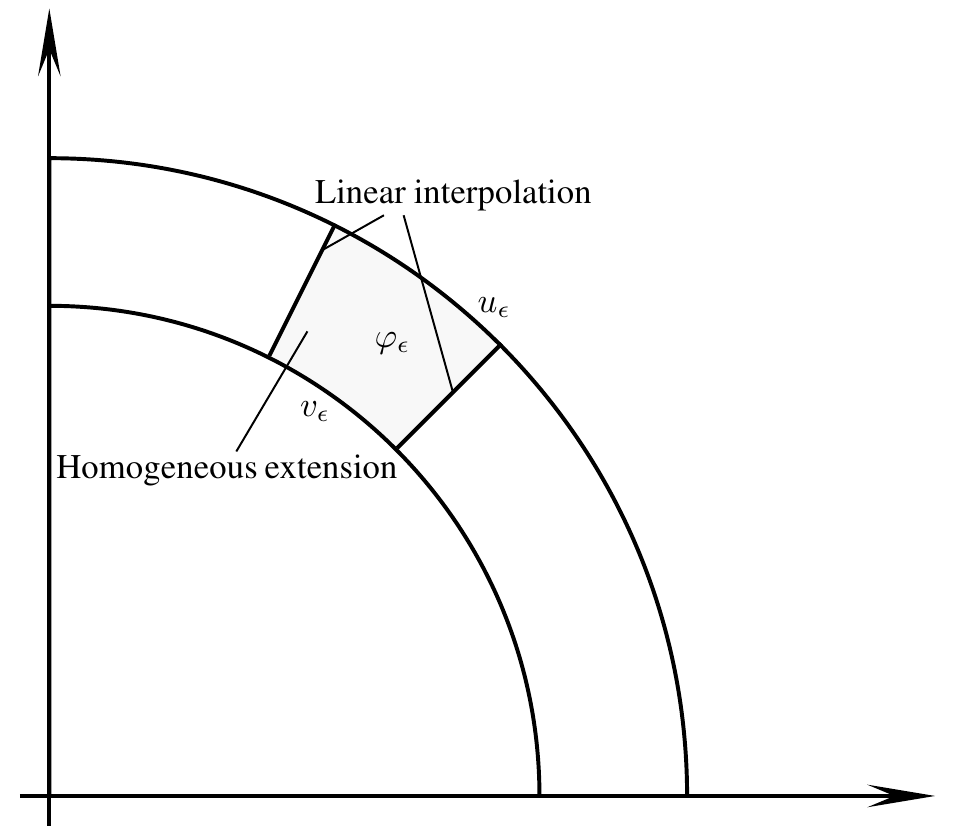}
 \end{minipage}%
        \caption{Left: a grid on a sphere. Right: the Luckhaus' construction.
        Given two maps~$u_\epsilon$, $v_\epsilon$ (respectively defined on the outer and inner boundary of a thin spherical shell), we construct a map~$\varphi_\epsilon$
        by using linear interpolation on the boundary of the cells, and homogeneous extension inside each cell.}
        \label{fig:luckhaus}
\end{figure}

We will discuss here a couple of variants of this lemma.
In our first result, we suppose that only the map $u_\epsilon\colon \partial B_1 \to \Sz$ is prescribed, so we need to find both $v_\epsilon\colon \partial B_1\to \NN$ and~$\varphi_\epsilon$.
Approximating $u_\epsilon$ with a $\NN$-valued map~$v_\epsilon$ may be impossible, due to topological obstructions.
However, this is possible if the energy of~$u_\epsilon$ is small, compared to~$|\log\epsilon|$.
More precisely, we assume that
\begin{equation} \label{hp_interpolation2:1}
  E_\epsilon(u_\epsilon, \, \partial B_1) \leq \eta_0 \abs{\log\epsilon}
\end{equation}
for some small constant $\eta_0 > 0$. 
For technical reasons, we also require a $L^\infty$-bound on~$u_\epsilon$, namely
\begin{equation} \label{hp_interpolation2:2}
  \norm{u_\epsilon}_{L^\infty(\partial B_1)} \leq \kappa 
\end{equation}
where~$\kappa$ is an~$\epsilon$-independent constant.
In the applications, $u_\epsilon$ will be a Landau-de Gennes minimizer and~\eqref{hp_interpolation2:2} will be satisfied, because of~\eqref{hp:H}.

\begin{prop} \label{prop:interpolation2}
 For any $\kappa > 0$, there exist positive numbers $\eta_0$, $\epsilon_1$, $C$ with the following property.
 For any $0 < \eta \leq \eta_0$, any~$0 < \epsilon \leq \epsilon_1$ and any~$u_\epsilon\in (H^1\cap L^\infty)(\partial B_1, \, \Sz)$ satisfying \eqref{hp_interpolation2:1}--\eqref{hp_interpolation2:2},
 there exist maps~$v_\epsilon\in H^1(\partial B_1, \, \NN)$ and $\varphi_\epsilon\in H^1(B_1 \setminus B_{1 - h(\epsilon)}, \, \Sz)$ which satisfy~\eqref{phi bd value},
 \begin{gather}
  \frac12 \int_{\partial B_1} \abs{\nabla v_\epsilon}^2 \, \d \H^2 \leq C E_\epsilon(u_\epsilon, \, \partial B_1), \label{v energy} \\
  E_\epsilon(\varphi_\epsilon, \, B_1 \setminus B_{1 - h(\epsilon)}) \leq C h(\epsilon) E_\epsilon(u_\epsilon, \, \partial B_1) \label{phi energy}
 \end{gather}
 for~$h(\epsilon) := \epsilon^{1/2} |\log\epsilon|$.
\end{prop}

We will discuss the proof of this proposition later on.
Before that, we remark that~$v_\epsilon$ effectively approximates~$u_\epsilon$, i.e. their distance --- measured in a suitable norm --- tends to $0$ as $\epsilon\to 0$.

\begin{cor} \label{cor:interpolation2}
Under the same assumptions of Proposition~\ref{prop:interpolation2}, there holds
\[
 \norm{u_\epsilon - v_\epsilon}_{L^2(\partial B_1)} \leq C h^{1/2}(\epsilon) E^{1/2}_\epsilon(u_\epsilon, \, \partial B_1).
\]
\end{cor}
Notice that the right-had side tends to~$0$ as~$\epsilon\to 0$, due to~\eqref{hp_interpolation2:1} and the choice of~$h(\epsilon)$.
\begin{proof}
 We can estimate the $L^2$-distance between~$u_n$ and~$v_n$ thanks to~\eqref{phi bd value}:
 \[
 \begin{split}
  \norm{u_\epsilon - v_\epsilon}_{L^2(\partial B_1)}^2 &\stackrel{\eqref{phi bd value}}{=} \int_{\partial B_1} \abs{\varphi_\epsilon(x) - \varphi_\epsilon(x - h(\epsilon)x)}^2 \, \d \H^2(x) \\
   &= \int_{\partial B_1} \abs{\int_{1 - h(\epsilon)}^1 \nabla\varphi_\epsilon(tx)x\, \d t \,}^2 \, \d \H^2(x) .
  \end{split}
 \]
 Then, by H\"older inequality,
 \[
 \begin{split}
  \norm{u_\epsilon - v_\epsilon}_{L^2(\partial B_1)}^2 &\leq h(\epsilon) \int_{\partial B_1}\int_{1 - h(\epsilon)}^1 \abs{ \nabla\varphi_\epsilon(tx)}^2 \, \d t \, \d \H^2(x) \\
   &\leq \frac{h(\epsilon)}{(1 - h(\epsilon))^2 } E_{\epsilon}(\varphi_\epsilon, \, B_1 \setminus B_{1 - h(\epsilon)})
   \stackrel{\eqref{phi energy}}{\leq} C h(\epsilon) E_{\epsilon}(u_\epsilon, \, \partial B_1) . \qedhere
 \end{split}
 \]
\end{proof}

Combining Lemma~\ref{lemma:interpolation1} and Proposition~\ref{prop:interpolation2}, we obtain a third extension result.
In this case, both the boundary values~$u$, $v$ are prescribed and, unlike Luckhaus' lemma, we provide a control over the potential energy of the extension~$\epsilon^{-2}f(\varphi_\epsilon)$.

\begin{prop} \label{prop:interpolation3}
 Let $\{\sigma_\epsilon\}_{\epsilon >0}$ be a positive sequence such that $\sigma_\epsilon\to 0$, and let $u_\epsilon$, $v_\epsilon$ be given functions in~$H^1(\partial B_1, \, \Sz)$.
 For all $\epsilon >0$, assume that $u_\epsilon$ satisfies~\eqref{hp_interpolation2:2}, that $v_\epsilon(x)\in \NN$ for $\H^2$-a.e.~$x\in \partial B_1$ and that
 \begin{equation} \label{K}
  \int_{\partial B_1} \left\{ \abs{\nabla u_\epsilon}^2 + \frac{1}{\epsilon^2} f(u_\epsilon) + \abs{\nabla v_\epsilon}^2 + \frac{\abs{u_\epsilon - v_\epsilon}^2}{\sigma_\epsilon^2}\right\} \,\d\H^2 \leq C
 \end{equation}
 for an $\epsilon$-independent constant~$C$. Set
 \begin{equation*}
  \nu_\epsilon := h(\epsilon) + \left(h^{1/2}(\epsilon) + \sigma_\epsilon\right)^{1/4} \left(1 - h(\epsilon)\right) .
 \end{equation*}
 Then, there exist a number~$\epsilon_1> 0$ and, for $0 < \epsilon \leq \epsilon_1$, a function~$\varphi_\epsilon\in H^1(B_1 \setminus B_{1 - \nu_\epsilon}, \, \Sz)$ which satisfies~\eqref{phi bd value} and
 \begin{equation*}
  E(\varphi_\epsilon, \, B_1 \setminus B_{1 - \nu_\epsilon}) \leq C \nu_\epsilon .
 \end{equation*}
\end{prop}
The assumption~\eqref{K} could be relaxed by requiring just a logarithmic bound, of the order of~$\eta_0|\log\epsilon|$ for small $\eta_0 > 0$, with additional assumptions on~$\sigma_\epsilon$.
However, the result as it is presented here suffices for our purposes.
\begin{proof}[Proof of Proposition~\ref{prop:interpolation3}]
Thanks to~\eqref{K} and~\eqref{hp_interpolation2:2},
we can apply Proposition~\ref{prop:interpolation2} to the function~$u_\epsilon$.
We obtain two maps~$w_\epsilon\in H^1(\partial B_1, \, \NN)$ and~$\varphi_\epsilon^1\in H^1(B_1 \setminus B_{1 - h(\epsilon)}, \, \Sz)$, which satisfy
\begin{gather}
 \varphi_\epsilon^1(x) = u_\epsilon(x) \quad \textrm{and} \quad \varphi_\epsilon^1(x - h(\epsilon)x) = w_\epsilon(x) \qquad \textrm{for }\H^2\textrm{-a.e. } x\in \partial B_1 , \nonumber \\
 \int_{\partial B_1} \abs{\nabla w_\epsilon}^2 \, \d \H^2 \leq C , \nonumber \\
 E_\epsilon(\varphi_\epsilon^1, \, B_1 \setminus B_{1 - h(\epsilon)}) \leq C h(\epsilon). \label{interp3,1}
\end{gather}
Corollary~\ref{cor:interpolation2}, combined with~\eqref{K}, entails
\[
 \norm{w_\epsilon - v_\epsilon}_{L^2(\partial B_1)} \leq \norm{w_\epsilon - u_\epsilon}_{L^2(\partial B_1)} + \norm{u_\epsilon - v_\epsilon}_{L^2(\partial B_1)} 
 \leq C \left( h^{1/2}(\epsilon) + \sigma_\epsilon\right) .
\]
Therefore, setting~$\tilde\sigma_\epsilon :=  h^{1/2}(\epsilon) + \sigma_\epsilon$, we have
\[
 \int_{\partial B_1} \left\{\abs{\nabla w_\epsilon}^2 + \abs{\nabla v_\epsilon}^2 + \frac{\abs{w_\epsilon - v_\epsilon}^2}{\tilde\sigma_\epsilon^2}\right\} \,\d\H^2 \leq C
\]
Then, we can apply Lemma~\ref{lemma:interpolation1} to~$v_\epsilon$ and~$w_\epsilon$, choosing $\sigma= \tilde \sigma_\epsilon$, $\beta = 3/4$ and $\lambda := \tilde\sigma_\epsilon^{1/4}$.
By rescaling, we find a map~$\varphi_\epsilon^2\in H^1(B_{1 - h(\epsilon)} \setminus B_{\nu_\epsilon}, \, \Sz)$ which satisfies
\begin{gather}
 \int_{B_{1 - h(\epsilon)} \setminus B_{\nu_\epsilon}} \abs{\nabla \varphi_\epsilon^2}^2 \leq C \tilde\sigma_\epsilon^{1/4} \left(1 - h(\epsilon) \right) \nonumber \\
 \dist(\varphi_\epsilon^2(x), \, \NN) \leq C \tilde\sigma_\epsilon^{1/8} \qquad \textrm{for all } x\in B_{1 - h(\epsilon)} \setminus B_{\nu_\epsilon} . \label{interp3,2}
\end{gather}
Since $\tilde\sigma_\epsilon\to 0$, there exists $\epsilon_1>0$ such that $\varphi_\epsilon^2(x)\notin\Cs$ for any~$0 < \epsilon \leq \epsilon_1$ and~$x$.
Therefore, the function
\[
 \varphi_\epsilon(x) := \begin{cases}
                         \varphi_\epsilon^1(x)      & \textrm{if } x\in B_1 \setminus B_{1 - h(\epsilon)} \\
                         \RR\circ\varphi_\epsilon^2(x) & \textrm{if } x\in B_{1 - h(\epsilon)} \setminus B_{\nu_\epsilon} 
                        \end{cases}
\]
is well-defined, belongs to~$H^1(B_1\setminus B_{\nu_\epsilon}, \, \NN)$, satisfies~\eqref{phi bd value} and
\[
 E_\epsilon(\varphi_\epsilon, \, B_1\setminus B_{\nu_\epsilon}) = E_\epsilon(\varphi_\epsilon^1, \, B_1\setminus B_{1 - h(\epsilon)}) +
 \int_{B_{1 - h(\epsilon)} \setminus B_{\nu_\epsilon}} \abs{\nabla \varphi_\epsilon^2}^2 \stackrel{\eqref{interp3,1}\textrm{--}\eqref{interp3,2}} \leq C\nu_\epsilon . \qedhere
\]
\end{proof}

Subsections \ref{subsect:grids}--\ref{subsect:eta implies C} are devoted to the proof of Proposition~\ref{prop:interpolation2}, which we sketch here.
From now on, we assume that there exists a positive constant~$M$ such that
 \begin{equation*} \label{u energy} \tag{M$_\epsilon$}
  E_\epsilon(u_\epsilon, \, \partial B_1) \leq M \abs{\log\epsilon} \qquad \textrm{for all } 0 < \epsilon < 1.
 \end{equation*}
The assumption~\eqref{hp_interpolation2:1} clearly implies~\eqref{u energy}.
As in Luckhaus' arguments, the key ingredient of the construction is the choice of a grid on the unit sphere~$\partial B_1$, with special properties.
In Subsection~\ref{subsect:grids} we construct a family of grids~$\{\mathscr G^\epsilon\}$, whose cells have size controlled by~$h(\epsilon)$,
and we prove that there exists $\epsilon_1> 0$ such that
 \begin{equation*}
  \dist(u_\epsilon(x), \, \NN) \leq \delta_0 \qquad \textrm{for any } \epsilon\in (0, \, \epsilon_1) \textrm{ and any } x\in R^\epsilon_1 .
 \end{equation*}
Here $R^\epsilon_1$ denotes the $1$-skeleton of~$\mathscr G^\epsilon$, i.e. the union of the boundaries of all the cells,
and~$\delta_0$ is given by Lemma~\ref{lemma:f below}.
In particular, 
the composition~$\RR\circ u_\epsilon$ is well-defined on~$R^\epsilon_1$ when $\epsilon < \epsilon_1$.
We wish to extend~$\RR\circ u_\epsilon$ to a map~$v_\epsilon\colon \partial B_1 \to \NN$.
This may be impossible, depending on the homotopy properties of~$u_\epsilon$.
A sufficient condition for the existence of~$v_\epsilon$ is the following:
 \begin{enumerate}[label=\textup{ }{(C\textsubscript{$\epsilon$})}\textup{ },ref={C\textsubscript{$\epsilon$}}]
  \item \label{approximable}
   For any $2$-cell $K$ of~$\mathscr G^\epsilon$, the loop ${\RR\circ u_\epsilon}_{|\partial K}\colon \partial K \to \NN$ is homotopically trivial.
  \end{enumerate}
This condition makes sense for any $u_\epsilon\in H^1(\partial B_1, \, \Sz)$,
for we construct~$\mathscr G^\epsilon$ in such a way that~$u_\epsilon$ restricted to~$R^\epsilon_1$ belongs to~$H^1\hookrightarrow C^0$.
  
In Subsection~\ref{subsect:v_eps}, we assume that \eqref{u energy} and \eqref{approximable} hold and we construct a function~$v_\epsilon\in H^1(\partial B_1, \, \NN)$, 
whose energy is controlled by the energy of~$u_\epsilon$.
Basically, we extend~${\RR\circ u_\epsilon}_{|\partial K}$ inside every $2$-cell $K\in\mathscr G^\epsilon$, which is possible by Condition~\eqref{approximable}.
Once~$v_\epsilon$ is known, we construct~$\varphi_\epsilon$ by Luckhaus' method. 
Particular care must be taken here, as we need to bound the potential energy of~$\varphi_\epsilon$ as well.

Finally, in Subsection~\ref{subsect:eta implies C} we show that the logarithmic bound~\eqref{hp_interpolation2:1}, for a small enough constant $\eta_0$, implies that Condition~\eqref{approximable} is satisfied.
Arguing by contra-position, we assume that~\eqref{approximable} is not satisfied. 
Then,~${\RR\circ u_\epsilon}_{|\partial K}$ is non-trivial for at least one $2$-cell $K\in\mathscr G^\epsilon$.
In this case, using Jerrard-Sandier type lower bounds, we prove that the energy~$E_\epsilon(u_\epsilon, \, \partial B_1)$ blows up at least as~$\eta_1 |\log\epsilon|$ for some~$\eta_1 >0$.
Taking $\eta_0 < \eta_1$, we have a contradiction because of~\eqref{hp_interpolation2:1}, so the proof is complete.
  

\subsection{Good grids on the sphere}
\label{subsect:grids}

Consider a decomposition of $\partial B_1$ of the form
 \[
  \partial B_1 = \bigcup_{j = 0}^{2} \bigcup_{i = 1}^{k_j} K_{i, j} ,
 \]
where the sets~$K_{i, j}$ are mutually disjoint, and each~$K_{i, j}$ is bilipschitz equivalent to a $j$-di\-men\-sio\-nal ball.
The collection of all the $K_{i, j}$'s will be called a \emph{grid} on $\partial B_1$. Each $K_{i, j}$ will be called a $j$-cell of the grid.
We define the $j$-skeleton of the grid as
\[
  R_j := \bigcup_{i=1}^{k_j} K_{i, j} \qquad \textrm{for } j \in \{0, \, 1, \, 2\} .
\]
For our purposes, we need to consider grids with some special properties.
 
\begin{defn} \label{def:good grid}
 Let $h\colon (0, \, \epsilon_1] \to (0, \, +\infty)$ be a fixed function.
 A family of grids~$\mathscr G := \{\mathscr G^\epsilon\}_{0 < \epsilon \leq \epsilon_1}$ will be called a \emph{good family of grids of size~$h$}
 if there exists a constant~$C_{\mathscr{G}}>0$ which satisfies the following properties.
 \begin{enumerate}[label=\textup{ }{(G\textsubscript{\arabic*})}\textup{ },ref={(G\textsubscript{\arabic*})}]
  \item \label{grid:lipschitz} For each $\epsilon$, $i$, $j$, there exists a bilipschitz homeomorphism
  $\varphi^\epsilon_{i,j} \colon K^\epsilon_{i, j} \to B^j_{h(\epsilon)}$ such that
  \[
    \norm{\D\varphi^\epsilon_{i,j}}_{L^\infty} + \norm{\D(\varphi^\epsilon_{i,j})^{-1}}_{L^\infty} \leq C_{\mathscr{G}} .
  \] 
  \item \label{grid:1-2 cells} For all $p\in \{1, \, 2, \, \ldots, k_1\}$ we have
  \[
    \# \left\{q\in \{1, \, 2, \, \ldots, k_2\} \colon K^\epsilon_{p, 1} \subseteq K^\epsilon_{q, 2} \right\} \leq C_{\mathscr{G}} ,
  \]
  i.e., each $1$-cell is contained in the boundary of at most $C_{\mathscr{G}}$ $2$-cells. 
  \item \label{grid:energy R1} We have
   \[
    E_\epsilon(u_\epsilon, \, R^\epsilon_1) \leq C_{\mathscr{G}} h^{-1}(\epsilon) \, E_\epsilon(u_\epsilon, \, \partial B_1) ,
   \]
   where~$R^\epsilon_1$ denotes the $1$-skeleton of~$\mathscr G^\epsilon$.
   \item \label{grid:f(u) R1} There holds
   \[
    \int_{R^\epsilon_1} f(u_\epsilon) \, \d \H^1 \leq C_{\mathscr{G}} h^{-1}(\epsilon) \int_{\partial B_1} f(u_\epsilon) \, \d \H^2 .
   \]
 \end{enumerate}
\end{defn}
Of course, this definition depends on the family~$\{u_\epsilon\}$, which we assume to be fixed once and for all.
  
\begin{lemma} \label{lemma:grid_exist}
 For any strictly positive function~$h$, a good family of grids of size~$h$ exists.
\end{lemma}
\begin{proof}
 On the unit cube~$\partial [0, \, 1]^3$, consider the uniform grid of size~$\lceil h^{-1}(\epsilon)\rceil^{-1}$, i.e. the grid spanned by the points
 \[
    \left(\lceil h^{-1}(\epsilon)\rceil^{-1} \mathbb Z^3\right) \cap \partial [0, \, 1]^3 
 \]
 (where $\lceil x \rceil$ is, by the definition, the smallest integer $k$ such that $k\geq x$). 
 By applying a bilipschitz homeomorphism $[0, \, 1]^3 \to \overline{B}_1$, one obtains a grid $\mathscr F^\epsilon$ on $\partial B_1$ which satisfy \ref{grid:lipschitz}--\ref{grid:1-2 cells}.
 Denote by $T^\epsilon_1$ the $1$-skeleton of~$\mathscr F^\epsilon$.
 By an average argument, as in \cite[Lemma~1]{Luckhaus-PartialReg}, we find a rotation $\omega\in \mathrm{SO}(3)$ such that
 \[
   E_\epsilon(u_\epsilon, \, \omega(T^\epsilon_1)) \leq C h^{-1}(\epsilon) \, E_\epsilon(u_\epsilon, \, \partial B_1)
 \]
 and
 \[
  \int_{\omega(T^\epsilon_1)} f(u_\epsilon) \, \d \H^1 \leq C h^{-1}(\epsilon) \int_{\partial B_1} f(u_\epsilon) \, \d \H^2.
 \]
 Thus,
 \[
  \mathscr G^\epsilon := \left\{\omega(K) \colon K\in \mathscr F^\epsilon \right\}
 \]
 is a good family of grids of size~$h$.
\end{proof}
 Good families of grids enjoy the following property.

\begin{lemma} \label{lemma:conv_grid}
 Let $\mathscr G$ be a good family of grids on~$\partial B_1$, of size $h$.
 Assume that~\eqref{u energy} holds, and that there exists $\alpha \in(0, \, 1)$ such that
 \begin{equation} \label{hp passo grid}
  \lim_{\epsilon\to 0}  \epsilon^{\alpha} h^{-1}(\epsilon) = 0.
 \end{equation}
 Then, there holds
 \[
  \lim_{\epsilon\to 0}\sup_{x\in R^\epsilon_1} \dist(u_\epsilon(x), \, \NN) = 0 .
 \]
 \end{lemma}
\begin{proof}
The arguments below are adapted from \cite[Lemmas~3.4 and~3.10]{AlbertiBaldoOrlandi} (the reader is also referred to \cite[Lemmas~2.2,~2.3 and~2.4]{BethuelOrlandiSmets-Jacobian}).
Since the Landau-de Gennes potential satisfies \eqref{f dist2} by Lemma~\ref{lemma:f below}, there exist positive numbers~$\beta$,~$C$ and a continuous function~$\psi\colon [0, \, +\infty) \to \R$ such that
\[
 \begin{cases}
  \psi(s) = \beta  s^2 &\textrm{for } 0 \leq s < \delta_0 \\
  0 < \psi(s) \leq C &\textrm{for  } s\geq \delta_0 \\
  \psi(\dist(v, \, \NN)) \leq f(v) &\textrm{for any } v\in\Sz .
 \end{cases}
\]
Denote by $G$ a primitive of $\psi^{1/6}$, and set $d_\epsilon:= \dist(u_\epsilon, \, \NN)$.
Since the function $\dist(\cdot, \, \NN)$ is $1$-Lipschitz continuous, we have $d_\epsilon\in H^1(\Omega, \, \R)$ and $\abs{\nabla d_\epsilon} \leq \abs{\nabla u_\epsilon}$.
Moreover, $\psi(d_\epsilon) \leq f(u_\epsilon)$ by construction of $\psi$.
Thus,~\eqref{u energy} and~\ref{grid:energy R1} entail
\[
 h(\epsilon)\int_{R^\epsilon_1} \left\{\frac12 \abs{\nabla d_\epsilon}^2 + \epsilon^{-2} \psi(d_\epsilon)\right\} \, \d \H^1 \leq MC_{\mathscr{G}}|\log\epsilon|.
\]
By applying Young's inequality $a + b \geq C a^{3/4}b^{1/4}$, we obtain
\begin{equation} \label{control gradient R_1}
\begin{split}
 MC_{\mathscr{G}}|\log\epsilon| &\geq C \epsilon^{-1/2}h(\epsilon)\int_{R^\epsilon_1} \abs{\nabla d_\epsilon}^{3/2}\psi^{1/4}(d_\epsilon) \, \d \H^1\\
                        &= C \epsilon^{-1/2} h(\epsilon)\int_{R^\epsilon_1} \abs{\nabla G(d_\epsilon)}^{3/2} \, \d \H^1 .
 \end{split}
\end{equation}

Fix a $1$-cell $K$ of $\mathscr{G}^\epsilon$.
We control the oscillations of $G(d_\epsilon)$ over $K$ thanks to the Sobolev embedding $W^{1, 3/2}(K) \hookrightarrow C^0(K)$ and~\eqref{control gradient R_1}:
\[
\begin{split}
 \left(\underset{K}{\mathrm{osc} \, } G(d_\epsilon)\right)^{3/2} &\leq C h^{1/2}(\epsilon)\int_K \abs{\nabla G(d_\epsilon)}^{3/2} \, \d \H^1 \\
                                                      &= C \epsilon^{1/2} h^{-1/2}(\epsilon) \abs{\log\epsilon} .
\end{split}
\]
In view of~\eqref{hp passo grid}, we obtain
\[
 \underset{R^\epsilon_1}{\mathrm{osc} \, } G(d_\epsilon) \to 0
\]
as $\epsilon\to 0$.
The function~$G$ is a continuous and strictly increasing, so $G$ has a continuous inverse. This implies
\begin{equation} \label{osc v_eps}
 \underset{R^\epsilon_1}{\mathrm{osc} \, }  d_\epsilon \to 0
\end{equation}
as $\epsilon\to 0$. On the other hand,~\eqref{u energy},~\ref{grid:energy R1} and \eqref{hp passo grid} yield
\begin{equation} \label{mean phi v_eps}
 \fint_K \psi(d_\epsilon) \, \d \H^1 \leq \frac{1}{h(\epsilon)} \int_{R^\epsilon_1} f(u_\epsilon) \, \d\H^1 \to 0
\end{equation}
as~$\epsilon\to 0$, for any $1$-cell $K$ of $\mathscr G_\epsilon$.
As we will see in a moment, this implies
\begin{equation} \label{mean v_eps}
 \sup_K \fint_K  d_\epsilon \, \d \H^1 \to 0 .
\end{equation}
Combining~\eqref{mean v_eps} with~\eqref{osc v_eps}, we conclude that $d_\epsilon$ converges uniformly to $0$ as $\epsilon\to 0$.

Now, we check that~\eqref{mean v_eps} holds. There exists a constant $\lambda > 0$ such that
 \[
  \norm{d_\epsilon}_{L^\infty(\Omega)} \leq \lambda
 \]
 (this follows from the uniform $L^\infty$-estimate for $u_\epsilon$,~\eqref{hp_interpolation2:2}).
 For any $\delta\in(0, \, \lambda)$, set
 \[
  \psi_*(\delta) :=\inf_{\delta \leq s \leq \lambda} \psi(s) > 0 . 
 \]
 Then,
 \begin{equation} \label{ostrogoto}
  \frac{\H^1\left(\left\{d_\epsilon \geq \delta\right\} \cap K\right)}{\H^1(K)} \, \psi_*(\delta)
  \leq \frac{1}{\H^1(K)} \int_{\{d_\epsilon \geq \delta\}\cap K} \psi(d_\epsilon)  \, \d \H^1 
  \leq \fint_K \psi(d_\epsilon)  \, \d \H^1 .
 \end{equation}
 Thus, for any $1$-cell $K$, we have
 \[
 \begin{split}
  0 \leq \fint_K d_\epsilon \, \d \H^1 &= \frac{1}{\H^1(K)} \int_{\{d_\epsilon \leq \delta\}\cap K} d_\epsilon  \, \d \H^1 
      + \frac{1}{\H^1(K)} \int_{\{d_\epsilon \geq \delta\} \cap K} d_\epsilon  \, \d \H^1 \\
      &\leq \frac{\H^1\left(\left\{d_\epsilon \leq \delta \right\} \cap K\right)}{\H^1(K)} \, \delta 
      + \frac{\H^1\left(\left\{d_\epsilon \geq \delta \right\} \cap K\right)}{\H^1(K)} \, \lambda \\
      &\stackrel{\eqref{ostrogoto}}{\leq} \delta + \frac{\lambda}{\psi_*(\delta)} \fint_K \psi(d_\epsilon)  \, \d \H^1 \\
      &\stackrel{\eqref{mean phi v_eps}}{\leq} \delta + \frac{C\lambda}{\psi_*(\delta)} \epsilon^2 h^{-1}(\epsilon)  \abs{\log\epsilon} .
  \end{split}
 \]
 We pass to the limit first as~$\epsilon\to 0$, then as~$\delta\to 0$. Using~\eqref{hp passo grid}, we deduce~\eqref{mean v_eps}.
\end{proof}

\begin{remark} \label{remark:constants1}
 As a byproduct of the proof, under the assumptions of Lemma~\ref{lemma:grid_exist} the following property holds.
 For any~$\delta > 0$, there exists a positive number~$\epsilon_1$ \emph{which only depend on $M$, $\kappa$, $C_{\mathscr{G}}$ and the potential~$f$}, such that
 \[
  \dist(u_\epsilon(x), \, \NN) \leq \delta
 \]
 for any~$0 < \epsilon \leq \epsilon_1$ and any~$x\in R^\epsilon_1$.
 Here~$M$, $\kappa$ and~$C_{\mathscr{G}}$ are given respectively by~\eqref{u energy}, \eqref{hp_interpolation2:2}, and Definition~\ref{def:good grid}.
\end{remark}

\subsection{Construction of~$v_\epsilon$ and~$\varphi_\epsilon$}
\label{subsect:v_eps}

First, we construct the approximating map~$v_\epsilon\colon \partial B_1\to\NN$.

 \begin{lemma} \label{lemma:v_eps}
 Assume that~\eqref{u energy}, \eqref{approximable} hold.
 There exists~$\epsilon_1> 0$ such that, for any $0 < \epsilon \leq \epsilon_1$, there exists a map~$v_\epsilon\in H^1(\partial B_1, \, \NN)$ which satisfy~\eqref{v energy},
  \begin{equation} \label{v_eps R1}
    v_\epsilon(x) = \RR(u_\epsilon(x)) \qquad \textrm{and} \qquad \abs{u_\epsilon(x) - v_\epsilon(x)} \leq \delta_0 
  \end{equation}
 for every~$x\in R^\epsilon_1$.
 \end{lemma}
 
\begin{proof}
 To construct~$v_\epsilon$, we take a family~$\mathscr G = \{\mathscr G^\epsilon\}_{\epsilon > 0}$ of grids of size
  \begin{equation} \label{h(eps)}
  h(\epsilon) := \epsilon^{1/2} \abs{\log \epsilon}
 \end{equation}
 (such a family exists by Lemma~\ref{lemma:grid_exist}). 
 Condition~\eqref{hp passo grid} is satisfied for $\alpha = 1/2$, so by Lemma~\ref{lemma:conv_grid} there exists $\epsilon_1> 0$ such that
 \begin{equation} \label{u dist leq delta0}
  \dist(u_\epsilon(x), \, \NN) \leq \delta_0 \qquad \textrm{for any } \epsilon\in (0, \, \epsilon_1) \textrm{ and any } x\in R^\epsilon_1 .
 \end{equation}
 The constant~$\delta_0$ is given by Lemma~\ref{lemma:f below}.
 In particular, the formula
 \[
  v_\epsilon(x) := \RR(u_\epsilon(x)) \qquad \textrm{for all } x\in R^\epsilon_1
 \]
 defines a function~$v_\epsilon\in H^1(R^\epsilon_1, \, \Sz)$ which satisfies~\eqref{v_eps R1}.

 To extend $v_\epsilon$ inside each $2$-cell, we use Lemma~\ref{lemma:extension2}.
 Fix a $2$-cell~$K$ of~$\mathscr G_\epsilon$.
 Since we have assumed that~\eqref{approximable} holds, $v_{\epsilon|\partial K}$ is homotopically trivial. 
 Therefore, Lemma~\ref{lemma:extension2} and  \ref{grid:lipschitz} imply that there exists~$v_{\epsilon, K}\in H^1(K,\, \NN)$ such that $v_{\epsilon, K} = v_\epsilon$ on~$\partial K$ and
 \[
   \int_{K} \abs{\nabla v_{\epsilon, K}}^2 \, \d \H^2 \leq C h(\epsilon) \int_{\partial K} \abs{\nabla v_\epsilon}^2 \, \d \H^1 .
 \]
Define $v_\epsilon\colon \partial B_1 \to \NN$ by setting $v_\epsilon := v_{\epsilon, K}$ on each $2$-cell $K$.
This function agrees with ${v_\epsilon}_{|R^1_\epsilon}$ previously defined by~\eqref{v_eps R1}, hence the notation is not ambiguous.
Moreover, $v_\epsilon\in H^1(\partial B_1, \, \NN)$ and 
\[
\begin{split}
 \int_{\partial B_1 } \abs{\nabla v_\epsilon}^2 \, \d \H^2 &\leq \sum_{K}\int_{K} \abs{\nabla v_\epsilon}^2 \, \d \H^2 \leq C h(\epsilon) \sum_K \int_{\partial K} \abs{\nabla v_\epsilon}^2 \, \d \H^1 \\
 &\stackrel{\ref{grid:1-2 cells}}{\leq} C h(\epsilon) \int_{R^\epsilon_1} \abs{\nabla v_\epsilon}^2 \, \d \H^1 \stackrel{\eqref{v_eps R1}}{\leq} C h(\epsilon) \int_{R^\epsilon_1} \abs{\nabla u_\epsilon}^2 \, \d \H^1 \\
 &\stackrel{\ref{grid:energy R1}}{\leq} C E_\epsilon(u_\epsilon, \, \partial B_1) ,
 \end{split} 
\]
where the sum runs over all the $2$-cells $K$ of $\mathscr G_\epsilon$.
Thus~$v_\epsilon$ satisfies~\eqref{v energy}, so  the lemma is proved.
\end{proof}

Now, we construct the interpolation map~$\varphi_\epsilon\colon \partial B_1 \to \Sz$.

\begin{lemma} \label{lemma:phi_eps}
  Assume that the conditions \eqref{u energy}, \eqref{approximable} are fulfilled.
  Then, for any~$0 < \epsilon \leq \epsilon_1$ there exists a map~$\varphi_\epsilon\in H^1(B_1 \setminus B_{1 - h(\epsilon)}, \, \Sz)$
  which satisfies~\eqref{phi bd value} and~\eqref{phi energy}.
\end{lemma}
 
\begin{proof}
 Set $A_\epsilon := B_1 \setminus B_{1 - h(\epsilon)}$.
 The grid $\mathscr G^\epsilon$ on $\partial B_1 $ induces a grid $\hat{\mathscr G}^\epsilon$ on $A_\epsilon$, whose cells are
 \[
  \hat K := \left\{x\in \R^3 \colon 1 - h(\epsilon) \leq \abs{x} \leq 1, \frac{x}{\abs{x}}\in K \right\} \qquad \textrm{for each } K\in \mathscr G^\epsilon .
 \]
 If $K$ is a cell of dimension $j$, then $\hat K$ has dimension $j + 1$.
 For $j\in\{0, \, 1, \, 2\}$, we call $\hat R^\epsilon_ j$ the union of all the $(j + 1)$-cells of~$\hat{\mathscr G}^\epsilon$.
 
 The function~$\varphi_\epsilon$ is constructed as follows.
 If $x\in \partial B_1 \cup \partial B_{1 - h(\epsilon)}$, then $\varphi_\varepsilon(x)$ is determined by \eqref{phi bd value}.
 If $x\in \hat R^\epsilon_1\cup \hat R^\epsilon_1$, we define $\varphi_\epsilon(x)$ by linear interpolation:
 \begin{equation} \label{phi linear}
  \varphi_\epsilon(x) := \frac{1 - \abs{x}}{h(\epsilon)} \, u_\epsilon\left(\frac{x}{\abs{x}}\right) + \frac{h(\epsilon) - 1 + \abs{x}}{h(\epsilon)} \, v_\epsilon\left(\frac{x}{\abs{x}}\right)  .
 \end{equation}
 For any $3$-cell $\hat K$ of $\mathscr G_\epsilon$, we extend homogeneously (of degree~$0$) the function ${\varphi_\epsilon}_{|\partial \hat K}$ on~$\hat K$.
 This gives a map $\varphi_\epsilon\in H^1(\hat K)$, because $\hat K$ is a cell of dimension $3$.
 As a result, we obtain a map~$\varphi_\epsilon \in H^1(A_\epsilon, \, \Sz)$ which satisfies~\eqref{phi bd value}.
 
 To complete the proof of the lemma, we only need to bound the energy of $\varphi_\epsilon$.
 Since $\varphi_\epsilon$ has been obtained by homogeneous extension on cells of size $h(\epsilon)$,  we have
 \begin{equation} \label{phi_eps energy 1}
 \begin{split}
  E_\epsilon(\varphi_\epsilon, \, A_\epsilon) &\stackrel{\ref{grid:lipschitz}}{\leq} C h(\epsilon) \sum_{\hat K} E_\epsilon(\varphi_\epsilon, \, \partial \hat K) \\
  &\stackrel{\ref{grid:1-2 cells}}{\leq} C h(\epsilon) \left\{ E_\epsilon(u_\epsilon, \, \partial B_1) + E_\epsilon(v_\epsilon, \, \partial B_{1 - h(\epsilon)}) + 
  E_\epsilon(\varphi_\epsilon, \, \hat R^\epsilon_1) \right\} ,
  \end{split}
 \end{equation}
 where the sum runs over all the $3$-cells $\hat K$ of $\hat{\mathscr G}^\epsilon$.
 To conclude the proof, we invoke the following fact.

 \begin{lemma} \label{lemma:phi_eps R1}
  We have
  \[
   E_\epsilon(\varphi_\epsilon, \, \hat{R}^\epsilon_1) \leq C \left(\epsilon^2 h^{-2}(\epsilon) + 1 \right) E_\epsilon(u_\epsilon, \, \partial B_1) .
  \]
 \end{lemma}

 From~\eqref{phi_eps energy 1} and Lemma~\ref{lemma:phi_eps R1} we get
\begin{equation*}
 \begin{split}
  E_\epsilon(\varphi_\epsilon, \, A_\epsilon) &\leq C h(\epsilon) \bigg\{ \left(\epsilon^2 h^{-2}(\epsilon) + 1\right) E_\epsilon(u_\epsilon, \, \partial B_1) + E_\epsilon(v_\epsilon, \, \partial B_{1 - h(\epsilon)}) \bigg\} \\
  &\stackrel{\eqref{v energy}}{\leq} C h(\epsilon) \left(\epsilon^2 h^{-2}(\epsilon) + 1\right) E_\epsilon(u_\epsilon, \, \partial B_1)
  \end{split}
 \end{equation*}
and, thanks to our choice~\eqref{h(eps)} of~$h(\epsilon)$, we conclude that~\eqref{phi energy} holds, so Lemma~\ref{lemma:phi_eps} is proved.
\end{proof}

\begin{remark} \label{remark:constants2}
 We can keep track of the constants in the proof of Lemmas~\ref{lemma:v_eps} and~\ref{lemma:phi_eps}.
 By doing so, one sees that the constant~$C$ given by Proposition~\ref{prop:interpolation2} (Equations~\eqref{v energy} and~\eqref{phi energy}) only depends on~$C_{\mathscr{G}}$ and the potential~$f$.
\end{remark}

\begin{proof}[Proof of Lemma~\ref{lemma:phi_eps R1}]
 We consider first the contribution of the potential energy.
 Thanks to~\eqref{f loc_convex}, \eqref{phi linear} and~\eqref{v_eps R1}, we deduce that
 \[
  f(\varphi_\epsilon(x)) \leq C \left(\frac{1 - \abs{x}}{h(\epsilon)}\right)^2 \, (f\circ u_\epsilon)\left(\frac{x}{\abs{x}}\right) \qquad \textrm{for } x\in \hat R^\epsilon_1 .
 \]
 By integration, this gives
 \begin{equation} \label{R1, 1}
  \int_{\hat R^\epsilon_1} f(\varphi_\epsilon) \, \d \H^2 \leq C h(\epsilon) \int_{R^\epsilon_1} f(u_\epsilon) \, \d \H^2 .
 \end{equation}
 Now, we consider the elastic part of the energy.
 Using again the definition~\eqref{phi linear} of $\varphi_\epsilon$ on $\hat R^\epsilon_1$, we have
 \begin{equation} \label{R1, 2}
   \int_{\hat R^\epsilon_1} \abs{\nabla \varphi_\epsilon}^2 \, \d \H^2 \leq C h^{-1}(\epsilon) \int_{R^\epsilon_1} \abs{u_\epsilon - v_\epsilon}^2 \, \d \H^1 .
 \end{equation}
 The condition~\eqref{f dist2} on the Landau-de Gennes potential, together with~\eqref{v_eps R1}, implies
 \begin{equation} \label{R1, 3}
  \int_{R^\epsilon_1} \abs{u_\epsilon - v_\epsilon}^2 \, \d\H^1 \leq C \int_{R^\epsilon_1} f(u_\epsilon) \, \d \H^1 .
 \end{equation}
 Using~\eqref{R1, 1},~\eqref{R1, 2} and~\eqref{R1, 3}, we deduce that
 \[
  E_\epsilon(\varphi_\epsilon, \hat R^\epsilon_1) \leq C \left(h^{-1}(\epsilon) + \epsilon^{-2}h(\epsilon) \right) \int_{R^\epsilon_1} f(u_\epsilon) \, \d \H^1 .
 \]
 Because of Condition~\ref{grid:f(u) R1} in Definition~\ref{def:good grid}, we obtain
  \[
  E_\epsilon(\varphi_\epsilon, \hat R^\epsilon_1) \leq C \left(h^{-2}(\epsilon) + \epsilon^{-2} \right) \int_{\partial B_1} f(u_\epsilon) \, \d \H^2
  \]
 so the lemma follows easily.
\end{proof}

\subsection{Logarithmic bounds for the energy imply~\eqref{approximable}}
\label{subsect:eta implies C}

The aim of this subsection is to establish the following lemma, and conclude the proof of Proposition~\ref{prop:interpolation2}.

\begin{lemma} \label{lemma:eta implies C}
 There exists $\eta_1 = \eta_1(\NN, \, C_{\mathscr{G}}, \, M, \, \epsilon_1)$ such that, if $0 < \epsilon < \epsilon_1$ and $u_\epsilon$ satisfies~\eqref{u energy}, \eqref{hp_interpolation2:2} but \emph{not}~\eqref{approximable}, then
 \[
  E_\epsilon(u_\epsilon, \, \partial B_1) \geq \eta_1 \abs{\log\epsilon} .
 \]
\end{lemma}

Once Lemma~\ref{lemma:eta implies C} is proved, Proposition~\ref{prop:interpolation2} follows in an elementary way.
\begin{proof}[Proof of Proposition~\ref{prop:interpolation2}]
 Choose $\eta_0 := \eta_1/2$.
 If~$u_\epsilon$ satisfies~\eqref{hp_interpolation2:1} with this choice of~$\eta_0$ and~\eqref{hp_interpolation2:2}, then it must satisfy Condition~\eqref{approximable}, otherwise Lemma~\ref{lemma:eta implies C}
 would yield a contradiction.
 Then, the proposition follows by Lemmas~\ref{lemma:v_eps} and~\ref{lemma:phi_eps}.
\end{proof}

\begin{proof}[Proof of Lemma~\ref{lemma:eta implies C}]
By assumption, Condition~\eqref{approximable} is not satisfied, so there exists a $2$-cell $K^*\in\mathscr G^\epsilon$ such that~${\RR\circ u_\epsilon}_{|\partial K^*}$ is non-trivial.
By Definition~\ref{def:good grid}, there exists a bilipschitz homeomorphism $\varphi\colon K_* \to B_{h(\epsilon)}$ which satisfies~\ref{grid:lipschitz}.
Therefore, up to composition with~$\varphi$ we can assume that $K_*$ is a $2$-dimensional disk, $K_* = B^2_{h(\epsilon)}$.
Lemma~\ref{lemma:conv_grid} implies that $u_\epsilon(x)\notin \mathscr C_0$ for every $x\in \partial K_*$, for $0 < \epsilon\leq\epsilon_1$.
Then, by applying Corollary~\ref{cor:lower bound} we deduce
\[
 E_\epsilon(u_\epsilon, \, K_*) + Ch(\epsilon) E_\epsilon(u_\epsilon, \, \partial K_*) \geq \kappa_* \phi_0^2(u_\epsilon, \, \partial K_*) \log\frac{h(\epsilon)}{\epsilon} - C 
\]
Notice that~$\phi_0(u_\epsilon, \, \partial K_*)\geq 1/2$ if~$\delta_0$ is small enough, because of~\eqref{u dist leq delta0}.
On the other hand, condition \ref{grid:energy R1} yields
\[
 E_\epsilon(u_\epsilon, \, K_*) + Ch(\epsilon) E_\epsilon(u_\epsilon, \, \partial K_*) \leq C E_\epsilon(u_\epsilon, \, \partial B_1) .
\]
Due to the previous inequalities and~\eqref{h(eps)}, we infer
\[
  E_\epsilon(u_\epsilon, \, \partial B_1) \geq C \left\{ \log\left(\epsilon^{-1/2}\abs{\log\epsilon}\right) - 1 \right\} \geq C \left( \frac12 \abs{\log\epsilon} - 1 \right) 
\]
for all $0 < \epsilon \leq \epsilon_1< 1$, so the lemma follows.
\end{proof}

\section{Compactness of Landau-de Gennes minimizers: Proof of Theorem~\ref{th:convergence}}
\label{sect:convergence}

\subsection{Concentration of the energy: Proof of Proposition~\ref{prop:desiderata}}
\label{subsect:desiderata}
 
 The whole section aims at proving Theorem~\ref{th:convergence}.
 In this subsection, we prove Proposition~\ref{prop:desiderata} by applying the results of Section~\ref{sect:extension}.
 
 Let~$\eta_0$,~$\epsilon_1$ be given by Proposition~\ref{prop:interpolation2}, and set~$\epsilon_0 := \epsilon_1\theta$.
 Throughout the section, the same symbol~$C$ will be used to denote several different constants, possibly depending on~$\theta$ and~$\epsilon_1$, but not on~$\varepsilon$, $R$.
 To simplify the notation, from now on we assume that~$x_0 = 0$.
 For a fixed~$0 < \varepsilon \leq \epsilon_0 R$, define the set
 \[
  D^\varepsilon := \left\{r\in (\theta R, \, R) \colon E_\varepsilon(\Qe, \, \partial B_r) \leq \frac{2\eta}{1 - \theta} \log\frac{R}{\varepsilon} \right\} .
 \]
 The elements of $D^\varepsilon$ are the ``good radii'', i.e.~$r\in D^\varepsilon$ means that we have a control on the energy on the sphere of radius~$r$.
 Assume that the condition~\eqref{desid: hp} is satisfied. Then, by an average argument we deduce that
 \begin{equation} \label{D small}
  \H^1(D^\varepsilon) \geq \frac{(1 - \theta) R}{2} .
 \end{equation}
 For any $r\in D^\varepsilon$ we have
 \[
  E_\varepsilon(\Qe, \, \partial B_r) \leq \frac{2\eta}{1 - \theta} \left( \log\frac{r}{\varepsilon} - \log\theta \right) ,
 \]
 since $R \leq \theta^{-1} r$. By choosing $\eta$ small enough, we can assume that
 \begin{equation} \label{to scale}
  E_\varepsilon(\Qe, \, \partial B_r) \leq \eta_0 \log\frac{r}{\varepsilon} \qquad \textrm{for any } r\in D^\varepsilon \textrm{ and } 0 < \varepsilon \leq \epsilon_0 R. 
 \end{equation}
 In particular, our choice of $\eta$ depends on $\epsilon_1$, $\eta_0$, $\theta$.
 
 \begin{lemma} \label{lemma:HKL ineq}
 For any~$0 < \varepsilon \leq \epsilon_0 R$ and any~$r\in D^\varepsilon$, there holds
 \[
   E_\varepsilon(\Qe, \, B_r) \leq C R \left(E_\varepsilon^{1/2} (\Qe, \, \partial B_r) + 1\right) .
 \]
 \end{lemma}
 A similar inequality was obtained by Hardt, Kinderlehrer and Lin in \cite[Lemma~2.3, Equation~(2.3)]{HKL}, and it played a crucial role in the proof of their energy improvement result.
  
 \begin{proof}[Proof of Lemma~\ref{lemma:HKL ineq}]
 To simplify the notations, we get rid of $r$ by a scaling argument. Set $\epsilon := \varepsilon/r$, and define the function $u_\epsilon\colon B_1 \to \Sz$ by
 \[
  u_\epsilon(x) := \Qe(r x) \qquad \textrm{for all } x\in B_1 .
 \]
 Notice that~$\epsilon \leq \epsilon_1$, since~$\varepsilon \leq \epsilon_0 R$ and~$\theta R < r$.
 The lemma will be proved once we show that
 \begin{equation} \label{HKL scaled}
  E_\epsilon(u_\epsilon, \, B_1) \leq C E_\epsilon^{1/2}(u_\epsilon, \, \partial B_1) + 1
 \end{equation}
 (multiplying both sides of~\eqref{HKL scaled} by~$r\leq R$ yields the lemma).
 Since we have assumed that~$r\in D^\varepsilon$ we have, by~\eqref{to scale},
 \begin{equation*} 
  E_\epsilon(u_\epsilon, \, \partial B_1) \leq \eta_0 \abs{\log\epsilon}.
 \end{equation*}
 Moreover, $u_\epsilon$ satisfies the $L^\infty$-bound~\eqref{hp_interpolation2:2}, due to~\eqref{hp:H}.
 Therefore, we can apply Proposition~\ref{prop:interpolation2} and find~$v_\epsilon\in H^1(\partial B_1, \, \NN)$, $\varphi_\epsilon\in H^1(A_\epsilon, \, \Sz)$ which satisfy
 \begin{gather}
  \varphi_\epsilon(x) = u_\epsilon(x) \quad \textrm{and} \quad \varphi_\epsilon(x - h(\epsilon)x) = v_\epsilon(x) \qquad \textrm{for }\H^2\textrm{-a.e. } x\in \partial B_1 \nonumber \\
  \int_{\partial B_1} \abs{\nabla v_\epsilon}^2 \d \H^2 \leq C E_\epsilon(u_\epsilon, \, \partial B_1) , \label{v energy2} \\
  E_\epsilon(\varphi_\epsilon, \, A_\epsilon) \leq C h(\epsilon) E_\epsilon(u_\epsilon, \, \partial B_1) . \label{phi energy2}
 \end{gather}
 Here~$h(\epsilon):= \epsilon^{1/2}|\log\epsilon|$ and~$A_\epsilon := B_1\setminus B_{1 - h(\epsilon)}$.
 By applying Lemma~\ref{lemma:extension1} to~$v_\epsilon$, we find a map~$w_\epsilon\in H^1(B_1, \, \NN)$ such that ${w_\epsilon}_{|\partial B_1}$ and 
 \begin{equation} \label{w_eps energy}
   \begin{split}
    \int_{B_1} \abs{\nabla w_\epsilon}^2 \leq C \left\{\int_{\partial B_1} \abs{\nabla v_\epsilon}^2 \, \d \H^2\right\}^{1/2}
    \stackrel{\eqref{v energy2}}{\leq} C E^{1/2}_\epsilon(u_\epsilon, \, \partial B_1) .
   \end{split}
 \end{equation}
 Now, define the function $\tilde w_\epsilon\colon B_1 \to \Sz$ by
 \[
  \tilde w_\epsilon(x) := \begin{cases}
  \varphi_\epsilon(x) & \textrm{for } x \in A_\epsilon \\
  w_\epsilon \left(\dfrac{x}{1 - h(\epsilon)}\right)  & \textrm{for } x\in B_{1 - h(\epsilon)} .
  \end{cases} 
 \]
 The energy of $\tilde w_\epsilon$ in the spherical shell~$A_\epsilon$ is controlled by~\eqref{phi energy2}.
 Due to our choice of the parameter $h(\epsilon)$, we deduce that
 \[
  E_\epsilon(\tilde w_\epsilon, \, A_\epsilon) \leq 1
 \]
 provided that $\epsilon_1$ is small enough.
 Combining this with~\eqref{w_eps energy}, we obtain
 \[
  E_\epsilon(\tilde w_\epsilon, \, B_1) \leq C E^{1/2}_\epsilon(u_\epsilon, \, \partial B_1) + 1 .
 \]
 But $\tilde w_\epsilon$ is an admissible comparison function for $u_\epsilon$ on $B_1$, because $\tilde w_\epsilon = u_\epsilon$ on $\partial B_1$.
 Thus, the minimality of $u_\epsilon$ implies~\eqref{HKL scaled}.
  \end{proof}
  
  Lemma~\ref{lemma:HKL ineq} can be seen as a non-linear differential inequality for the function~$y\colon r\in (\theta R, \, R)\mapsto E_\varepsilon(Q, B_r)$.
  The conclusion of the proof of Proposition~\ref{prop:desiderata} follows now by a simple ODE argument.
  
  \begin{lemma}\label{lemma:ODE}
   Let~$\alpha$, $\beta$ be two positive numbers. Let~$y\in W^{1, 1}([r_0, \, r_1], \, \R)$ be a function such that~$y^\prime\geq 0$ a.e., 
   and let~$D\subseteq (r_0, \, r_1)$ be a measurable set such that~$\H^1(D) \geq (r_1 - r_0)/2$.
   If the function~$y$ satisfies
   \begin{equation} \label{ODE-hp}
    y(r) \leq \alpha {y^\prime}(r)^{1/2} + \beta \qquad \textrm{for } \H^1\textrm{-a.e. } r\in D,
   \end{equation}
   then there holds
   \[
    y(r_0) \leq \beta + \frac{2\alpha^2}{r_1 - r_0} .
   \]
  \end{lemma}
  \proof
   If there exists a point~$r_*\in (r_0, \, r_1)$ such that~$y(r_*) \leq \beta$, then~$y(r_0) \leq \beta$
   (because~$y$ is an increasing function) and the lemma is proved. Therefore, we can assume WLOG that~$y - \beta > 0$ on~$(r_0, \, r_1)$.
   Then, Equation~\eqref{ODE-hp} and the monotonicity of~$y$ imply
   \[
    \frac{y^\prime(r)}{\left(y(r) - \beta\right)^2} \geq \alpha^{-2} \one_D(r) \qquad \textrm{for a.e. } r\in (r_0, \, r_1)
   \]
   where~$\one_D$ is the characteristic function of~$D$ (that is, $\one_{D}(r) = 1$ if $r\in D$ and $\one_{D}(r) = 0$ otherwise).
   By integrating this inequality on~$(0, \, r)$, we deduce
   \[
    \frac{1}{y(r_0) - \beta} - \frac{1}{y(r) - \beta} \geq \alpha^{-2} \H^1\left((r_0, \, r) \cap D\right) \qquad \textrm{for any } r\in (r_0, \, r_1) .
   \]
   Since we have assumed that~$\H^1(D) \geq (r_1 - r_0)/2$, we obtain
   \[
    \H^1((r_0, \, r) \cap D) \geq \left(r - \frac{r_0 + r_1}{2}\right)^+ := \max\left\{r - \frac{r_0 + r_1}{2}, \, 0 \right\}
   \]
   so, via an algebraic manipulation,
   \[
    y(r) \geq \beta + \frac{y(r_0) - \beta}{1 - \alpha^{-2}\left(r - (r_0 + r_1)/2\right)^+ \left(y(r_0) - \beta\right)} \qquad \textrm{for any } r\in (r_0, \, r_1) .
   \]
   Since~$y$ is well-defined (and finite) up to~$r = r_1$, there must be
   \[
    1 - \frac{r_1 - r_0}{2\alpha^2}\left(y(r_0) - \beta\right) > 0 ,
   \]
   whence the lemma follows.
  \endproof
  
  \begin{proof}[Conclusion of the proof of Proposition~\ref{prop:desiderata}]
  Thanks to Lemma~\ref{lemma:HKL ineq} and~\eqref{D small}, we can apply Lemma~\ref{lemma:ODE} 
  to the function~$y(r) := E_\varepsilon(Q_\varepsilon, \, B_r)$, for~$r\in (\theta R, \, R)$, and the set~$D := D^\varepsilon$.
  This yields
  \[
   E_\varepsilon(Q_\varepsilon, \, B_{\theta R}) \leq C R ,
  \]
  so the proposition is proved.
 \end{proof}
 
 \subsection{Uniform energy bounds imply convergence to a harmonic map}
 \label{subsect:H1 bounds}
 
 In this subsection, we suppose that minimizers satisfy
 \begin{equation} \label{energy bounded}
   E_\varepsilon(\Qe, \, B_R(x_0)) \leq C R
 \end{equation}
 on a ball~$B_r(x_0)\csubset\Omega$.
 In interesting situations, where line defects appear, such an estimate is not valid over the whole of~$\Omega$ but it is satisfied locally, away from a singular set.
 The main result of this subsection is the following:
 
 \begin{prop} \label{prop:intro-H1 conv}
 Assume that $\overline B_R(x_0) \subseteq\Omega$ and that \eqref{energy bounded} is satisfied for some $R$, $C > 0$. Fix $0 < \theta < 1$.
 Then, there exist a subsequence $\varepsilon_n\searrow 0$ and a map~${Q}_0\in H^1(B_{\theta R}(x_0), \, \NN)$ such that
 \[
  {Q}_{\varepsilon_n} \to {Q}_0 \qquad \textrm{strongly in }  H^1(B_{\theta R}(x_0), \, \Sz) .
 \]
 The map~${Q}_0$ is minimizing harmonic in $B_{\theta R}(x_0)$, that is, for any~$Q\in H^1(B_{\theta R}(x_0), \, \NN)$ such that~$Q = {Q}_0$ on~$\partial B_{\theta R}(x_0)$ there holds
 \[
  \frac12 \int_{B_{\theta R}(x_0)} \abs{\nabla {Q}_0}^2 \leq \frac 12 \int_{B_{\theta R}(x_0)}\abs{\nabla Q}^2 .
 \]
 \end{prop}
 
 In general, we cannot expect the map ${Q}_0$ to be smooth (see the example of Section~\ref{sect:SX}).
 In contrast, by Schoen and Uhlenbeck's partial regularity result~\cite[Theorem~II]{SchoenUhlenbeck} we know that there exists a finite set $\Spt\subseteq B_{\theta R}(x_0)$ such that ${Q}_0$ is smooth on $B_{\theta R}(x_0)\setminus\Spt$.
 Accordingly, the sequence $\{Q_{\varepsilon_n}\}$ will not converge uniformly to ${Q}_0$ on the whole of $B_{\theta R}(x_0)$, in general, but we can prove the uniform convergence away from the singularities of ${Q}_0$.
 
 \begin{prop} \label{prop:unif conv}
  Let $K\subseteq B_{\theta R}(x_0)$ be such that~${Q}_0$ is smooth on the closure of~$K$.
  Then ${Q}_{\varepsilon_n}\to {Q}_0$ uniformly on~$K$.
 \end{prop}
 
 The asymptotic behaviour of minimizers of the Landau-de Gennes functional, in the bounded-energy regime~\eqref{energy bounded}, was already studied by Majumdar and~Zarnescu in~\cite{MajumdarZarnescu}.
 In that paper, $H^1$-convergence to a harmonic map and local uniform convergence away from the singularities of ${Q}_0$ were already proven.
 However, in our case some extra care must be taken, because of the local nature of our assumption~\eqref{energy bounded}.
 
 \begin{proof}[Proof of Proposition~\ref{prop:intro-H1 conv}]
 Up to a translation, we assume that $x_0 = 0$.
 In view of~\eqref{energy bounded}, there exists a subsequence $\varepsilon_n \searrow 0$ and a map~\mbox{${Q}_0\in H^1(B_R, \, \Sz)$} such that
 \[
  {Q}_{\varepsilon_n} \to {Q}_0 \qquad \textrm{weakly in } H^1(B_R, \, \Sz), \textrm{ strongly in } L^2(B_R, \, \Sz) \textrm{ and a.e.}
 \]
 Using Fatou's lemma and~\eqref{energy bounded} again, we also see that
 \[
  \int_{B_R} f({Q}_0) \leq \liminf_{n\to+\infty} \varepsilon_n^2 E_{\varepsilon_n}({Q}_{\varepsilon_n}, \, B_R) \leq \liminf_{n\to +\infty} \varepsilon_n^2 C R = 0,
 \]
 hence~$f({Q}_0) = 0$~a.e. or, equivalently,
 \[
  {Q}_0(x)\in\NN \qquad \textrm{for a.e. }x\in B_1 .
 \]
 
 By means of a comparison argument, we will prove that ${Q}_{\varepsilon_n}$'s actually converge \emph{strongly} in~$H^1$.
 Fatou's lemma combined with~\eqref{energy bounded} gives
 \begin{equation} \label{H1 conv 1}
  \int_{\theta R}^R \liminf_{n\to +\infty} E_{\varepsilon_n}({Q}_{\varepsilon_n}, \, \partial B_r) \, \d r 
  \leq \liminf_{n\to+\infty} E_{\varepsilon_n}({Q}_{\varepsilon_n}, \, B_R \setminus B_{\theta R}) \leq C R .
 \end{equation}
 Therefore, the set
 \[
  \left\{r\in (0, \, R]\colon \liminf_{n\to+\infty} E_{\varepsilon_n}({Q}_{\varepsilon_n}, \, \partial B_r) >  \frac{2C}{1 - \theta} \right\}
 \]
 must have length $\leq (1 - \theta)R/2$, otherwise~\eqref{H1 conv 1} would be violated.
 In particular, there exist a radius $r\in (\theta R, \, R]$ and a relabeled subsequence such that
 \[
  E_{\varepsilon_n}({Q}_{\varepsilon_n}, \, \partial B_r) \leq \frac{2C}{1 - \theta} .
 \]
 For ease of notation we scale the variables, setting $\epsilon_n := \varepsilon_n/r$,
 \[
  u_n(x) := {Q}_{\varepsilon_n}\left(r x\right) \qquad \textrm{and} \qquad u_*(x) := {Q}_0(rx) \qquad \textrm{for } x\in B_1 .
 \]
 The scaled maps satisfy
 \begin{gather}
  u_n \to u_* \qquad \textrm{weakly in } H^1(B_1, \, \Sz), \textrm{ strongly in } L^2(B_1, \, \Sz) \textrm{ and a.e.,} \label{H1 conv 2} \\
  u_*(x)\in \NN \qquad \textrm{for a.e. } x\in B_1, \label{H1 conv 3} \\
  E_{\epsilon_n}(u_n, \, \partial B_1) \leq C . \label{H1 conv 4} 
 \end{gather}
 By~\eqref{H1 conv 2} and the trace theorem, $u_n \rightharpoonup u_*$ weakly in $H^{1/2}(\partial B_1, \, \Sz)$ and hence, by compact embedding, strongly in $L^2(\partial B_1, \, \Sz)$.
 Moreover, by~\eqref{H1 conv 4} $u_n \rightharpoonup u_*$ weakly in $H^1(\partial B_1, \, \Sz)$, so
 \begin{equation} \label{H1 conv 5}
  \frac12 \int_{\partial B_1} \abs{\nabla u_*}^2 \, \d \H^2 \leq \limsup_{n\to+\infty} E_{\epsilon_n}(u_n, \, \partial B_r) \leq C .
 \end{equation}
 
 We are going to apply Proposition~\ref{prop:interpolation3} to interpolate between~$u_n$ and~$u_*$.
 Set~$\sigma_n := \|u_n - u_*\|_{L^2(\partial B_1)}$. Then $\sigma_n\to 0$ and
 \[
  \int_{\partial B_1} \left\{ \abs{\nabla u_n}^2 +\frac{1}{\epsilon_n}f(u_n) + \abs{\nabla u_*}^2 + \frac{\abs{u_n - u_*}^2}{\sigma_n} \right\} \d \H^2 \leq C,
 \]
 because of~\eqref{H1 conv 4},~\eqref{H1 conv 5}. Moreover, the $W^{1, \infty}$-estimate~\eqref{hp_interpolation2:2} is satisfied by Lemma~\ref{lemma:Linfty}.
 Thus, Proposition~\ref{prop:interpolation3} applies. We find a positive sequence $\nu_n \to 0$ and functions $\varphi_n\in H^1(B_1 \setminus B_{1 - \nu_n}, \, \Sz)$ which satisfy
 \[
  \varphi_n(x) = u_n(x), \qquad \varphi_n(x - \nu_n x) = u_*(x) 
 \]
 for $\H^2$-a.e. $x\in\partial B_1$ and
 \begin{equation} \label{H1 conv 6}
  E_{\epsilon_n}(\varphi_n, \, B_1 \setminus B_{1 - \nu_n}) \leq C \nu_n .
 \end{equation}
 
 Now, let $w_*\in H^1(B_1, \, \NN)$ be a minimizing harmonic extension of ${u_*}_{| \partial B_1}$, i.e.
 \begin{equation} \label{H1 conv 7}
  \frac12 \int_{B_1} \abs{\nabla w_*}^2 \leq \frac12 \int_{B_1} \abs{\nabla w}^2 
 \end{equation}
 for any $w\in H^1(B_1, \, \NN)$ such that $w_{| \partial B_1} = {u_*}_{| \partial B_1}$.
 Such a function exists by classical results (see e.g.~\cite[Proposition~3.1]{SchoenUhlenbeck2}).
 Define~$w_n\colon B_1 \to \Sz$ by
 \[
  w_n(x) := \begin{cases}
             \varphi_n(x)                            & \textrm{if } x\in B_1 \setminus B_{1 - \nu_n} \\
             w_*\left(\dfrac{x}{1 - \nu_n}\right) & \textrm{if } x \in B_{1 - \nu_n} .
            \end{cases}
 \]
 The function $w_n$ is an admissible comparison function for $u_n$, i.e.~$w_n\in H^1(B_1, \, \Sz)$ and ${w_n}_{|\partial B_1} = {u_n}_{|\partial B_1}$. Hence,
 \[
  E_{\epsilon_n}(u_n, \, B_1) \leq E_{\epsilon_n}(w_n, \, B_1) = \frac{1 - \nu_n}{2} \int_{B_1} \abs{\nabla w_*}^2 + 
  E_{\epsilon_n}(w_n, \, B_1 \setminus B_{1 - \nu_n}) .
 \]
 When we take the limit as $n\to+\infty$, $\nu_n\to 0$ and the energy in the shell~$B_1 \setminus B_{1 - \nu_n}$ converges to~$0$, due to~\eqref{H1 conv 6}.
 Keeping~\eqref{H1 conv 2} in mind, we obtain
 \[
 \begin{split}
  \frac12 \int_{B_1} \abs{\nabla u_*}^2 &\leq \liminf_{n\to +\infty} \frac12 \int_{B_1} \abs{\nabla u_n}^2 \leq \limsup_{n\to+\infty} \frac12 \int_{B_1} \abs{\nabla u_n}^2 \\
  &\leq \limsup_{n\to+\infty} E_{\epsilon_n}(u_n, \, B_1) \leq \frac12 \int_{B_1} \abs{\nabla w_*}^2 \leq \frac12 \int_{B_1} \abs{\nabla u_*}^2 ,
  \end{split}
 \]
 where the last inequality follows by the~minimality of $w_*$,~\eqref{H1 conv 7}.
 But this implies
 \[
  \lim_{n\to +\infty} \frac12 \int_{B_1} \abs{\nabla u_n}^2 = \frac12 \int_{B_1} \abs{\nabla u_*}^2 ,
 \]
 which yields the strong $H^1$ convergence $u_n \to u_*$, as well as
 \begin{equation} \label{H1 conv 8}
  \lim_{n\to +\infty} \frac{1}{\epsilon_n} \int_{B_1} f(u_n) = 0 .
 \end{equation}
 Moreover, $u_*$ must be a minimizing harmonic map.
 
 Scaling back to~${Q}_{\varepsilon_n}$, $Q_0$, we have shown that ${Q}_{\varepsilon_n} \to {Q}_0$ strongly in $H^1(B_r, \, \Sz)$ and that ${Q}_0$ is minimizing harmonic in~$B_r$, where $r\geq \theta R$.
 In particular, the proposition holds true.
\end{proof}

Once Proposition~\ref{prop:intro-H1 conv} is established, Proposition~\ref{prop:unif conv} can be proved arguing as in Majumdar and Zarnescu's paper~\cite{MajumdarZarnescu}.
As a byproduct of the previous proof (Equation~\eqref{H1 conv 8}), we obtain the condition
\[
  \lim_{n\to+\infty} \frac{1}{\varepsilon_n^2} \int_{B_{\theta R}(x_0)} f({Q}_{\varepsilon_n}) = 0 ,
\]
which is involved in Majumdar and Zarnescu's arguments (see, in particular,~\cite[Proposition~4]{MajumdarZarnescu}).
%

\subsection{The singular set}
\label{subsect:S}

In this subection, we complete the proof of Theorem~\ref{th:convergence} by defining the singular set~$\Sl$ and showing that it is a rectifiable set of finite length.
For each $0 < \varepsilon < 1$, define the measure $\mu_\varepsilon$ by
\begin{equation} \label{mueps}
 \mu_\varepsilon(B) := \frac{E_\varepsilon(\Qe, \, B)}{\abs{\log\varepsilon}}  \qquad \textrm{for } B\in\mathscr B(\overline\Omega) .
\end{equation}
In view of our main assumption~\eqref{hp:H}, the measures $\{\mu_\varepsilon\}_{0 < \varepsilon < 1}$ have uniformly bounded mass.
Therefore, we may extract a subsequence $\varepsilon_n\searrow 0$ such that
\begin{equation} \label{mu converge}
 \mu_{\varepsilon_n} \rightharpoonup^\star \mu_0 \qquad \textrm{weakly}^\star \textrm{ in } \mathscr M(\overline\Omega) := C(\overline\Omega)^\prime.
\end{equation}
Set $\Sl := \mathrm{supp}\,\mu_0$. By definition, $\Sl$ is a closed subset of~$\overline\Omega$.
Let~$\eta$ be given by Proposition~\ref{prop:desiderata}, corresponding to the choice~$\theta = 1/2$.

\begin{lemma} \label{lemma:small eta}
Let  $x_0\in\Omega$ and $R > 0$ be such that~$\overline{B}_R(x_0)\subset\Omega$. If
\begin{equation} \label{mu eta}
 \mu_0 \left(\overline{B}_R(x_0)\right) < \eta R
\end{equation}
then
\[
 \mu_0 \left(B_{R/2}(x_0)\right) = 0 , 
\]
that is $B_{R/2}(x_0)\subseteq \Omega \setminus \Sl$.
\end{lemma}
\begin{proof}
In force of~\eqref{mu converge} and~\eqref{mu eta}, we know that
\[
 \limsup_{n\to+\infty} \frac{E_{\varepsilon_n}({Q}_{\varepsilon_n}, \, B_R(x_0))}{R \log\left(\varepsilon_n/R\right)} < \eta .
\]
In particular, the assumption~\eqref{desid: hp} is satisfied along the subsequence~$\{\varepsilon_n\}$.
Then, we can apply Proposition~\ref{prop:desiderata} with $\theta = 1/2$ and we obtain
\[
 E_{\varepsilon_n}(B_{R/2}(x_0)) \leq C R
\]
for~$n$ large enough. Due to~\eqref{mu converge}, we deduce
\[
 \mu_0 \left(B_{R/2}(x_0)\right) \leq \liminf_{n\to+\infty} \mu_{\varepsilon_n}\left(B_{R/2}(x_0)\right) = 0 . \qedhere
\]
\end{proof}

By the monotonicity formula (Lemma~\ref{lemma:monotonicity}), for any~$x\in\Omega$ the function
\[
  r\in (0, \, \dist(x, \, \partial \Omega)) \mapsto \frac{\mu_0\left(\overline{B}_r(x)\right)}{2r}                                                           
\]
is non-decreasing, so the limit
\begin{equation} \label{Theta}
 \Theta(x) := \lim_{r\to 0^+}\frac{\mu_0\left(\overline{B}_r(x)\right)}{2r} 
\end{equation}
exists. The function $\Theta$ is usually called the ($1$-dimensional) density of~$\mu_0$ (see \cite[p.~10]{Simon-GMT}).

\begin{lemma} \label{lemma:density positive}
For all $x\in\Sl\cap\Omega$, we have $\Theta(x) \geq \eta/2$.
\end{lemma}
\begin{proof}
This follows immediately by Lemma~\ref{lemma:small eta}.
Indeed, if $x\in\Sl\cap\Omega$ then for any $r > 0$ we have $\mu_0(B_r(x)) > 0$, so Lemma~\ref{lemma:small eta} implies
\[
 \frac{\mu_0(\overline{B}_{2r}(x))}{4r} \geq \frac{\eta}{2} .
\]
Passing to the limit as $r\to 0$, we conclude.
\end{proof}

The strict positivity of~$\Theta$ has remarkable consequences.

\begin{prop} \label{prop:S}
The set~$\Sl\cap\Omega$ is countably $\H^1$-rectifiable, with $\H^1(\Sl\cap\Omega) < + \infty$. Moreover, there holds
\begin{equation*} 
 (\mu_0\mres\Omega)(B) = \int_{B \cap\Sl\cap\Omega} \Theta(x) \, \d \H^1 (x)  \qquad \textrm{for all } B \in \mathscr B(\overline\Omega) .
\end{equation*}
\end{prop}
\begin{proof}
Lemma~\ref{lemma:density positive}, together with~\cite[Theorem~3.2.(i), Chapter~1]{Simon-GMT} and~\eqref{hp:H}, implies
\[
 \H^1(\Sl\cap\Omega) \leq 2\eta^{-1} \mu_0(\Omega) \leq 2\eta^{-1} M < + \infty .
\]
Moreover, since the $1$-dimensional density of $\mu_0\mres\Omega$ exists and is essentially bounded away from zero,
the support is a $\H^1$-rectifiable set and $\mu_0\mres\Omega$ is absolutely continuous with respect to $\H^1\mres(\Sl\cap\Omega)$.
This fact was proved by Moore~\cite{Moore} and is a special case of Preiss' theorem~\cite[Theorem~5.3]{Preiss}, 
which holds true for measures in $\R^n$ having positive $k$-dimensional density, for any $k\leq n$. 
Thus, there exists a positive, $\H^1$-integrable function $g\colon\Omega\to\R$ such that
\begin{equation*} 
 (\mu_0\mres\Omega)(B) = \int_{B\cap\Sl\cap\Omega} g(x) \, \d \H^1 (x)  \qquad \textrm{for all } B \in \mathscr B(\overline\Omega) .
\end{equation*}
By Besicovitch differentiation theorem, there holds
\begin{equation*} 
 \lim_{r \to 0^+} \frac{\mu_0(\overline{B}_r(x))}{\H^1(B_r(x)\cap\Sl)} = g(x) \qquad \textrm{for } \H^1\textrm{-a.e. } x\in\Sl\cap\Omega .
\end{equation*}
On the other hand, because~$\Sl\cap\Omega$ is rectifiable and has finite length, \cite[Theorem~3.2.19]{Federer} implies that
\begin{equation*} 
 \lim_{r\to 0^+} \frac{\H^1(B_r(x)\cap\Sl)}{2r} = 1 \qquad \textrm{for } \H^1\textrm{-a.e. } x\in\Sl\cap\Omega .
\end{equation*}
Combining these facts with~\eqref{Theta}, we obtain that~$\Theta = g$ $\H^1$-a.e. on~$\Sl\cap\Omega$, so the proposition follows.
\end{proof}

To complete the proof of Theorem~\ref{th:convergence}, we check that ${Q}_{\varepsilon_n}$ locally converge to a harmonic map, away from~$\Sl$.

\begin{prop} \label{prop:loc_conv}
 There exists a map ${Q}_0\in H^1_{\mathrm{loc}}(\Omega\setminus \Sl, \, \NN)$ such that, up to a relabeled subsequence,
 \[
  {Q}_{\varepsilon_n} \to {Q}_0 \qquad \textrm{strongly in } H^1_{\mathrm{loc}}(\Omega\setminus \Sl, \, \Sz) .
 \]
The map ${Q}_0$ is minimizing harmonic on every ball $B\csubset \Omega\setminus \Sl$.
Moreover, there exists a locally finite set $\Spt\subseteq\Omega\setminus\Sl$ such that ${Q}_0$ is of class~$C^\infty$ on~$\Omega\setminus(\Sl\cup\Spt)$, and
\[
  {Q}_{\varepsilon_n} \to {Q}_0 \qquad \textrm{locally uniformly in } \Omega\setminus(\Sl\cup\Spt) .
\]
\end{prop}
\begin{proof}
Fix an open subset~$K\csubset\Omega\setminus\Sl$.
Combining Proposition~\ref{prop:desiderata} with a standard covering argument, we deduce that minimizers~$Q_\varepsilon$ 
satisfy~$E_\varepsilon(Q_\varepsilon, \, K) \leq C= C(K)$, therefore they are weakly compact in~$H^1(K, \, \Sz)$.
It follows from Proposition~\ref{prop:intro-H1 conv} that the convergence is strong, 
and any limit map~$Q_0$ is locally minimizing harmonic.
Then, 
on each ball $B\csubset \Omega \setminus \Sl$ there exists a finite set~$X_B\subseteq B$ such that ${Q}_0\in C^\infty(B \setminus X_B, \, \Sz)$,
because of~\cite[Theorem~II]{SchoenUhlenbeck}. 
Therefore ${Q}_0 \in C^\infty(\Omega\setminus\Sl\cup\Spt)$, where~$\Spt := \cup_B X_B$ is locally finite in~$\Omega\setminus\Sl$.
The locally uniform convergence ${Q}_{\varepsilon_n} \to {Q}_0$ on~$\Omega\setminus (\Sl\cup\Spt)$ 
follows from Proposition~\ref{prop:unif conv} and a covering argument.
\end{proof}

 \subsection{The analysis near the boundary}
 \label{subsect:boundary}
 
 Proposition~\ref{prop:desiderata}, which is the key step in the proof of our main theorem, has been proven on balls contained in the domain.
 In this subsection, we aim at proving a similar result in case the ball intersects the boundary of~$\Omega$.
 For this purpose, we need an additional assumption on the behaviour of the boundary datum.
 Let~$\Gamma$ be a relatively open subset of~$\partial\Omega$.
 We assume that
 \begin{enumerate}[label = (H$_\Gamma$), ref = H$_\Gamma$]
  \item \label{hp:H_Gamma} For any~$0 < \varepsilon < 1$, there holds $g_\varepsilon\in (H^1_{\mathrm{loc}}\cap L^\infty_{\mathrm{loc}})(\Gamma, \, \Sz)$.
  Moreover, for any~$K\csubset\Gamma$ there exists a constant~$C_K$ such that
  \[
   E_\varepsilon(g_\varepsilon, \, K) \leq C_K \qquad \textrm{and} \qquad \norm{g_\varepsilon}_{L^\infty(K)} \leq C_K
  \]
 for any~$0 < \varepsilon < 1$.
 \end{enumerate}
 For instance, the families of boundary data given by~\eqref{discl_bord} and~\eqref{discl_bord_reg} satisfies Condition~\eqref{hp:H_Gamma} on~$\Gamma := \partial\Omega\setminus\Sigma$.
  
 \begin{prop} \label{prop:desiderata_bd}
  Assume that the conditions~\eqref{hp:H} and~\eqref{hp:H_Gamma} hold.
  For any $0 < \theta < 1$ there exist positive numbers~$\eta$,~$\epsilon_0$ and~$C$ such that, 
  for any~$x_0\in\overline\Omega$, $R >0$ satisfying $\overline B_R(x_0)\cap\partial\Omega \subseteq \Gamma$ and any~$0 < \varepsilon \leq \epsilon_0 R$, if
 \begin{equation} \label{desid_bd: hp}
  E_\varepsilon(\Qe, \, B_R(x_0) \cap \Omega) \leq \eta R \log\frac{R}{\varepsilon}
 \end{equation}
 then
 \begin{equation*} 
   E_\varepsilon(\Qe, \, B_{\theta R}(x_0) \cap \Omega) \leq C R .
 \end{equation*}
 \end{prop}
 
 By a standard covering argument, 
 we see that Proposition~\ref{prop:desiderata_bd} implies the weak compactness of minimizers up to the boundary.
 More precisely, we have
 
 \begin{cor} \label{cor:convergence_bd}
 Let~$\Gamma$ be a relatively open subset of~$\partial\Omega$. 
 Assume that the conditions~\eqref{hp:H} and~\eqref{hp:H_Gamma} are satisfied.
 Then, there exist a subsequence $\varepsilon_n\searrow 0$, a closed set $\Sl\subseteq\overline\Omega$ and a map~${Q}_0\in H^1_{\mathrm{loc}}((\Omega\cup\Gamma)\setminus\Sl, \, \NN)$ 
 which satisfy \ref{th:first}--\ref{th:last} in Theorem~\ref{th:convergence} and 
 \[
  Q_{\varepsilon_n} \rightharpoonup Q_0 \qquad \textrm{weakly in } H^1_{\mathrm{loc}}((\Omega\cup\Gamma)\setminus\Sl, \, \Sz) .
 \]
 \end{cor}
 The set~$\Sl$ is again defined as the support of the measure~$\mu_0$, where~$\mu_0$ is a weak$^\star$ limit of~$\{\mu_\varepsilon\}_{0 < \varepsilon < 1}$ in~$C(\overline\Omega)^\prime$ 
 and the~$\mu_\varepsilon$'s are given by~\eqref{mueps}.
 The proofs in Subsection~\ref{subsect:S} remain unchanged.
 We cannot expect strong~$H^1$ convergence of minimizers up to the boundary, unless some additional assumption on the boundary datum is made.
 Moreover, the intersection~$\Sl\cap\Gamma$ may be non-empty. An example is given in Section~\ref{subsect:torus}, Proposition~\ref{prop:torus}.
 
 \begin{proof}[Proof of Proposition~\ref{prop:desiderata_bd}]
  For the sake of simplicity, we assume that~$x_0 = 0$ and set~$F_\varepsilon(r) := E_\varepsilon(Q_\varepsilon, \, B_r\cap\Omega)$ for~$0 < r < R$.
  The coarea formula implies
  \[
   F_\varepsilon(r) = \int_0^r E_\varepsilon(Q_\varepsilon, \, \partial B_s \cap \Omega) \, \d s \qquad \textrm{for } 0 < r < R,
  \]
  so~$F_\varepsilon^\prime(r) = E_\varepsilon(Q_\varepsilon, \, \partial B_r \cap \Omega)$ for a.e.~$0 < r < R$. Define the set
  \[
   \tilde D^\varepsilon := \left\{r\in(\theta R, \, R) \colon F_\varepsilon^\prime (r) \leq \frac{2\eta}{1 - \theta} \log\frac{R}{\varepsilon} \right\} .
  \]
  The assumption~\eqref{desid_bd: hp} and an average argument give
  \begin{equation} \label{desid_bd1}
   \H^1(\tilde D^\varepsilon) \geq \frac{(1 - \theta) R}{2}.
  \end{equation}
  On the other hand, for any radius~$r\in \tilde D^\varepsilon$ we have
  \[
   \begin{split}
    E_\varepsilon(Q_\varepsilon, \, \partial(B_r \cap \Omega)) = F_\varepsilon^\prime(r) + E_\varepsilon(Q_\varepsilon, \, B_r \cap \partial\Omega) 
    \stackrel{\eqref{hp:H_Gamma}}{\leq} \frac{2\eta}{1 - \theta} \left(\log\frac{r}{\varepsilon} - \log\theta \right) + C ,
   \end{split}
  \]
  where~$C$ is a constant depending on~$x_0$ and~$R$.
  Therefore, by choosing~$\eta$ small enough we obtain
  \[
   E_\varepsilon(Q_\varepsilon, \, \partial(B_r \cap \Omega)) \leq \eta_0 \log \frac{r}{\varepsilon} \qquad \textrm{for } 0 < \varepsilon \leq \epsilon_0 R,
  \]
  where~$\eta_0$ and~$\epsilon_1$ are given by Proposition~\ref{prop:interpolation2}.
  With the help of this estimate, and since~$B_r\cap\Omega$ is bilipschitz equivalent to a ball, we can repeat the proof of Lemma~\ref{lemma:HKL ineq}. We deduce that
  \[
   F_\varepsilon(r) \leq CR \left( E_\varepsilon(Q_\varepsilon, \, \partial(B_r \cap \Omega))^{1/2} + 1 \right) 
   \qquad \textrm{for any } r\in \tilde D^\varepsilon \textrm{ and } 0 < \varepsilon \leq \epsilon_0 R .
  \]
  Then, using the elementary inequality $(a + b)^{1/2} \leq a^{1/2} + b^{1/2}$ and~\eqref{hp:H_Gamma} again, we infer
  \begin{equation} \label{desid_bd2}
   F_\varepsilon(r) \leq CR\left\{ \left(F_\varepsilon^\prime(r) + E_\varepsilon(Q_\varepsilon, \, B_r \cap \partial\Omega)\right)^{1/2} + 1 \right\}
   \leq C R \left( F_\varepsilon^\prime(r)^{1/2} + 1\right)
  \end{equation}
  for any~$r\in \tilde D^\varepsilon$ and~$0 < \varepsilon \leq \epsilon_0 R$.
  Thanks to~\eqref{desid_bd1} and~\eqref{desid_bd2}, we can apply Lemma~\ref{lemma:ODE} to~$y := F_\varepsilon$. This yields the conclusion of the proof.
 \end{proof}
 
\section{Structure of the singular set: Proof of Proposition~\ref{prop:intro-S}}
\label{sect:S}

\subsection{The limit measure is a stationary varifold}
\label{subsect:stationary}
 
The aim of this section is to prove Proposition~\ref{prop:intro-S}.
We start by showing that $\mu_0\mres\Omega$ is a stationary varifold.
These objects, introduced by Almgren~\cite{Almgren66}, can be thought as weak counterparts of manifolds with vanishing mean curvature.
For more details, the reader is referred to the paper by Allard~\cite{Allard} or the book by Simon~\cite{Simon-GMT}.


Before stating the following proposition, let us recall some basic facts.
The rectifiability of~$\mu_0\mres\Omega$ (Proposition~\ref{prop:S}), together with \cite[Remarks~1.9 and~11.5, Theorem~11.6]{Simon-GMT},
implies that for $\mu_0$-a.e. $x\in\Omega$ there exists a unique $1$-dimensional subspace $L_x\subseteq\R^n$ such that
\begin{equation} \label{tangent_line}
 \lim_{\lambda \to 0} \int_{\R^d} \lambda^{-1}\varphi\left(\frac{z - x}{\lambda}\right) \, \d\mu_0(z) 
 = \Theta(x) \int_{L_x} \varphi(y) \, \d \H^1(y) \qquad \textrm{for all } \varphi\in C_c (\R^3) . 
\end{equation}
Such line is called the \emph{approximate tangent line} of~$\mu_0$ at~$x$, and noted~$\mathrm{Tan}(\mu_0, \, x)$.
Now, let~$\mathbf G_{1, 3}\subseteq \Mat$ be the set of matrices representing orthogonal projections on $1$-subspaces of~$\R^3$.
Let~$A(x)\in \mathbf{G}_{1,3}$ denote the orthogonal projection on~$\mathrm{Tan}(\mu_0, \, x)$, for a.e.~$x\in\Omega$.
A varifold is a Radon measure on~$\Omega\times\mathbf{G}_{1,3}$.
The varifold associated with~$\mu_0\mres\Omega$ is defined as the push-forward measure~$\mathbf V_0 := (\Id, \, A)_\#(\mu_0\mres\Omega)$, i.e.
the measure~$\mathbf V_0 \in \mathscr{M}(\Omega\times\mathbf G_{1, 3})$ given by
\begin{equation*} 
 \mathbf{V}_0(E) := \mu_0\left\{x\in\Omega\colon (x, \, A(x))\in E \right\} \qquad \textrm{for any Borel set } E\subseteq\Omega\times\mathbf{G}_{1, 3}.
\end{equation*}
The varifold~$\mathbf{V}_0$ is stationary (see~\cite[\Stipografico~4.2]{Allard}) if and only if there holds
\begin{equation} \label{Sstationary}
  \int_\Omega A_{ij}(x) \frac{\partial \X_i}{\partial x_j}(x) \, \d \mu_0(x) = 0 \qquad \textrm{for any } \X\in C^1_c(\Omega, \, \R^3).
\end{equation}

\begin{prop} \label{prop:S stationary}
 The varifold~$\mathbf{V}_0$ associated with~$\mu_0\mres\Omega$ is stationary. 
\end{prop}

\begin{proof}
 The proposition follows by adapting Ambrosio and Soner's analysis in~\cite{AmbrosioSoner}.
 For the convenience of the reader, we give here the proof.
 Define the matrix-valued map $A^\varepsilon = (A^\varepsilon_{ij})_{i,j}\colon \Omega\to\Mat$ by
\[
 A^\varepsilon_{ij} := \frac{1}{\abs{\log\varepsilon}}\left(e_\varepsilon(\Qe)\delta_{ij} - \frac{\partial \Qe}{\partial x_i} \cdot 
 \frac{\partial \Qe}{\partial x_j}\right) \qquad \textrm{for } i,j\in\{1, \, 2, \, 3\}.
\]
Then $A^\varepsilon$ is a symmetric matrix, such that
\begin{equation} \label{Sstat 1}
 \tr A^\varepsilon = \frac{1}{\abs{\log\varepsilon}}\left(3e_\varepsilon(\Qe) - \abs{\nabla \Qe}^2\right) \geq \mu_\varepsilon
\end{equation}
and
\begin{equation} \label{Sstat 2}
 \abs{A^\varepsilon} \leq C \mu_\varepsilon .
\end{equation}
For any vector $v\in\S^2$, there holds
\begin{equation} \label{Sstat 3}
 A^\varepsilon_{ij}v_i v_j = \frac{1}{\abs{\log\varepsilon}}\left(e_\varepsilon(\Qe) - \abs{v_i\frac{\partial \Qe}{\partial x_i}}^2 \right) \leq \mu_\varepsilon,
\end{equation}
so the eigenvalues of $A^\varepsilon$ are less or equal than $\mu_\varepsilon$.
Moreover, by integrating by parts the stress-energy identity (Lemma~\ref{lemma:stress-energy}) we obtain
 \begin{equation} \label{Sstat 4}
  \int_\Omega A^\varepsilon_{ij}(x) \frac{\partial \X_i}{\partial x_j}(x) \, \d x = 0 \qquad \textrm{for any } \X\in C^1_c(\Omega, \, \R^3).
 \end{equation}
In view of~\eqref{Sstat 2}, and extracting a subsequence if necessary, we have that $A^\varepsilon \rightharpoonup^\star A^0$ in the \mbox{weak-$\star$} topology 
of~$\mathscr M(\Omega, \, \Mat) := C_c(\Omega, \, \Mat)^\prime$.
The limit measure~$A^0$ satisfies $|A^0| \leq C (\mu_0\mres\Omega)$, in particular is absolutely continuous with respect to~$\mu_0\mres\Omega$.
Therefore, there exists a matrix-valued function $A^*\in L^1(\Omega, \, \mu_0; \, \Mat)$ such that
\[
 \d A^0 = A^*(x) \, \d(\mu_0\mres\Omega) \qquad \textrm{as measures in } \mathscr M(\Omega, \, \Mat) .
\]
Passing to the limit in~\eqref{Sstat 1},~\eqref{Sstat 3} and~\eqref{Sstat 4}, for $\mu_0$-a.e. $x$ we obtain that $A^*(x)$ is a symmetric matrix,
with $\tr A^*(x) \geq 1$ and eigenvalues less or equal than $1$, such that
\begin{equation} \label{Sstat 5}
 \int_\Omega A^*_{ij}(x) \frac{\partial \X_i}{\partial x_j}(x) \, \d \mu_0(x) = 0 \qquad \textrm{for any } \X\in C^1_c(\Omega, \, \R^3).
\end{equation}
Now, fix a Lebesgue point~$x$ for~$A^*$ (with respect to~$\mu_0$) and $0 < \lambda < \dist(x, \, \partial\Omega)$. Condition~\eqref{Sstat 5} implies
\begin{equation} \label{Sstat 7}
 \lambda^{-1} \int_{\R^3} A^*(z) \cdot \nabla \X \left(\frac{z - x}{\lambda}\right) \, \d\mu_0(z) = 0 \qquad \textrm{for any } \X\in C^1_c(B_1, \, \R^3).
\end{equation}
Then,
\begin{equation*}
 \begin{split}
  &\abs{\lambda^{-1}\int_{\R^3} \left(A^*(z) - A^*(x)\right) \cdot \nabla \X \left(\frac{z - x}{\lambda}\right) \, \d\mu_0(z)} \\
  &\qquad\qquad\qquad\leq \underbrace{\frac{\mu_0(\overline B_\lambda(x))}{\lambda}}_{\to\Theta(x)/2} \norm{\nabla \X}_{L^\infty(B_1)} \fint_{B_\lambda(x)} \abs{A^*(z) - A^*(x)} \, \d\mu_0(z) \to 0
 \end{split}
\end{equation*}
as~$\lambda\to 0$. Combined with~\eqref{tangent_line} and~\eqref{Sstat 7}, this provides
\[
 \Theta(x) A^*(x) \cdot \int_{\mathrm{Tan}(\mu_0, x)} \nabla \X \, \d \H^1 = \lim_{\lambda\to 0}\lambda^{-1} \int_{\R^3} A^*(x) \cdot \nabla \X\left(\frac{z - x}{\lambda}\right) \d \mu_0(x) = 0 
\]
for any~$\X\in C_c^1(B_1, \, \R^3)$.
Since $\Theta(x) > 0$ by Lemma~\ref{lemma:density positive}, applying \cite[Lemma~3.9]{AmbrosioSoner} 
(with $\beta = s = 1$ and $\nu = \frac12 \H^1 \mres \mathrm{Tan}(\mu_0, x)$) we deduce that at least two eigenvalues of~$A^*(x)$ vanish, for $\mu_0$-a.e. $x$.
On the other hand, we know already that $\tr A^*(x) = 1$ with eigenvalues $\leq 1$. 
Therefore, the eigenvalues of $A^*(x)$ 
are~$(1, \, 0, \, 0)$ and~$A^*(x)$ represents the orthogonal projection on a line.

The push-forward measure $\mathbf V := (\Id, \, A^*)_\#(\mu_0\mres\Omega)$ 
is a varifold, and~\eqref{Sstat 5} means that $\mathbf V$ is stationary.
A classical result by Allard (see~\cite[Rectifiability Theorem, \Stipografico~5.5]{Allard} or \cite[Theorem~3.3]{AmbrosioSoner}) 
asserts that every varifold with locally bounded first variation and positive density is rectifiable.
In our case, $\mathbf V$ has vanishing first variation, and the density is bounded from below by Lemma~\ref{lemma:density positive}.
Therefore, by Allard's theorem $\mathbf V$ is rectifiable. 
In particular $A^*(x)$ is the orthogonal projection on~$\mathrm{Tan}(\Sl, \, x)$ for $\mu_0$-a.e. $x\in\Omega$, so~$\mathbf{V} = \mathbf{V}_0$ and the proposition follows.
\end{proof}

\begin{remark} \label{remark:not_stationary_bd}
 In general, we cannot expect that~$\mu_0$ is associated with a stationary varifold, i.e. stationarity may fail on the boundary of the domain (see Section~\ref{subsect:torus}).
 Indeed, stationarity is deduced by taking the limit in the Euler-Lagrange system associated to the energy, and such a system is not satisfied on the boundary.
\end{remark}

Stationary varifolds of dimension~$1$ are essentially the sum of straight line segments~(see Allard and Almgren,~\cite{AllardAlmgren}). However, the sum can be locally infinite. 
To rule out this possibility, in the rest of the section we prove that the $1$-dimensional density of~$\mu_0\mres\Omega$ is constant a.e.
As a consequence, we obtain that~$\Sl\cap\Omega$ is essentially a locally finite union of line segments~\cite[Theorem p.~89]{AllardAlmgren}.
In order to compute the density of~$\mu_0\mres\Omega$, we apply an argument by Lin and Rivi\`ere (see~\cite[Section~III.1]{LinRiviere}).
Essentially, by scaling we reduce to an auxialiary problem defined on a cylinder, for which we prove refined energy estimates.
This requires, once again, interpolation and extension arguments.
For the convenience of the reader, we work out this argument, which is sketched in~\cite{LinRiviere}, in detail.

\subsection{An auxiliary problem: energy bounds on a cylinder}
\label{subsect:cylinder}

We consider the following auxiliary problem.
Given some (small) parameters~$0 < \delta, \, \epsilon < 1$, we consider the closed cylinder~$\Lambda_\delta := \bar{B}_\delta^2\times [-1, \, 1]$
with lateral surface~$\Gamma_\delta := \partial B^2_\delta \times [-1, \, 1]$.
Let~$g_{\delta,\epsilon}\in H^1(\partial\Lambda_\delta, \, \Sz)$ be a boundary datum which satisfies the following conditions:
\begin{gather}
 \|g_{\delta,\epsilon}\|_{L^\infty(\Lambda_\delta)} \leq M \label{hp_cylinder:Linfty} \\
 E_\epsilon(g_{\delta,\epsilon}, \, B^2_\delta\times\{-1, \, 1\}) \leq M \log\frac{\delta}{\epsilon} \label{hp_cylinder:log} \\
 E_\epsilon(g_{\delta,\epsilon}, \, \Gamma_\delta) \leq \eta \log\frac{\delta}{\epsilon}, \label{hp_cylinder:small_log}
\end{gather} 
for some positive constants~$M$ and~$\eta$.
Let~$u_{\delta,\epsilon}$ be a minimizer of the Landau-de Gennes energy~\eqref{energy} in the class~$H^1_{g_{\delta,\epsilon}}(\Lambda_\delta, \, \Sz)$.

\begin{lemma} \label{lemma:cylinder}
 For any~$M > 0$, there exists~$\eta_0 > 0$ and for any~$0 < \eta \leq \eta_0$, \mbox{$0< \delta < 1$}
 there exist positive numbers~$\epsilon_0$, $C$ and~$\alpha(M, \, \eta, \, \delta)$ with the following properties.
 If~$0 < \epsilon \leq \epsilon_0$ and~$g_{\delta,\epsilon}$ satisfies~\eqref{hp_cylinder:Linfty}--\eqref{hp_cylinder:small_log}, then either
 \begin{equation} \label{cylinder:trivial}
  E_\epsilon(u_{\delta,\epsilon}, \, \Lambda_\delta) \leq \alpha(M, \, \eta, \, \delta) \log\frac{\delta}{\epsilon}
 \end{equation}
 or
 \begin{equation} \label{cylinder:nontrivial}
   \left(2 \kappa_* - \alpha(M, \, \eta, \, \delta)\right) \log\frac{\delta}{\epsilon} - C \leq E_\epsilon(u_{\delta,\epsilon}, \, \Lambda_\delta)
   \leq \left(2 \kappa_* + \alpha(M, \, \eta, \, \delta)\right) \log\frac{\delta}{\epsilon} + C.
 \end{equation}
 Moreover, we can choose the number~$\alpha(M, \, \eta, \, \delta)$ in such a way that
 \[
  \alpha(M, \, \eta, \, \delta) \leq C \left(\delta M + \delta^2\eta + \delta\eta + \eta + \delta^{-1}\eta \right).
 \]
\end{lemma}

%

Again, the key step in the proof is to approximate $u_{\delta,\epsilon}$ with an~$\NN$-valued map, defined on the later surface of the cylinder.
This is possible, because the energy on~$\Gamma_\delta$ is small compared to~$|\log\epsilon|$, by~\eqref{hp_cylinder:small_log}.
Set~$h(\epsilon) := \epsilon^{1/2}|\log\epsilon|$ and
\[
 A_{\delta,\epsilon} := \left(\bar{B}^2_\delta \setminus B^2_{\delta - \delta h(\epsilon)}\right) \times [-1, \, 1], 
 \qquad D_{\delta,\epsilon} := \partial B^2_{\delta - \delta h(\epsilon)} \times [-1, 1] . 
\]
Then, by arguing exactly as in the proof of Proposition~\ref{prop:interpolation2}, we obtain

\begin{lemma} \label{lemma:interpolation}
 For any $M > 0$, there exist positive numbers $\eta_0$ and~$C$ and, for any $0 < \eta \leq \eta_0$ and $0< \delta < 1$,
 there exists~$\epsilon_0 > 0$ with the following property.
 If $0 < \epsilon \leq \epsilon_0$ and~$g_{\delta,\epsilon}$ satisfies \eqref{hp_cylinder:Linfty}--\eqref{hp_cylinder:small_log},
 then there exist maps~$v_{\delta, \epsilon}\in H^1(D_{\delta,\epsilon}, \, \NN)$ and $\varphi_{\delta, \epsilon}\in H^1(A_{\delta,\epsilon}, \, \Sz)$ which satisfy
 \begin{gather}
  \varphi_{\delta, \epsilon} = g_{\delta,\epsilon} \ \H^2\textrm{-a.e. on } \partial A_{\delta, \epsilon}\setminus D_{\delta,\epsilon}, 
  \qquad \varphi_{\delta, \epsilon} = v_{\delta, \epsilon} \ \H^2\textrm{-a.e. on } D_{\delta,\epsilon} \label{phi_bd} \\
  \frac12 \int_{D_{\delta,\epsilon}} \abs{\nabla v_{\delta, \epsilon}}^2 \, \d \H^2 \leq C \eta \log\frac{\delta}{\epsilon}, \label{v_energy} \\
  E_\epsilon(\varphi_{\delta, \epsilon}, \, A_{\delta,\epsilon}) \leq C \eta\, h(\epsilon) \log\frac{\delta}{\epsilon} . \label{phi_energy}
 \end{gather}
\end{lemma}

\begin{proof}[Sketch of the proof]
 In Proposition~\ref{prop:interpolation2}, the datum~$g_{\delta,\epsilon}$ is defined on a sphere of fixed radius.
 Here, in constrast, the domain is the lateral surface of a cylinder of variable radius~$\delta$.
 To overcome these issues, we first rescale the domain so that we work in the cylinder~$B^2_1\times [-\delta^{-1}, \, \delta^{-1}]$.
 Then, we construct a good grid of size~$h(\epsilon)$, in the sense of Definition~\ref{def:good grid}.
 By an average argument, the constant~$C_{\mathscr{G}}$ in Definition~\ref{def:good grid} behaves as~$\mathrm{O}(\delta)$, and in particular is uniformly bounded.
 Despite the geometry is different, we can repeat the proof of Proposition~\ref{prop:interpolation2} because the construction used in the proof is local, that is,
 the behaviour of~$v_{\delta, \epsilon}$ and~$\varphi_{\delta, \epsilon}$ on a cell of the grid only depends on the behaviour of~$u_{\delta,\epsilon}$ on the same cell.
 By Remark~\ref{remark:constants2}, the constants in~\eqref{v_energy}--\eqref{phi_energy} only depend on the shape of a given cell (i.e.~on~$C_{\mathscr{G}}$), not on the size of the whole domain.
 Therefore, they are uniformly bounded with respect to~$\delta$.
\end{proof}

Since~$v_{\delta, \epsilon}$ is an~$\NN$-valued~$H^1$-map, it is possible to define its homotopy class (see Lemma~\ref{lemma:extension4}).
If such homotopy class is trivial, there is no topological obstruction, therefore the energy of a minimizer 
is small compared to~$|\log\epsilon|$, that is, the upper bound~\eqref{cylinder:trivial} holds. 
Otherwise, we prove that~\eqref{cylinder:nontrivial} holds.
The lower bound follows by the Jerrard-Sandier type estimate (Corollary~\ref{cor:lower bound}), whereas the upper bound is obtained via a comparison argument.

\subsubsection*{The homotopy class of $v_{\delta, \epsilon}$ is trivial: proof of~\eqref{cylinder:trivial}}

We assume now that the homotopy class of~$v_{\delta, \epsilon}$ is trivial, and we construct a competitor which satisfies the energy bound~\eqref{cylinder:trivial}.
Using the properties~\eqref{v_energy},~\eqref{phi_energy} and a comparison argument, we find~$z_- \in (-1 + \delta, \, -1 + 2\delta)$ and~$z_+ \in (1 -2 \delta, \, 1 - \delta)$ such that
\begin{equation} \label{trivial1}
 \frac12 \int_{\partial B^2_{\delta - \delta h(\epsilon)}\times\{z_-, \, z_+\}} \abs{\nabla v_{\delta, \epsilon}}^2 \,\d\H^1 \leq \frac{C \eta}{\delta} \log\frac{\delta}{\epsilon}
\end{equation}
and
\begin{equation} \label{trivial2}
 E_\epsilon(\varphi_{\delta, \epsilon}, \, (B^2_\delta \setminus B^2_{\delta - \delta h(\epsilon)}) \times\{z_-, \, z_+\}) \leq \frac{C \eta}{\delta} \, h(\epsilon) \log\frac{\delta}{\epsilon} .
\end{equation}
Since the homotopy class of~$v_{\delta, \epsilon}$ is trivial, with the help of~\eqref{trivial1} and of Lemma~\ref{lemma:extension3} 
we find a map~$w_{\delta,\epsilon}\in H^1(B^2_{\delta - \delta h(\epsilon)}\times \{z_-, \, z_+\}, \, \NN)$ such that
\begin{equation} \label{trivial3}
 \frac12 \int_{B^2_{\delta - \delta h(\epsilon)}\times\{z_-, \, z_+\}} \abs{\nabla w_{\delta,\epsilon}}^2 \,\d\H^2 \leq C \eta \log\frac{\delta}{\epsilon} .
\end{equation}
We consider now four subdomains:
\[
 \Lambda_\delta^- := B^2_\delta \times (-1, \, z_-), \qquad \Lambda_{\delta,\epsilon}^0 := B^2_{\delta - \delta h(\epsilon)}\times (z_-, \, z_+),
 \qquad \Lambda_\delta^+ := B^2_\delta \times (z_+, \, 1)
\]
and
\[
 A^\prime_{\delta, \epsilon} := (B^2_\delta\setminus B^2_{\delta - \delta h(\epsilon)})\times [z_-, \, z_+].
\]
We are going to apply Lemma~\ref{lemma:3Dextension} to~$\Lambda_\delta^-$, $\Lambda^0_{\delta,\epsilon}$ and~$\Lambda_\delta^+$
(these subdomains are convex, so they are star-shaped with respect to each of their points).
We first consider~$\Lambda_\delta^+$, and we assign the boundary datum
\[
 g_{\delta,\epsilon}^+ := \begin{cases}
                      g_{\delta,\epsilon} & \textrm{on } \partial\Lambda_\delta^+ \cap \partial\Lambda_\delta \\
                      \varphi_{\delta, \epsilon} & \textrm{on } (B^2_\delta \setminus B^2_{\delta - \delta h(\epsilon)}) \times\{z_+\} \\
                      w_{\delta,\epsilon} & \textrm{on } B^2_{\delta - \delta h(\epsilon)} \times\{z_+\}.
                     \end{cases}
\]
Let~$u_{\delta,\epsilon}^+$ be a minimizer of~\eqref{energy} on~$\Lambda_\delta^+$, subject to the boundary condition~$u = g_{\delta,\epsilon}^+$ on~$\partial\Lambda_\delta^+$.
By applying Lemma~\ref{lemma:3Dextension} and~\eqref{hp_cylinder:log},~\eqref{hp_cylinder:small_log},~\eqref{trivial2},~\eqref{trivial3}, we obtain
\begin{equation} \label{trivial4}
 E_\epsilon(u_{\delta,\epsilon}^+, \, \Lambda_\delta^+) \leq C \delta \, E_\epsilon(g_{\delta,\epsilon}^+, \, \partial\Lambda_\delta^+) 
 \leq C\left(\delta M + \delta\eta + \eta \, h(\epsilon)\right) \log\frac{\delta}{\epsilon}.
\end{equation}
We define a function~$u_{\delta,\epsilon}^-$ in~$\Lambda_\delta^-$ a similar way. Finally, in~$\Lambda_{\delta,\epsilon}^0$ we consider the boundary datum
\[
 g_{\delta,\epsilon}^0 := \begin{cases}
                      v_{\delta, \epsilon} & \textrm{on } \partial B^2_{\delta - \delta h(\epsilon)} \times (z_-, \, z_+) \\
                      w_{\delta,\epsilon} & \textrm{on } B^2_{\delta - \delta h(\epsilon)} \times\{z_-, \, z_+\}.
                     \end{cases}
\]
and denote by~$u_{\delta,\epsilon}^0$ the corresponding minimizer.
By applying Lemma~\ref{lemma:3Dextension},~\eqref{v_energy} and~\eqref{trivial3}, we deduce
\begin{equation} \label{trivial5}
 E_\epsilon(u_{\delta,\epsilon}^0, \, \Lambda_{\delta,\epsilon}^0)
 \leq C E_\epsilon(g_{\delta,\epsilon}^0, \, \partial\Lambda_{\delta,\epsilon}^0) \leq C \eta \log\frac{\delta}{\epsilon}.
\end{equation}
The boundary conditions we have defined on the boundaries of~$\Lambda_\delta^+$, $\Lambda_{\delta,\epsilon}^0$, $\Lambda_\delta^-$ and~$A^\prime_{\delta,\epsilon}$ match.
Therefore, we can define an admissible comparison map by pasting $u_{\delta,\epsilon}^+$, $u_{\delta,\epsilon}^0$, $u_{\delta,\epsilon}^-$
and~$\varphi_{\delta,\epsilon}$ restricted to~$A^\prime_{\delta,\epsilon}$. 
Combining~\eqref{phi_energy}, \eqref{trivial4} and~\eqref{trivial5}, we obtain
\[
 E_\epsilon(u_{\delta,\epsilon}, \, \Lambda_\delta) \leq C\left(\delta M + \delta\eta + \eta \, h(\epsilon)\right) \log\frac{\delta}{\epsilon}.
\]
Since~$h(\epsilon) = \epsilon^{1/2}|\log\epsilon|\leq 2e^{-1}$ for~$0 < \epsilon < 1$, we conclude that~\eqref{cylinder:trivial} holds if
\begin{equation} \label{trivial6}
 \alpha(M, \, \eta, \, \delta) \geq C_1\left(\delta M + \delta\eta + \eta \right),
\end{equation}
for some universal constant~$C_1$.

\subsubsection*{The homotopy class of $v_{\delta, \epsilon}$ is non-trivial: proof of the upper bound in~\eqref{cylinder:nontrivial}}

We suppose now that the homotopy class of~$v_{\delta, \epsilon}$ is non-trivial and we prove the upper bound in~\eqref{cylinder:nontrivial}, again by a comparison argument.
To construct the comparison map, we consider the same decomposition of~$\Lambda_\delta$ into four subdomains as before.
We first construct a map~$u^0_{\delta, \epsilon}\in H^1(\Lambda^0_{\delta,\epsilon}, \, \Sz)$ by applying Lemma~\ref{lemma:extension4} with the choice~$g = v_{\delta, \epsilon}$.
Thanks to~\eqref{v_energy}, we obtain that
\begin{equation} \label{up_nontrivial1}
 E_\epsilon(u_{\delta,\epsilon}^0, \, \Lambda_{\delta,\epsilon}^0) \leq \left(2\kappa_* + C(\delta^{-1} + \delta) \eta \right) \log\frac{\delta}{\epsilon} + C
\end{equation}
and
\begin{equation} \label{up_nontrivial2}
 E_\epsilon(u_{\delta,\epsilon}^0, \, B^2_{\delta - h(\epsilon)\delta}\times\{z_-, \, z_+\}) \leq \left(2\kappa_* + C(\delta^{-1} + \delta) \eta \right) \log\frac{\delta}{\epsilon} + C
\end{equation}
Next, we consider~$\Lambda_\delta^+$ and we assign the boundary datum
\[
 g_{\delta,\epsilon}^+ := \begin{cases}
                      g_{\delta,\epsilon} & \textrm{on } \partial\Lambda_\delta^+ \cap \partial\Lambda_\delta \\
                      \varphi_{\delta, \epsilon} & \textrm{on } (B^2_\delta \setminus B^2_{\delta - \delta h(\epsilon)}) \times\{z_+\} \\
                      u_{\delta,\epsilon}^0 & \textrm{on } B^2_{\delta - \delta h(\epsilon)} \times\{z_+\}.
                     \end{cases}
\]
Becuase of Lemma~\ref{lemma:3Dextension},~\eqref{hp_cylinder:log},~\eqref{hp_cylinder:small_log},~\eqref{trivial2} and~\eqref{up_nontrivial2},
a minimizer~$u_{\delta,\epsilon}^+$ corresponding to the boundary condition~$u = g_{\delta,\epsilon}^+$ on~$\partial\Lambda_\delta^+$ satisfies
\begin{equation} \label{up_nontrivial3}
 E_\epsilon(u_{\delta,\epsilon}^+, \, \Lambda_\delta^+) \leq C \delta \, E_\epsilon(g_{\delta,\epsilon}^+, \, \partial\Lambda_\delta^+) 
 \leq C\left(\delta M + \delta^2\eta + \delta\eta + \eta\right) \log\frac{\delta}{\epsilon}
\end{equation}
(we have also used that~$h(\epsilon) \leq C$).
In the subdomain~$\Lambda_\delta^-$, we define~$u_{\delta,\epsilon}^-$ in a similar way.
As before, pasting $u_{\delta,\epsilon}^-$, $u_{\delta,\epsilon}^0$, $u_{\delta,\epsilon}^+$ and~$\varphi_{\delta,\epsilon}$ restricted to~$A_{\delta,\epsilon}^\prime$ we obtain an admissible comparison map.
Therefore, combining~\eqref{up_nontrivial1}, \eqref{up_nontrivial3} and~\eqref{phi_energy}, we deduce that the upper bound in~\eqref{cylinder:nontrivial} holds, provided that
\begin{equation} \label{up_nontrivial4}
 \alpha(M, \, \eta, \, \delta) \geq C_2 \left(\delta M + \delta^2\eta + \delta\eta + \eta + \delta^{-1}\eta\right).
\end{equation}

\subsubsection*{The homotopy class of $v_{\delta, \epsilon}$ is non-trivial: proof of the lower bound in~\eqref{cylinder:nontrivial}}

Finally, we need to prove the lower bound in~\eqref{cylinder:nontrivial}, again assuming that the homotopy class of~$v_{\delta, \epsilon}$ is non-trivial.
The essential tool, here, is the Jerrard-Sandier type estimate (Corollary~\ref{cor:lower bound}).
However, in order to be able to apply such an estimate, once again we need to take care of the boundary conditions by means of an interpolation argument.
Using cylindric coordinates~$(\rho, \, \theta, \, z)\in [0, \, \delta]\times[0, \, 2\pi)\times[-1, \, 1]$, we define the map~$\tilde u_{\delta, \epsilon}\colon \Lambda_\delta\to\Sz$ by
\[
 \tilde u_{\delta, \epsilon}(\rho e^{i\theta}, \, z) := \begin{cases}
                 u_{\delta,\epsilon}\left(\dfrac{\rho e^{i\theta}}{1 - h(\epsilon)}, \, \theta, \, z\right) 
                 & \textrm{if } \rho \leq \delta - \delta h(\epsilon) \textrm{ and } |z| \leq 1 \\
                 \varphi_{\delta, \epsilon}\left((2\delta - \delta h(\epsilon) - \rho) e^{i\theta}, \, z\right) 
                 & \textrm{if } \delta - \delta h(\epsilon) \leq \rho \leq \delta \textrm{ and } |z| \leq 1.
                \end{cases}
\]
This map belongs to~$H^1$, satisfies
\begin{equation} \label{nontrivial1}
 \tilde u_{\delta, \epsilon}(\delta e^{i\theta}, \, z) = v_{\delta, \epsilon}(e^{i\theta}, \, z) \qquad \textrm{for } \H^2\textrm{-a.e. } (\theta, \, z)\in [0, \, 2\pi)\times[-1, \, 1]
\end{equation}
and
\begin{equation} \label{nontrivial2}
 E_\epsilon(\tilde u_{\delta, \epsilon}, \, \Lambda_\delta) \leq E_\epsilon(u_{\delta,\epsilon}, \, \Lambda_\delta) + E_\epsilon(\varphi_{\delta, \epsilon}, \, A_{\delta,\epsilon})
 \stackrel{\eqref{phi_energy}}{\leq}  E_\epsilon(u_{\delta,\epsilon}, \, \Lambda_\delta) + C \eta \, h(\epsilon) \log\frac{\delta}{\epsilon} .
\end{equation}
By Fubini theorem and~\eqref{nontrivial1}, for a.e.~$z\in [-1, \, 1]$ the map~$\tilde u_{\delta, \epsilon}$ restricted to~$\partial B^2_\delta\times\{z\}$ 
belongs to~$H^1$, is~$\NN$-valued and has a nontrivial homotopy class.
Moreover, we can always assume WLOG that~$\epsilon \leq \epsilon_0 < \delta/2$. 
Let us apply Corollary~\ref{cor:lower bound} to the function~$\tilde{u}_{\delta, \epsilon}$ restricted to~$B^2_\delta\times\{z\}$.
This yields
\[
 E_\epsilon(\tilde u_{\delta, \epsilon}, \, B_\delta^2\times\{z\}) + C \delta \, E_\epsilon(\tilde u_{\delta, \epsilon}, \, \partial B_\delta^2\times\{z\}) \geq \kappa_* \log\frac{\delta}{\epsilon} - C .
\]
By integrating with respect to~$z\in[-1, \, 1]$, and using~\eqref{nontrivial1} again, we deduce that
\[
 E_\epsilon(\tilde u_{\delta, \epsilon}, \, \Lambda_\delta) + C \delta \, \int_{D_{\delta,\epsilon}}\abs{\nabla v_{\delta, \epsilon}}^2 \,\d\H^2 \geq 2\kappa_* \log\frac{\delta}{\epsilon} - C.
\]
Then, thanks to~\eqref{v_energy}, we obtain
\[
 E_\epsilon(\tilde u_{\delta, \epsilon}, \, \Lambda_\delta) \geq \left(2\kappa_* - C\delta\eta\right) \log\frac{\delta}{\epsilon} - C,
\]
Finally, combining this inequality with~\eqref{nontrivial2}, we conclude that
\[
 E_\epsilon(u_{\delta,\epsilon}, \, \Lambda_\delta) \geq \left(2\kappa_* - C\delta\eta - C\eta \, h(\epsilon) \right) \log\frac{\delta}{\epsilon} - C
\]
and, since~$h(\epsilon) :=\epsilon^{1/2}|\log\epsilon|$ is bounded, the lower bound in~\eqref{cylinder:nontrivial} is satisfied if
\begin{equation} \label{nontrivial3}
 \alpha(M, \, \eta, \, \delta) \geq C_3 \left(\delta \eta + \eta\right).
\end{equation}
Thanks to~\eqref{trivial6}, \eqref{up_nontrivial4} and~\eqref{nontrivial3}, Lemma~\ref{lemma:cylinder} is satisfied if we set
\[
 \alpha(M, \, \eta, \, \delta) :=\max\{C_1, \, C_2, \, C_3\} \left(\delta M + \delta^2\eta + \delta\eta + \eta + \delta^{-1}\eta\right).
\]

\subsection{The singular measure has constant density}
\label{subsect:constant_density}

Aim of this subsequence is to prove the following

\begin{prop}\label{prop:constant_density}
 For $\H^1$-a.e.~$x\in\Sl\cap\Omega$, there holds $\Theta(x) = \kappa_*$.
\end{prop}

This property is of crucial importance, because it allow us to describe the structure of the singular set in the interior of the domain and to prove Proposition~\ref{prop:intro-S}.

\begin{proof}[Proof of Proposition~\ref{prop:constant_density}]
 Because~$\mu_0\mres\Omega$ is rectifiable, $\Theta$ is approximately continuous and $\mu_0$ has an approximate tangent line (i.e., \eqref{tangent_line} holds) at $\H^1$-a.e. point $x_0\in\Sl\cap\Omega$.
 Fix such a point~$x_0$. By~\eqref{tangent_line}, there exists a line~$L$ such that the measures~$(\nu_\lambda)_{\lambda >0}$ defined by
 \[
  \nu_\lambda(A) := \lambda^{-1} \mu_0(\lambda A\cap\Omega) \qquad \textrm{for } A\in\mathscr{B}(\R^3)
 \]
 satisfy
 \begin{equation} \label{tangent_measure}
  \nu_\lambda \rightharpoonup^\star \nu_0 := \Theta(x_0)\H^1\mres L \qquad \textrm{weakly}^\star \textrm{ in }\mathscr{M}_{\mathrm{b}}(\Omega) \textrm{ as } \lambda\to 0.
 \end{equation}
 Up to rotations and translations, we can assume WLOG that~$x_0 = 0$ and~$L = \{x\in\R^3\colon x_1 = x_2 = 0\}$. Let
 \begin{equation} \label{Mprime}
  M^\prime := \frac{2\sqrt{2} M}{\dist(x_0, \, \partial\Omega)} ,
 \end{equation}
 where~$M$ is given by assumption~\eqref{hp:H}, and let~$\eta_0 = \eta_0(M^\prime)$ be the corresponding number given by Lemma~\ref{lemma:cylinder}.
 Let~$0< \eta \leq \eta_0$ and~$0 < \delta < 1$ be two small parameters, to be choosen later.
 We consider again the cylinder~$\Lambda_\delta := B^2_\delta\times[-1, \, 1]$, with lateral surface~$\Gamma_\delta := \partial B^2_\delta\times [-1, \, 1]$.
 Since~$\nu_0(\Gamma_\delta) = 0$, because of~\eqref{tangent_measure} there exists a positive number $\lambda_0 = \lambda_0(\eta, \, \delta) < \dist(x_0, \, \partial\Omega)/(2\sqrt{2})$ such that
 \[
  \mu_0(\lambda\Gamma_\delta) \leq \frac{\lambda\eta}{2} \qquad \textrm{for } 0 < \lambda \leq \lambda_0.
 \]
 Then, for a fixed~$0 < \lambda \leq \lambda_0$, thanks to~\eqref{mueps} we find a positive number~$n_0 = n_0(M^\prime \! , \, \eta, \, \delta, \, \lambda)$ such that
 \begin{equation} \label{small_log_unscaled}
  E_{\varepsilon_n}(Q_{\varepsilon_n}, \, \lambda\Gamma_\delta) \leq \lambda\eta \log\frac{\lambda\delta}{\varepsilon_n} \qquad \textrm{for any integer } n \geq n_0.
 \end{equation}
 Moreover, the cylinder~$\lambda\Lambda_\delta$ is contained in a ball centered at~$x_0$ with radius~$\sqrt{2}\lambda < r_0 := \dist(x_0, \, \partial\Omega)/2$.
 Then, because of the monotonicity formula (Lemma~\ref{lemma:monotonicity}) and~\eqref{hp:H} we have
 \[
  E_{\varepsilon_n}(Q_{\varepsilon_n}, \, \lambda\Lambda_\delta) \leq \frac{\sqrt 2\lambda}{r_0} E_{\varepsilon_n}(Q_{\varepsilon_n}, \, B_{r_0}) 
  \stackrel{\eqref{hp:H}-\eqref{Mprime}}{\leq} \lambda M^\prime \log\frac{\lambda\delta}{\varepsilon_n}
 \]
 for any integer $n\geq n_0$. Thanks to Fatou lemma, we deduce
 \[
  \int_{-\lambda}^\lambda \liminf_{n\to+\infty} E_{\varepsilon_n}(Q_{\varepsilon_n}, \, B^2_{\lambda\delta}\times\{z\}) \,\d z \leq \lambda M^\prime \log\frac{\lambda\delta}{\varepsilon_n}
 \]
 so, by an average argument, we can find two numbers~$z_-\in [-\lambda, \, -3\lambda/4]$, $z_+\in [3\lambda/4, \, \lambda]$ and a subsequence (still denoted~$\varepsilon_n$) such that
 \begin{equation} \label{log_unscaled}
  E_{\varepsilon_n}(Q_{\varepsilon_n}, \, B^2_{\lambda\delta}\times\{z_-, \, z_+\}) \leq M^\prime \log\frac{\lambda\delta}{\varepsilon_n} 
  \qquad \textrm{for any integer } n \geq n_0.
 \end{equation}
 To avoid notation, we will assume that~$z_- = -\lambda$ and~$z_+ = \lambda$.
 Reducing the value of~$n_0$ if necessary, we can assume that~$\varepsilon_n \leq \lambda\epsilon_0$, 
 where $\epsilon_0 = \epsilon_0(M^\prime, \, \eta, \, \delta)$ is given by Lemma~\ref{lemma:cylinder}.
 
 Now, fix an integer~$n\geq n_0$. We set~$\epsilon := \varepsilon_n/\lambda$ (notice that~$\epsilon \leq \epsilon_0$) and
 \[
  u_{\delta,\epsilon}(y) := Q_{\varepsilon_n}\left(\lambda y\right) \qquad \textrm{for } y\in\Lambda_\delta .
 \]
 From~\eqref{small_log_unscaled} and~\eqref{log_unscaled} we deduce that
 \[
  E_\epsilon(u_{\delta,\epsilon}, \, \Gamma_\delta) \stackrel{\lambda < 1}{\leq} \eta \log\frac{\delta}{\epsilon} \qquad \textrm{and} \qquad
  E_\epsilon(u_{\delta,\epsilon}, \, B^2_{\delta}\times\{-1, \, 1\}) \leq M^\prime \log\frac{\delta}{\epsilon}
 \]
 so the conditions~\eqref{hp_cylinder:log} and~\eqref{hp_cylinder:small_log} are satisfied; moreover,~$u_{\delta,\epsilon}$ satisfies~\eqref{hp_cylinder:Linfty} because of~\eqref{hp:H}.
 Therefore, we can apply Lemma~\ref{lemma:cylinder}. Scaling back to~$Q_{\varepsilon_n}$, we conclude that either it holds
 \begin{equation} \label{trivial_case}
  E_{\varepsilon_n}(Q_{\varepsilon_n}, \, \lambda\Lambda_\delta) \leq \lambda\alpha \log\frac{\lambda\delta}{\varepsilon_n}
 \end{equation}
 or it holds
 \begin{equation} \label{nontrivial_case}
  \lambda\left(2 \kappa_* - \alpha\right) \log\frac{\lambda\delta}{\varepsilon_n} - \lambda C \leq E_{\varepsilon_n}(Q_{\varepsilon_n}, \, \lambda\Lambda_\delta) 
  \leq \lambda\left(2 \kappa_* + \alpha\right) \log\frac{\lambda\delta}{\varepsilon_n} + \lambda C ,
 \end{equation}
 where~$\alpha = \alpha(M^\prime \!, \, \eta, \, \delta)$ is a positive number which satisfies
 \[
  \alpha \leq C \left(\delta M^\prime + \delta^2\eta + \delta\eta + \eta + \delta^{-1}\eta\right) .
 \]
 At this point, we choose~$\delta := \eta^{1/2}$, so that~$\alpha\to 0$ when~$\eta\to 0$.
 
 Suppose that the inequality~\eqref{trivial_case} holds. then, passing to the limit as~$n\to +\infty$, thanks to~\eqref{mueps} we find that
 \[
  \mu_0(\textrm{interior of } \lambda\Lambda_\delta) \leq \lambda\alpha .
 \]
 Passing to the limit as~$\lambda\to 0$, with the help of~\eqref{tangent_measure} we obtain
 \[
  2\Theta(x_0) = \nu_0(\Lambda_\delta) \leq \alpha
 \]
 and finally, letting~$\eta\to 0$ (so that~$\alpha\to 0$) we conclude that~$\Theta(x_0) = 0$.
 This is a contradiction, because~$x_0$ is supposed to be an approximate continuity point for~$\Theta$ and~$\Theta$ is bounded away from~$0$ on~$\Sl\cap\Omega$ (Lemma~\ref{lemma:density positive}).
 Therefore, \eqref{trivial_case} does not hold, and so~\eqref{nontrivial_case} must be satisfied instead.
 Passing to the limit as~$n\to+\infty$ and~$\lambda\to 0$, and using~\eqref{mueps} and~\eqref{tangent_measure} again, we deduce that
 \[
  2\kappa_* - \alpha \leq \nu_0(\Lambda_\delta) \leq 2\kappa_* + \alpha
 \]
 or equivalently
 \[
  \kappa_* - \frac{\alpha}{2} \leq \Theta(x_0) \leq \kappa_* + \frac{\alpha}{2}.
 \]
 Letting~$\eta\to 0$, we conclude that~$\Theta(x_0) = \kappa_*$. 
\end{proof}

\begin{remark} \label{remark:nontrivial_class}
 As a byproduct of the previous proof, we obtain a topological information about~$Q_0$.
 Let~$x_0\in\Sl\cap\Omega$ be as in the previous proof (i.e., $\Theta$ is approximately continuous at~$x_0$ and~$\mathrm{Tan}(\mu_0, \, x)$ exists).
 If~$D\csubset\Omega$ is a disk and~$D\cap\Sl = \{x_0\}$, then the homotopy class of~$Q_0$ restricted to~$\partial D$ is non-trivial.
 For we know by the previous proof that~\eqref{nontrivial_case} must be satisfied, 
 so we are in the case~$v_{\delta,\epsilon}$ has a non-trivial class (we are using the notation of Section~\ref{subsect:cylinder}).
 This means that~$\RR\circ u_{\delta,\epsilon}$, where is well-defined, has a non-trivial class too.
 Taking the limit as~$\epsilon\to 0$, we conclude that~$Q_0$ has a non-trivial class 
 because $u_{\delta,\epsilon}$ converge locally uniformly to~$Q_0$, away from the singular set~$\Sl\cup\Spt$ (Theorem~\ref{th:convergence}).
\end{remark}

Proposition~\ref{prop:intro-S} now follows quite easily from a result by Allard and Almgreen, which is a structure theorem for stationary varifolds of dimension~$1$.

\begin{proof}[Proof of Proposition~\ref{prop:intro-S}]
 Let~$\SS^\prime$ be the set of points~$x\in\Sl\cap\Omega$ such that~$\Theta(x) = \kappa_*$.
 Since~$\mu_0\mres\Omega$ is a stationary varifold, \cite[Theorem and Remark p.~89]{AllardAlmgren} imply that
 $\SS^\prime$ is a relatively open subset of~$\Sl\cap\Omega$, such that
 \begin{equation} \label{Sprime}
  \H^1((\Sl\cap\Omega)\setminus\SS^\prime) = 0, 
 \end{equation}
 and $\SS^\prime\cap K$ is a finite union of straight segments, for any open set~$K\csubset\Omega$.
 Moreover, the set~$\SS^\prime$ must be dense in~$\Sl\cap\Omega$.
 Indeed, suppose that there exists a point~$x_0\in\Sl\cap\Omega$ and an open neightborhood~$B\subseteq\Omega$ of~$x_0$, such that~$B$ does not intersect the closure of~$\SS^\prime$.
 Then, we have~$\H^1(\Sl\cap B) = 0$ because of~\eqref{Sprime}, therefore~$\mu_0(B) = 0$ by Proposition~\ref{prop:S} and so~$\Sl\cap B = \emptyset$, which is a contradiction.
 It follows that, for any open set~$K\csubset\Omega$, $\Sl\cap\overline{K}$ is a finite union of closed line segments, $L_1, \, \ldots, \, L_p$.
 By subdividing the segments, if necessary, we can assume WLOG that, for each~$i\neq j$, either~$L_i$ and~$L_j$ are disjoint or their intersection is a common endpoint.
 
 Property~\ref{item:S-nontrivial} now follows directly from Remark~\ref{remark:nontrivial_class}.
 We still have to show Property~\ref{item:S-even}.
 Suppose that~$x_0\in K$ is an endpoint of exactly~$q\leq p$ line segments, say~$L_1, \, \ldots, \, L_q$.
 We claim that~$q$ is even. Let~$V$ be the~$\delta$-neighborhood of~$\Sl$, for a small, positive number~$\delta$.
 Pick a cylinder~$\Lambda\subseteq K$ which contains~$x_0$, such that the lateral surface of~$\Lambda$ does not intersect~$V$ (see Figure~\ref{fig:branching}).
 Since~$(\Spt\setminus V)\cap K$ is finite, modifying $\Lambda$ if necessary we can assume that~$\partial\Lambda\setminus V$ 
 does not contain any singular point~$\Spt$, so~$Q_0$ is well-defined and continuous on~$\partial\Lambda\setminus V$.
 In particular, if we denote by~$U_-$, $U_+$ the two bases of the cylinder, the maps~${Q_0}_{|\partial U_+}$ and~${Q_0}_{|\partial U_-}$ are homotopic to each other.
 Assume now, by contradiction, that~$q$ is odd.
 Then one of the bases --- say~$U_+$ --- must intersect an even number of segments, and the other must intersect an odd number of segments.
 Therefore, due to Remark~\ref{remark:nontrivial_class}, the homotopy class of~${Q_0}_{|\partial U_+}$ must be trivial, and the homotopy class of~${Q_0}_{|\partial U_-}$ must be non-trivial.
 This is a contradiction, hence~$q$ is even.
\end{proof}

\begin{figure}
 \centering
  \includegraphics[height=.35\textheight]{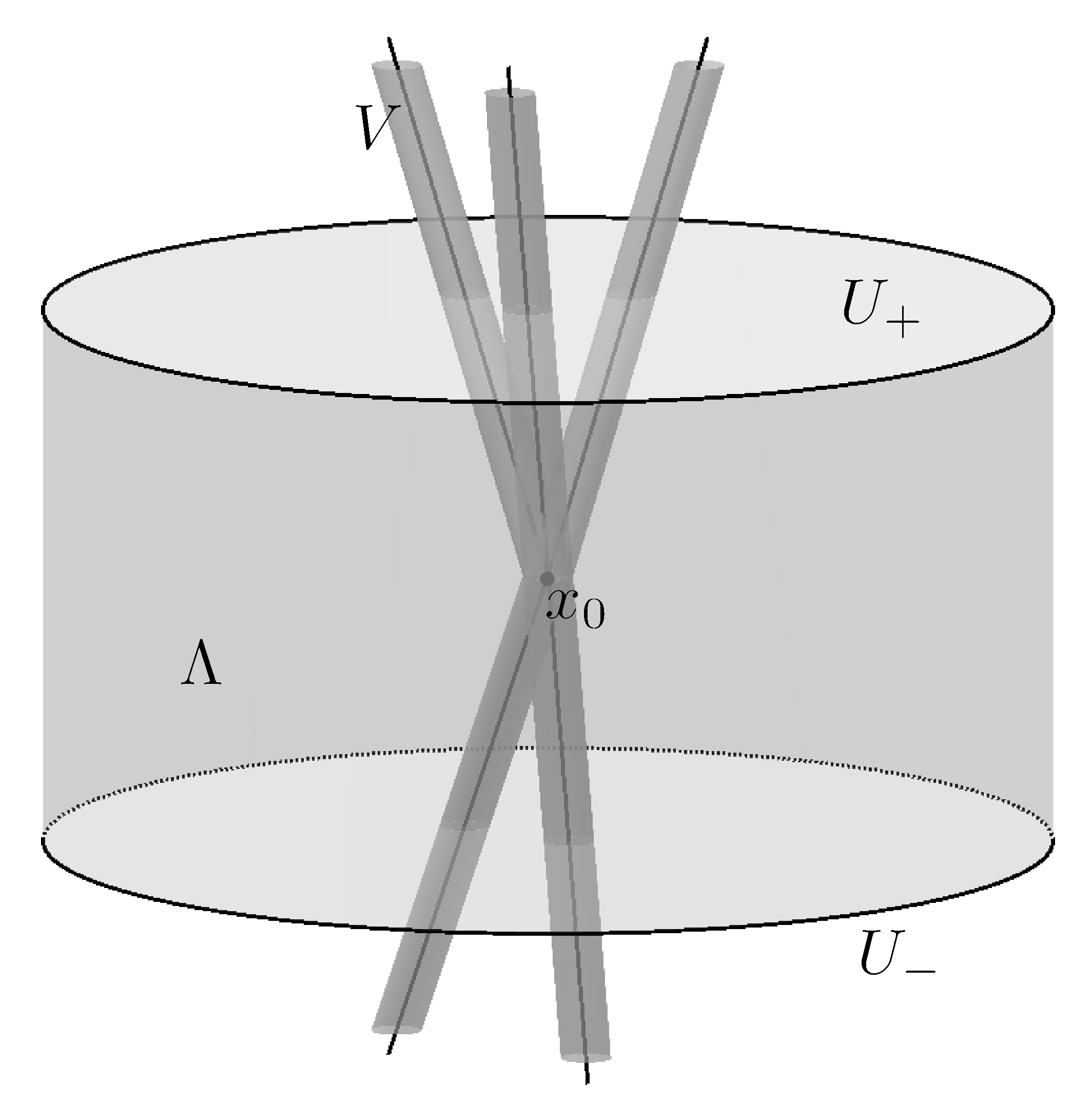}
  \caption{A branching point for the singular set~$\Sl$, surrounded by a cylinder~$\Lambda$.
  The homotopy class of~$Q_0$ restricted to the boundary of a disk which crosses transversely~$\Sl$ is determined by the number of intersections between the disk and~$\Sl$, modulo~$2$.
  A configuration such as the one represented in this figure cannot occur, otherwise $Q_0$ restricted to~$\partial U_+$ and~$\partial U_-$ would be in different homotopy classes.}
 \label{fig:branching}
\end{figure}

\begin{remark} \label{remark:zerosum}
 For~$k\in\{1, \, \ldots, \, q\}$, let~$\nu^{(k)}$ be the direction vector associated with~$L_k$ which has unit norm and points outward from~$x_0$.
 Then, the stationarity of~$\Sl$ implies that
 \[
  \sum_{k = 1}^q \nu^{(k)} = 0.
 \]
 Indeed, taking an arbitrary vector field~$\X$ supported in a small neighborhood of~$x_0$, thanks to~\eqref{Sstationary} we have
 \[
  0 
    = \kappa_* \sum_{k = 1}^q \int_0^{\mathrm{length}(L_k)} \nu^{(k)}_i \nu^{(k)}_j \frac{\partial \X_i}{\partial x_j} (x_0 + t\nu^{(k)}) \, \d t
    = - \kappa_* \sum_{k = 1}^q \nu^{(k)}\cdot \X (x_0).
 \]
\end{remark}

\subsection{Concentration of the energy at the boundary: an example}
\label{subsect:torus}

The arguments we presented in this section may not be extended to the analysis of the singular set near the boundary of the domain.
In particular, the stationarity of~$\mu_0$ may fail. 
We discuss now an example where the boundary datum is independent of~$\varepsilon$ and smooth, yet the geometry of the domain forces the energy of the minimizers to concentrate at the boundary.
As a result, $\Sl$ is non-empty but~$\Sl\subseteq\partial\Omega$, and $\Sl$ is \emph{not} a locally finite union of straight line segments.

Let~$\r\colon(0, \, +\infty)\times\R\times [0, \, 2\pi]\to\R^3$ be the cylindrical change of coordinates, given by 
\[
  \r(\rho, \, z, \, \theta) := (\rho\cos\theta, \, \rho\sin\theta, \, z)^{\mathsf{T}}.
\]
Let~$D$ be the disk~$B_1(2, 0)$ in the~$(\rho, \, z)$-plane, and let~$\Omega$ be the solid torus generated by the revolution of~$D$, that is
\[
 \Omega := \r\left(D\times [0, \, 2\pi)\right) = \left\{(x_1, \, x_2, \, x_3)\in\R^3 \colon  \left(\sqrt{x_1^2 + x_2^2} - 2\right)^2 + x_3^2 < 1\right\}.
\]
We consider the boundary datum~$g\in C^1(\partial\Omega, \, \NN)$ given by~$g = h\circ\r^{-1}$, where
 \[
  h(\rho, \, z) := s_* \left\{ \left(\e_1 \cos\frac{\varphi(\rho, \, z)}{2} + \mathbf{e}_3 \sin\frac{\varphi(\rho, \, z)}{2} \right)^{\otimes 2} - \frac13 \Id \right\} ,
 \]
and~$\varphi(\rho, \, z)$ is the oriented angle between 
the ray starting at~$(2, \, 0)$ and passing through~$(\rho, \, z)$ and the positive $\rho$-axis.
Notice that, for each~$\theta\in[0, \, 2\pi)$, the restriction of~$g$ to the slice~$\r(\partial D\times\{\theta\})$ has a non-trivial homotopy class.
We will prove the following

\begin{prop} \label{prop:torus}
 For this choice of the domain and the boundary datum, we have
 \[
  \Sl = \r\left(\{(1, \, 0)\}\times [0, \, 2\pi)\right) = \left\{ (x_1, \, x_2, \, x_3)\in\R^3 \colon  x_1^2 + x_2^2 = 1, \ x_3 = 0\right\}.
 \]
 In particular, $\Sl\subseteq\partial\Omega$.
\end{prop}

Because of the cylindrical symmetry, the problem is essentially bidimensional.
Indeed, given any map~$Q\in H^1_g(\Omega, \, \Sz)$, by a change of variable we obtain
\begin{equation} \label{cylindrical}
\begin{split}
 E_\varepsilon(Q, \, \Omega) &= \int_0^{2\pi} \int_D\left\{\frac12 \abs{\partial_\rho(Q\circ\r)}^2 + \frac12 \abs{\partial_z(Q\circ\r)}^2 
 + \frac{1}{\varepsilon^2} f(Q\circ\r)\right\} \rho\,\d\rho\,\d z\,\d\theta \\
 &\qquad + \int_0^{2\pi} \int_D \frac{1}{2\rho}\abs{\partial_\theta (Q\circ\r)}^2 \,\d\rho\,\d z\,\d\theta .
\end{split}
\end{equation}
Therefore, a map~$Q_\varepsilon$ is a minimizer for~\eqref{energy} in the class~$H^1_g(\Omega, \, \Sz)$ if and only if~$Q_\varepsilon = P_\varepsilon\circ\r^{-1}$,
where~$P_\varepsilon$ only depend on~$(\rho, \, z)$ and is a minimizer for the weighted functional
\begin{equation} \label{energy2D}
 F_\varepsilon (P, \, D) := 2\pi\int_D \left\{ \frac12 \abs{\nabla P}^2 + \frac{1}{\varepsilon^2} f(P)\right\} \rho \, \d\rho\d z
\end{equation}
in the class
\[
 H^1_{h}(D, \, \Sz) := \left\{P\in H^1(D, \, \Sz) \colon P = h \textrm{ on } \partial D\right\}.
\]

\begin{lemma} \label{lemma:upperF}
For any~$\delta > 0$, there exists a constant~$C_\delta$ such that, for any~$0 < \varepsilon < \delta/4$, there holds
\[
 F_\varepsilon(P_\varepsilon, \, D)\leq 2\pi\kappa_* (1 + \delta)|\log\varepsilon| + C_\delta.
\]
\end{lemma}
\proof
Fix~$\delta>0$,~$0 < \varepsilon < \delta/4$ and let~$\rho_\delta := 1+ 3\delta/4$, $D_\delta := B^2_{\delta/4}(\rho_\delta, \, 0)$.
We define a map~$\tilde P_{\delta, \varepsilon}\colon D_\delta\to\Sz$ by setting
\[
 \tilde P_{\delta, \varepsilon}(\rho, \, z) := \eta_\varepsilon\left(\sqrt{(\rho - \rho_\delta)^2 + z^2}\right) h\left(2 + \frac{4}{\delta}(\rho - \rho_\delta), \, \frac{4z}{\delta}\right) 
 \qquad \textrm{if } (\rho - \rho_\delta)^2 + z^2 \leq \frac{\delta^2}{16},
\]
where~$\eta_\varepsilon$ is defined by~\eqref{eta-epsilon}.
The restriction of~$\tilde P_{\delta,\varepsilon}$ to~$\partial D_\delta$ is a $\NN$-valued loop which depends on~$\delta$ but not on~$\varepsilon$, and has a non-trivial homotopy class.
Therefore, there exists a map~$\tilde P_\delta\colon D\setminus D_\delta\to \NN$ such that
\[
 \tilde P_\delta = h \quad \textrm{on } \partial D, \qquad  \tilde P_\delta = \tilde P_{\delta,\varepsilon} \quad \textrm{on } \partial D_\delta.
\]
We extend~$\tilde P_{\delta, \, \varepsilon}$ to a new map, still denoted~$\tilde P_{\delta,\varepsilon}$, by setting~$\tilde P_{\delta,\varepsilon} := \tilde P_\delta$ on~$D\setminus D_\delta$.
Then~$\tilde P_{\delta,\varepsilon}\in H^1_h(D, \, \Sz)$ and, through a straightforward computation, we obtain
\[
\begin{split}
 F_\varepsilon(P_\varepsilon, \, D) &\leq F_\varepsilon(\tilde P_{\delta,\varepsilon}, \, D) 
 = F_\varepsilon(\tilde P_{\delta,\varepsilon}, D_\delta) + \pi\int_{D\setminus D_\delta} |\nabla\tilde P_{\delta}|^2 \,\d\H^2 \\
 &\leq 2\pi(1 + \delta) E_\varepsilon(\tilde P_{\delta,\varepsilon}, D_\delta) + C_\delta \\
 &\leq 2\pi(1 + \delta) \log\frac{\delta/4}{\varepsilon} + C_\delta = 2\pi(1 + \delta) |\log\varepsilon| + C_\delta,
 \end{split}
\]
where the symbol~$C_\delta$ denotes several constants which only depend on~$\delta$, $D$ and~$h$.
\endproof

This lemma has an immediate consequence on the limit measure of the three-dimensional minimization problem.

\begin{cor} \label{cor:upperF}
 There holds
 \[
  \mu_0(\overline\Omega) \leq 2\pi\kappa_* .
 \]
\end{cor}
\proof Since we have~$E_\varepsilon(Q_\varepsilon, \, \Omega) = F_\varepsilon(P_\varepsilon, \, D)$ because of~\eqref{cylindrical}, 
the corollary follows by taking the limit in Lemma~\ref{lemma:upperF} first as~$\varepsilon\to 0$, then as~$\delta\to 0$. \endproof

Now, since the $3$D-minimizers $Q_\varepsilon$ satisfy $Q_\varepsilon = P_\varepsilon\circ\r^{-1}$ and~$P_\varepsilon$ is independent of the~$\theta$-variable,
the singular set~$\Sl$ must be of the form~$\Sl = \r(\Sigma\times [0, \, 2\pi))$ for some~$\Sigma\subset\overline D$.
Notice that~$\Sigma$ is non-empty, because the homotopy class of the boundary datum~$h$ is non-trivial.
Moreover, $\Sigma$ is finite because~$\H^1(\Sl) < \infty$.

\begin{lemma} \label{lemma:lowerF}
 Suppose that~$(\rho_0, \, z_0)\in\Sigma$. Then
 \[
  \mu_0(\overline\Omega) \geq 2\pi \rho_0 \kappa_*.
 \]
\end{lemma}

Combining Corollary~\ref{cor:upperF} and Lemma~\ref{lemma:lowerF}, we immediately deduce that~$\rho_0 = 1$, that is, $\Sigma = \{(1, \, 0)\}$, whence Proposition~\ref{prop:torus} follows. 

\begin{proof}[Proof of Lemma~\ref{lemma:lowerF}]
 Fix a positive number~$\delta < 1/2$. It suffices to prove the inequality
 \begin{equation} \label{lowerF,1}
  F_\varepsilon(P_\varepsilon, \, D) \geq 2\pi \kappa_* (\rho_0 - \delta)  |\log\varepsilon| - C_\delta
 \end{equation}
 for any small $\varepsilon$ and for some positive constant~$C_\delta$ which depends only on~$\delta$, $D$, $h$.
 The lemma will follow by using~\eqref{cylindrical} and passing to the limit as~$\varepsilon\to 0$, $\delta\to 0$.
 
 Consider the disk~$D^\prime := B^2_{3/2}(2, \, 0)\supset\!\supset D$, as well as the associated torus of revolution~$\Omega^\prime := \r(D^\prime\times [0, \, 2\pi))$.
 We extend each minimizer~$P_\varepsilon$ to a new map defined on~$D^\prime$, still denoted by~$P_\varepsilon$, by setting~$P_\varepsilon := h$ on~$D^\prime\setminus D$.
 In the same way, we extend~$Q_\varepsilon$ by setting~$Q_\varepsilon := g$ on~$\Omega^\prime\setminus\Omega$.
 Moreover, we set
 \[
  D^\prime_\delta := B^2_{\delta/2}(\rho_0, \, z_0) .
 \]
 Since~$\Sigma$ is finite, reducing the value of $\delta$ if necessary we can assume that $D^\prime_\delta\cap\Sigma = \{(\rho_0, \, z_0)\}$.
 Therefore, Remark~\ref{remark:nontrivial_class} implies that $Q_\varepsilon$ restricted to~$\partial D^\prime_\delta$ has a non-trivial homotopy class, provided that~$\varepsilon$ is small enough. 
 Thus, we can apply Corollary~\ref{cor:lower bound} to bound from below the energy on~$D^\prime_\delta$. We have
 \begin{equation} \label{lowerF,2}
 \begin{split}
  F_\varepsilon(P_\varepsilon, \, D^\prime_\delta) &\geq 2\pi \left(\rho_0 - \frac{\delta}{2}\right) E_\varepsilon(P_\varepsilon, \, D^\prime_\delta) \\
  &\geq 2\pi\kappa_* \left(\rho_0 - \frac{\delta}{2}\right) \, \phi_0(P_\varepsilon, \, \partial D^\prime_\delta) \, \log\frac{\delta}{\varepsilon} 
  - C\bigg(E_\varepsilon(P_\varepsilon, \, \partial D^\prime_\delta) + 1\bigg).
 \end{split}
 \end{equation}
 Thanks to Theorem~\ref{th:convergence} and Corollary~\ref{cor:convergence_bd}, we have $Q_\varepsilon\to Q_0$ weakly in~$H^1$, 
 in a small neighborhood of~$\r(\partial D^\prime_\delta\times [0, \, 2\pi))\subset\Omega^\prime$.
 Therefore, modifying again the value of~$\delta$ if necessary, we can assume that
 \[
  Q_\varepsilon \rightharpoonup Q_0 \qquad \textrm{weakly in } H^1(\r(\partial D^\prime_\delta\times [0, \, 2\pi)), \, \Sz)
 \]
 (this property may not be satisfied for any value~$\delta$, but it is satisfied for almost every value thanks to Fubini theorem).
 By Fubini theorem again and compact Sobolev embedding, this yields
 \[
  P_\varepsilon \rightharpoonup P_0 \qquad \textrm{weakly in } H^1(\partial D^\prime_\delta, \, \Sz) \textrm{ and uniformly,}
 \]
 where~$P_0 := Q_0\circ\r$ is an~$\NN$-valued map. In particular, there holds
 \begin{equation} \label{lowerF,3}
  \phi_0(P_\varepsilon, \, \partial D^\prime_\delta) \to \phi_0(P_0, \, \partial D^\prime_\delta) = 1 \qquad \textrm{and} \qquad
  E_\varepsilon(P_\varepsilon, \, \partial D^\prime_\delta) \leq C_\delta. 
 \end{equation}
 Combining~\eqref{lowerF,2} and~\eqref{lowerF,3}, we deduce that for (almost every) small enough~$\delta$, there exist positive constants~$C_\delta$ and ~$\varepsilon_\delta$
 such that~\eqref{lowerF,1} is satisfied for any~$0 < \varepsilon \leq \varepsilon_\delta$.
 This is enough to conclude the proof of the lemma.
\end{proof}

\section{Sufficient conditions for~\eqref{hp:H}}
\label{sect:prop_H}

\subsection{Proof of Proposition~\ref{prop:intro-H1/2}}
\label{subsect:datoH1/2}

In this section, we analyze the role of the domain and the boundary data in connection with~\eqref{hp:H}, and prove sufficient conditions for~\eqref{hp:H} to hold true.
We prove first Proposition~\ref{prop:intro-H1/2}, namely, we assume that~$\Omega$ is a bounded, Lipschitz domain and the boundary datum is bounded in~$H^{1/2}(\partial\Omega, \, \NN)$,
and we show that the inequalities
\begin{equation} \label{Linfty}
  \norm{\Qe}_{L^\infty(\Omega)} \leq M
\end{equation}
 and
\begin{equation} \label{log energy}
  E_\varepsilon(\Qe) \leq M \left( \abs{\log\varepsilon} + 1 \right)
\end{equation}
are satisfied for some positive~$M$.
The~$L^\infty$-bound~\eqref{Linfty} is easily obtained by a comparison argument.

\begin{lemma} \label{lemma:Linfty g}
 Minimizers~$\Qe$ of~\eqref{energy} satisfy
 \[
  \norm{\Qe}_{L^\infty(\Omega)} \leq \max\left\{ \sqrt{\frac 23}s_*, \, \norm{g_\varepsilon}_{L^\infty(\partial\Omega)}\right\} .
 \]
\end{lemma}
\begin{proof}
Set
\[
 M := \max\left\{ \sqrt{\frac 23}s_*, \, \norm{\Qe}_{L^\infty(\partial\Omega)}\right\} ,
\]
and define~$\pi\colon\Sz\to\Sz$ by~$\pi(Q) := M|Q|^{-1} Q$ if~$|Q| \geq M$, $\pi(Q) := Q$ otherwise.
We have
\[
 \D f(Q) \cdot Q =  -a \abs{Q}^2 - b\tr Q^3 + c \abs{Q}^4 > 0 \qquad \textrm{when } \abs{Q} > \sqrt{\frac23} s_* 
\]
(this follows from the inequality $\sqrt 6|\tr Q^3| \leq |Q|^3$; see~\cite[Lemma~1]{Majumdar2010}).
We deduce that~$f(\pi(Q)) \geq f(Q)$ for any~$Q$.
Moreover,~$\pi$ is the projection on a convex set, so it is $1$-Lipschitz continuous.
Thus, the map~$\pi\circ Q_\varepsilon$ belongs to~$H^1_{g_\varepsilon}(\Omega, \, \Sz)$, satisfies~$|\nabla (\pi\circ Q_\varepsilon)| \leq |\nabla \Qe|$ a.e. and~$E_\varepsilon(\pi\circ Q_\varepsilon) \leq E_\varepsilon(\Qe)$, with strict inequality if~$|\Qe| > M$ on a set of positive measure. 
By minimality of~$\Qe$, we conclude that $|\Qe| \leq M$ a.e.
\end{proof}

The logarithmic energy bound~\eqref{log energy} is more delicate.
The proof is adapted from an argument by Rivi\`ere~\cite[Proposition~2.1]{Riviere-DenseSubsets}, which involves 
Hardt, Kinderlehrer and Lin's re-projection trick (see~\cite[Lemma~2.3]{HKL}) in an essential way.

\begin{proof}[Proof of Proposition~\ref{prop:intro-H1/2}]
Since the boundary data are supposed to be~$\NN$-valued, in particular they are bounded in~$L^\infty$, therefore~\eqref{Linfty} follows by Lemma~\ref{lemma:Linfty g}.
We prove~\eqref{log energy} by constructing a suitable comparison function, whose energy is bounded by the right-hand side of~\eqref{log energy}.
For any~$0 < \varepsilon < 1$, let $u_\varepsilon\in H^1(\Omega, \, \Sz)$ be the harmonic extension of~$g_\varepsilon$, i.e. the unique solution of
\[
 \begin{cases}
  -\Delta u_\varepsilon  = 0    & \textrm{in } \Omega \\
  u_\varepsilon = g_\varepsilon & \textrm{on } \partial\Omega .
 \end{cases}
\]
Since~$\{g_\varepsilon\}_\varepsilon$ is bounded in~$H^{1/2}\cap L^\infty$, the sequence $\{u_\varepsilon\}_\varepsilon$ is bounded in~$H^1\cap L^\infty$. 
Let~$\delta > 0$ be a small parameter to be chosen later.
For any $A\in \Sz$ with~$|A|\leq \delta$ and any~$\varepsilon$, we define
\[
 u^A_\varepsilon := \left(\eta_\varepsilon\circ\phi\right) \left( u_\varepsilon - A \right) \RR\left( u_\varepsilon - A\right)
\]
where $\phi\colon\Sz\to\R$ and~$\RR\colon\Sz\setminus\Cs\to\NN$ are defined respectively in Lemmas~\ref{lemma:phi}, \ref{lemma:R C1}, and~$\eta_\varepsilon\in C(\R^+, \, \R)$ is given by
\[
  \eta_\varepsilon(r) := \varepsilon^{-1} r \quad \textrm{if } 0 \leq r < \varepsilon, \qquad \eta_\varepsilon(r) = 1 \quad \textrm{if } r \geq \varepsilon. 
\]
By Lemma~\ref{lemma:R C1} and Corollary~\ref{cor:energy retraction}, we have~$u^A_\varepsilon\in (H^1\cap L^\infty)(\Omega, \, \Sz)$. 
We differentiate $u^A_\varepsilon$ and, taking advantage of the Lipschitz continuity of~$\phi$ (Lemma~\ref{lemma:phi}), we deduce
\[
 \abs{\nabla u^A_\varepsilon}^2 \leq C \left\{ \left(\eta_\varepsilon^\prime\circ\phi\right)^2(u_\varepsilon - A) \abs{\nabla u_\varepsilon}^2
   + \left(\eta_\varepsilon\circ\phi\right)^2(u_\varepsilon - A) \abs{\nabla \left(\RR(u_\varepsilon - A)\right)}^2 \right\} .
\]
We apply Corollary~\ref{cor:energy retraction} to bound the derivative of~$\RR(u_\varepsilon - A)$:
\[
 \abs{\nabla u^A_\varepsilon}^2 \leq C \left\{ \left(\eta_\varepsilon^\prime\circ\phi\right)^2(u_\varepsilon - A)
   + \frac{\left(\eta_\varepsilon\circ\phi\right)^2(u_\varepsilon - A)}{\phi^2(u_\varepsilon - A)} \right\} \abs{\nabla u_\varepsilon}^2 .
\]
On the other hand, there holds
\[
 f\left(u^A_\varepsilon\right) \leq C \one_{\{\phi(u_\varepsilon - A) \leq \varepsilon\}} ,
\]
so
\begin{equation} \label{H1/2 1}
 E_\varepsilon\left( u^A_\varepsilon\right) \leq C \int_\Omega \left\{ \left(\frac{\one_{\{\phi(u_\varepsilon - A)\geq \varepsilon\}}}{\phi^2(u_\varepsilon - A)}
 + \varepsilon^{-2}\one_{\{\phi(u_\varepsilon - A) \leq \varepsilon\}}\right)\abs{\nabla u_\varepsilon}^2 + \varepsilon^{-2}\one_{\{\phi(u_\varepsilon - A) \leq \varepsilon\}} \right\} .
\end{equation}
Now, fix a bounded subset~$K\subseteq\Sz$, so large that $u_\varepsilon(x) + B_\delta^{\Sz} \subseteq K$ for a.e.~$x\in\Omega$ and any~$\varepsilon$ 
(we denote by~$B_\delta^{\Sz}$ the set of~$Q\in\Sz$ with~$|Q|\leq\delta$).
We set~$K_\varepsilon := K \cap \{\phi\leq\varepsilon\}$.
We integrate~\eqref{H1/2 1} with respect to~$A\in B^{\Sz}_\delta$.
We apply Fubini-Tonelli theorem and introduce the new variable~$B := u_\varepsilon(x) - A$. We obtain
\begin{equation*} 
 \begin{split}
  \int_{B_\delta^{\Sz}} E_\varepsilon\left( u^A_\varepsilon\right) \, \d \H^5(A) 
  \leq C \int_{\Omega} \left\{ \left(\int_{K \setminus K_\varepsilon}\frac{\d \H^5(B)}{\phi^2(B)} + \varepsilon^{-2}\H^5(K_\varepsilon) \right) \abs{\nabla u_\varepsilon}^2 
  + \varepsilon^{-2}\H^5(K_\varepsilon) \right\} \, \d x .
 \end{split}
\end{equation*}
We claim that
\begin{equation} \label{H1/2 2}
  \H^5(K_\varepsilon) \leq C \varepsilon^2 \qquad \textrm{and} \qquad
  \int_{K \setminus K_\varepsilon}\frac{\d \H^5(B)}{\phi^2(B)} \leq C \left(|\log\varepsilon| + 1\right) .
\end{equation}
To simplify the presentation, we postpone the proof of this claim. With the help of~\eqref{H1/2 2}, we obtain
\begin{equation*} 
 \int_{B^5_\delta} E_\varepsilon\left( u^A_\varepsilon\right) \, \d \H^5(A) \leq C\left\{ \left(|\log\varepsilon| + 1\right) \norm{\nabla u_\varepsilon}^2_{L^2(\Omega)} + 1\right\} 
 \leq C \left(|\log\varepsilon| + 1\right) .
\end{equation*}
Therefore, we can choose~$A_0\in\Sz$ such that~$|A_0|\leq\delta$ and
\begin{equation} \label{H1/2 3}
 E_\varepsilon\left( u^{A_0}_\varepsilon \right) \leq C\left(|\log\varepsilon| + 1\right) .
\end{equation}

The map $u^{A_0}_\varepsilon$ satisfies the desired energy estimate, but it does not satisfy the boundary condition, since
\begin{equation} \label{H1/2 4}
 u^{A_0}_\varepsilon = \RR\left(g_\varepsilon - A_0 \right) \qquad \textrm{on } \partial\Omega .
\end{equation}
To correct this, we consider the maps~$\{\RR_A\}_{A\in B_\delta^{\Sz}}$ defined by
\[
 \RR_A \colon Q\in \NN \mapsto \RR(Q - A) .
\]
This is a continuous family of mappings in~$C^1(\NN, \, \NN)$ and~$\RR_0 = \Id_{\NN}$. 
Therefore, we can choose~$\delta$ so small that the map~$\RR_A\colon\NN \to \NN$ is a diffeomorphism for any~$A\in B_\delta^{\Sz}$ (in particular for~$A = A_0$).
Let~$\mathscr{U}$ be the set defined by
\[
 \mathscr{U} := \left\{\lambda Q\colon \lambda\in\R^+, \ Q\in\NN \right\} .
\]
We extend $\RR_{A_0}^{-1}$ to a Lipschitz function~$\mathscr F\colon\mathscr{U} \to \mathscr{U}$ by setting
\[
 \mathscr F(\lambda Q) := \lambda \RR_{A_0}^{-1}(Q) \qquad \textrm{for any } \lambda\in\R^+, \ Q\in\NN .
\]
Remark that any $P \in\mathscr{U}\setminus\{0\}$ can be uniquely written in the form $P = \lambda Q$ for~$\lambda\in\R^+$ and~$Q\in\NN$, so~$\mathscr F$ is well-defined.
Also, $f\circ\mathscr F (P)= f(P)$ because~$\mathscr F(P)$ and~$P$ have the same scalar invariants.
The map $P_\varepsilon := \mathscr F \circ u^{A_0}_\varepsilon$ is well-defined, because $u^{A_0}_\varepsilon\in\mathscr{U}$.
Moreover, $P_\varepsilon$ belongs to~$H^1_{g_\varepsilon}(\Omega, \, \Sz)$ thanks to~\eqref{H1/2 4}, and satisfies
\begin{equation*} 
 E_\varepsilon(P_\varepsilon) \leq C\left(|\log\varepsilon| + 1\right)
\end{equation*}
due to~\eqref{H1/2 3}. By comparison, the minimizers satisfy~\eqref{log energy}.
\end{proof}

The claim~\eqref{H1/2 2} follows by this
\begin{lemma} \label{lemma:claim H1/2 A}
 For any~$R > 0$, there exist positive constants~$C_R, \, M_R$ such that, for any non increasing, non negative function $g\colon\R^+\to\R^+$,
 there holds
 \[
  \int_{B_R^{\Sz}} \left(g\circ\phi\right)(Q) \, \d\H^5(Q) \leq C_R \int_0^{M_R} \left( s + s^4 \right) g(s) \, \d s.
 \]
\end{lemma}

Assuming that the lemma holds true, choose~$R$ so large that~$K\subseteq B_R^{\Sz}$.
Then, the two assertions of Claim~\eqref{H1/2 2} follow by taking~$g = \one_{(0, \, \varepsilon)}$ and~$g(s) = \varepsilon^{-2}\one_{(0, \, \varepsilon)}(s) + s^{-2}\one_{[\varepsilon, \, +\infty)}(s)$, respectively.
For the sake of clarity, we split the proof of Lemma~\ref{lemma:claim H1/2 A} into a few steps.
For~$r > 0$, we let~$\dist_r$ denote the geodesic distance in~$\partial B_r^{\Sz}$, that is
\begin{equation} \label{dist_geod}
 \dist_r(Q, \, A) := \inf \left\{\int_0^1 \abs{\gamma^\prime(t)} \, \d t\colon \gamma\in C^1([0, \, 1], \, \partial B_1^{\Sz}), \ \gamma(0) = Q, \
 \gamma(1) \in A \right\}
\end{equation}
for any~$Q\in \partial B_r^{\Sz}$ and~$A\subseteq \partial B_r^{\Sz}$, and set~$\NN^\prime_r := \Cs \cap \partial B_r^{\Sz}$.
Notice that there exists a positive constant~$C$ such that
\begin{equation} \label{geod-eucl}
 \dist_{|Q|}(Q, \, P) \leq C \abs{Q - P}
\end{equation}
for any~$Q$, $P\in\Sz$ with~$|Q| = |P|$.

\begin{lemma} \label{lemma:claim H1/2 B}
 There exists a positive constant~$\alpha$ such that
 \[
  \phi(Q) \geq \alpha \dist_{|Q|}\left(Q, \, \NN^\prime_{|Q|}\right) \qquad \textrm{for any } Q \in\Sz .
 \]
\end{lemma}
\begin{proof}
Fix~$Q\in\Sz$ and~$P\in\NN^\prime_{|Q|}$. By Lemmas~\ref{lemma:representation} and~\ref{lemma:R C1}, we can write
\[
 Q = s \left(\n^{\otimes 2} - \frac13 \Id\right)  + sr \left(\m^{\otimes 2} - \frac13 \Id\right) \qquad \textrm{and} \qquad
 P = - s^\prime \left(\p^{\otimes 2} - \frac 13 \Id \right) ,
\]
for some orthonormal pair~$(\n, \, \m)$, some unit vector~$\p$,~$s$, some positive numbers $s$, $s^\prime$ and~$0 \leq r \leq 1$.
Through simple algebra, we obtain
\begin{equation} \label{claimH1/2B1}
 \abs{Q - P}^2 = \frac23 s^2 \left(r^2 - r + 1\right) - \frac23 s s^\prime (1 - r) + \frac23 {s^\prime}^2 + 2ss^\prime\left\{(\n\cdot \p)^2 + r (\m\cdot \p)^2\right\} .
\end{equation}
By setting~$s^\prime = 0$ and~$s = 0$ in this identity, we see that~$|P| = |Q|$ if and only if
\begin{equation} \label{claimH1/2B2}
 {s^\prime}^2 = s^2 \left(r^2 - r + 1\right).
\end{equation}
By minimizing (with respect to~$s^\prime$,~$\p$) the right-hand side in~\eqref{claimH1/2B1}, subject to the constraint~\eqref{claimH1/2B2}, we find
\[
\begin{split}
 \dist^2(Q, \, \NN^\prime_1) = \frac23 s^2 \sqrt{r^2 - r + 1} \left\{ (1 - r)^2 - \left(\sqrt{r^2 - r + 1} - 1\right)^2 \right\}
 \leq \frac23 s^2 (1 - r)^2 = \frac23 s_*^2 \phi^2(Q) .
 \end{split}
\]
Combining this inequality with~\eqref{geod-eucl}, the lemma follows.
\end{proof}

\begin{lemma} \label{lemma:claim H1/2 D}
 Let~${\NN^\prime}$ be a compact $n$-submanifold of a smooth Riemann $m$-manifold~$\mathscr M$, and let 
 \[
  U_\delta := \left\{ x\in \mathscr M \colon \dist_{\mathscr M}(x, \, \NN^\prime) \leq \delta \right\}
 \]
 be the~$\delta$-neighborhood of~$\NN^\prime$ in~$\mathscr M$, for~$\delta > 0$ (here~$\dist_{\mathscr M}$ stands for the geodesic distance in~$\mathscr M$).
 There exist~$\delta_* > 0$ and, for any~$\delta\in (0, \, \delta_*)$, a constant $C = C(\mathscr M, \, \NN^\prime, \delta)> 0$ such that
 for any decreasing function~$h\colon\R^+\to\R^+$ there holds
 \[
  \int_{U_\delta} h\left( \dist_{\mathscr M}(x, \, \NN^\prime)\right) \, \d\H^m(x) \leq C \int_0^{C\delta} s^{m - n - 1}h(s) \, \d s .
 \]
\end{lemma}
\begin{proof}
We identify~$\R^m = \R^n \times \R^{m - n}$, and call the variable~$y = (y^\prime, \, z)\in\R^n \times\R^{m - n}$.
For a small~$\delta_*>0$, the~$\delta_*$-neighborhood~$U_{\delta_*}$ can be covered with finitely many open sets~$\{V_j\}_{1\leq j \leq K}$ and,
for each~$j$, there exists a bilipschitz homeomorphism~$\varphi_j\colon V_j\to W_j\subseteq\R^m$ which maps~$\NN^\prime\cap V_j$ onto~$\R^n\cap W_j$.
Due to the bilipschitz continuity of the~$\varphi_j$'s, there exist two constants~$\gamma_1$, $\gamma_2$ such that, for any~$j$ and any~$y = (y^\prime, \, z)\in W_j$, there holds
\[
\gamma_1 \abs{z} \leq \dist_{\MM}(\varphi_j^{-1}(y), \, \NN) \leq \gamma_2 \abs{z} .
\]
Therefore, if~$0 < \delta < \delta_*$ the change of variable~$x = \varphi_j^{-1}(y)$ implies
\[
\begin{split}
 \int_{U_\delta} h\left( \dist_{\mathscr M}(x, \, \NN^\prime)\right) \, \d\H^m(x) &\leq 
 \sum_{j = 1}^K \int_{\varphi^{-1}_j(V_j)} h\left(\gamma_1 |z|\right) \, \abs{\mathrm{J}\varphi_j^{-1}(y)} \, \d\H^m(y) \\
 &\leq M \int_{B^{m - n}(0, \, \gamma_2\delta)} h(\gamma_1|z|) \, \d\H^{m - n}(z) 
 \end{split}
\]
where~$M$ is an upper bound for the norm of the Jacobians~$\mathrm{J}\varphi_j^{-1}$.
Then, passing to polar coordinates,
\[
 \begin{split}
  \int_{U_\delta} h\left( \dist_{\mathscr M}(x, \, \NN^\prime)\right) \, \d\H^m(x) &\leq M \int_0^{\gamma_2\delta} \rho^{m - n - 1} h(\gamma_1\rho) \, \d\rho  \\
  &\leq M \gamma_1^{1 + n - m} \int_0^{\gamma_1\gamma_2\rho} s^{m - n - 1} h(s) \, \d s. \qedhere
 \end{split}
\]
\end{proof}

\begin{proof}[Proof of Lemma~\ref{lemma:claim H1/2 A}]
 By Lemma~\ref{lemma:phi}, the function~$\phi$ is positively homogeneous of degree~$1$.
 Then,
 \[
  \int_{B_R^{\Sz}} (g\circ\phi) (Q) \, \d\H^5(Q) = \int_0^R \rho^4 \int_{\partial B_1^{\Sz}} g\left(\rho\phi(Q) \right) \, \d\H^4(Q) \, \d\rho.
 \]
 By applying Lemma~\ref{lemma:claim H1/2 B}, and since~$g$ is a decreasing function,
 \[
  \int_{B_R^{\Sz}} (g\circ\phi) (Q) \, \d\H^5(Q) \leq \int_0^R \rho^4 \int_{\partial B_1^{\Sz}} g\left(\alpha \rho\dist_1(Q, \, \NN^\prime_1) \right) \, \d\H^4(Q) \, \d\rho.
 \]
 Now, we apply Lemma~\ref{lemma:claim H1/2 D} with~$\mathscr M = \partial B_1^{\Sz}$,~$\NN^\prime = \NN^\prime_1$ and $h\colon s\mapsto g(\alpha\rho s)$.
 We find constants~$\delta$ and~$C$ such that, letting~$U_\delta$ be the $\delta$-neighborhood of~$\NN^\prime_1$ in~$\partial B_1^{\Sz}$ and~$V_\delta : = \partial B_1^{\Sz}\setminus U_\delta$, we have
 \[
 \begin{split}
  \int_{B_R^{\Sz}} &(g\circ\phi)(Q)\, \d\H^5(Q) \\
  &= \int_0^R \rho^4 \left\{ \int_{U_\delta} g\left(\alpha \rho\dist_1(Q, \, \NN^\prime_1) \right) \, \d\H^4(Q) +
  \int_{V_\delta} g\left(\alpha \rho\dist_1(Q, \, \NN^\prime_1) \right) \, \d\H^4(Q)\right\} \, \d\rho \\
  &\leq C \int_0^R \rho^4 \left\{ \int_0^{C\delta} s g(\alpha\rho s) \, \d s +  g(\alpha\rho\delta) \, \H^4(V_\delta) \right\} \, \d\rho
 \end{split}
 \]
 (to bound the integral on~$V_\delta$, we use again that~$g$ is decreasing).
 Now, the two terms can be easily handled by changing the variables and using Fubini-Tonelli theorem:
 \[
  \begin{split}
   \int_{B_R^{\Sz}} (g\circ\phi)(Q)\, \d\H^5(Q) &\leq \alpha^{-2}C \int_0^R  \rho^2 \int_0^{\alpha\rho\delta C} tg(t) \,\d t\,\d\rho 
   + (\alpha\delta)^{-5}C \H^4(V_\delta) \int_0^{\alpha\delta R} t^4 g(t) \, \d t \\
   &\leq C_{\alpha, \delta, R} \int_0^{C_{\alpha, \delta, R}} \left(t + t^4\right) g(t) \, \d t .
  \end{split}
 \]
 Since~$\alpha$,~$\delta$ depend only on~$\phi$ and~$\NN^\prime_1$, the lemma is proved.
\end{proof}

\subsection{Proof of Propositions~\ref{prop:intro-H} and~\ref{prop:intro-log}}
\label{subsect:datoH1}

Condition~\eqref{hp:H} may be satisfied even if the boundary datum is not $\NN$-valued.
As we show in this section, if the doman satisfies a topological condition~\eqref{hp:domain} and
if boundary datum is smooth (at least~$H^1\cap L^\infty$), satisfying a uniform~$L^\infty$-bound and a logarithmic energy estimate~\eqref{hp:bd data},
then \eqref{Linfty}--\eqref{log energy} are satisfied for some constant~$M = M(\Omega, \, M_0) > 0$.
This proves Proposition~\ref{prop:intro-H}.
We remark that~\eqref{hp:domain} is a technical assumption which is needed in the proof, and we have not investigated whether it is necessary for~\eqref{hp:H}.
At the end of the section, we also prove Proposition~\ref{prop:intro-log}.

Once again, the $L^\infty$-bound~\eqref{Linfty} follows immediately from~\eqref{hp:bd data} and Lemma~\ref{lemma:Linfty g}.
Next, we prove~\eqref{log energy} by constructing an admissible comparison function whose energy is controlled by the right-hand side of~\eqref{log energy}.
If~$\Omega$ is a ball, it sufficies to extend homogeneously the boundary data. 
Since~$\Omega$ is bilipschitz equivalent to a handlebody by~\eqref{hp:domain}, we reduce to the case of a ball by cutting each handle of~$\Omega$ along a meridian disk.
This technique was used already in~\cite[Lemma~1.1]{HKL}.
In the following lemma, we construct cut disks with suitable properties.

\begin{lemma} \label{lemma:domain disks}
 Assume that~\eqref{hp:domain} and~\eqref{hp:bd data} hold.
 There exists a finite number of properly embedded disks\footnote{By saying that~$D_i$ is \emph{properly} embedded, we mean that~$\partial D_i = D_i \cap \partial \Omega$ and~$D_i$
 is transverse to $\partial \Omega$ at each point of~$\partial D_i$.}
 $D_1, \, D_2, \ldots, \, D_k\subseteq\Omega$ such that $\Omega\setminus \cup_{i=1}^k D_i$ is diffeomorphic to a ball,
 \begin{equation} \label{g energy disk}
  E_{\varepsilon}(g_{\varepsilon}, \, \partial D_i) \leq C \left( \abs{\log{\varepsilon}} + 1 \right)
 \end{equation}
 and
 \begin{equation} \label{g conv disk}
  \dist(g_{\varepsilon}(x), \, \NN) \to 0 \qquad \textrm{uniformly in } x\in \bigcup_{i = 1}^k \partial D_i .
 \end{equation}
\end{lemma}
\begin{proof}
For each handle~$i$ of~$\Omega$, there is an open set $U_i$ such that $\partial\Omega\cap U_i$ is foliated by
\[
 \partial \Omega \cap U_i = \coprod_{-a_0 < a < a_0} \partial D_i^a ,
\]
where the generic $D_i^a$ is a properly embedded disk, which cross transversely a generator of $\pi_1(\Omega)$ at some point. Then, Fatou's lemma implies that
\[
 \int_{-a_0}^{a_0} \liminf_{\varepsilon\to 0} E_\varepsilon(g_\varepsilon, \, {\partial D_i^a}) \, \d a 
 \leq \liminf_{\varepsilon\to 0} \int_{-a_0}^{a_0} E_\varepsilon(g_\varepsilon, \, {\partial D_i^a}) \, \d a \stackrel{\eqref{hp:bd data}}{\leq} C \left( \abs{\log\varepsilon} + 1 \right) ,
\]
so, by an average argument, we can choose the parameter~$a$ in such a way that $D_i := D^a_i$ satisfies~\eqref{g energy disk}.
Then,~\eqref{g conv disk} is obtained by the same arguments as Lemma~\ref{lemma:conv_grid}.
(As in the lemma, we apply Sobolev embedding inequality not on $\partial D_i$ directly, but on $1$-cells $K\subseteq \partial D_i$ of size comparable to $\varepsilon^\alpha\abs{\log\varepsilon}$.)
Furthermore, by construction $\Omega \setminus \cup_{i=1}^k D_i$ is a ball, since we have removed a meridian disk for each handle of $\Omega$.
\end{proof}

\begin{proof}[Proof of Proposition~\ref{prop:intro-H}]
The~$L^\infty$-bound~\eqref{Linfty} holds by virtue of Lemma~\ref{lemma:Linfty g}, so we only need to prove~\eqref{log energy}.
Assume for a moment that $\Omega = B_1$. In this case, define the function
\begin{equation} \label{homogeneous_ext}
 P_\varepsilon(x) := g_\varepsilon\left(\frac{x}{|x|}\right) \qquad \textrm{for } x \in B_1 .
\end{equation}
Then $P_\varepsilon\in H^1_{g_\varepsilon}(B_1, \, \Sz)$ and we easily compute
\[
\begin{split}
 E_\varepsilon(P_\varepsilon) = \int_0^1 \int_{\partial B_1} \left(\abs{\nabla_\top g_\varepsilon}^2 + \varepsilon^{-2}\rho^2 f(g_\varepsilon)\right) \, \d \H^2  \, \d\rho 
 \stackrel{\eqref{hp:bd data}}{\leq} C \left( \abs{\log \varepsilon} + 1 \right),
 \end{split}
\]
so the lemma holds true when $\Omega = B_1$.

Now, arguing as in~\cite[Lemma~1.1]{HKL}, we prove that the general case can be reduced to the previous one.
Let~$\Omega$ be any domain satisfying~\eqref{hp:domain}, and let $D_1, \, \ldots, \, D_k$ be the disks given by Lemma~\ref{lemma:domain disks}.
By~\eqref{g conv disk}, there exists $\varepsilon_0 > 0$ such that, for any $0 < \varepsilon \leq \varepsilon_0$ and any $x\in\cup_i \partial D_i$,
\[
 \dist(g_\varepsilon(x), \, \NN) \leq \delta_0 .
\]
For ease of notation, for a fixed $i\in\{1, \, \ldots, \, k\}$ we assume, up to a bilipschitz equivalence, that $D_i = B_1^2$.
Then, we define $\hat g_{\varepsilon, i} \colon B^2_1 \to \Sz$ by
\[
 \hat g_{\varepsilon, i}\left(x\right) := \begin{cases}
                                         \dfrac{ \delta_0 + \abs{x} - 1}{\delta_0} g_\varepsilon\left(\dfrac{x}{|x|}\right) +  \dfrac{1 - \abs{x}}{\delta_0} (\RR\circ g_\varepsilon)\left(\dfrac{x}{|x|}\right) & \textrm{if } 1 - \delta_0 \leq \abs{x} \leq 1 \\
                                         v_\varepsilon\left(\dfrac{x}{1 - \delta_0}\right) & \textrm{if }  \abs{x} \leq 1 - \delta_0 ,
                                        \end{cases}
\]
where $v_\varepsilon\in H^1(B^2_1, \, \Sz)$ is the extension of~${\RR\circ g_\varepsilon}_{|\partial B^2_1}$ given by Lemma~\ref{lemma:extension3}.
By a straightforward computation, one checks that
\begin{equation} \label{hat g}
 E_\varepsilon(\hat g_{\varepsilon, i}, \, D_i) \leq C \left( E_\varepsilon(g_\varepsilon, \, \partial D_i) + \abs{\log\varepsilon} + 1\right) .
\end{equation}
Now, consider two copies~$D_i^+$ and~$D^-_i$ of each disk~$D_i$.
Let $\Omega^\prime$ be a smooth domain such that
\[
 \Omega^\prime \simeq \left( \Omega\setminus \cup_i D_i\right) \cup_i D_i^+ \cup_i D^-_i ,
\]
and let $\varphi\colon\Omega^\prime \to \Omega$ be the smooth map which identifies each~$D^+_i$ with the corresponding~$D_i^-$ (see Figure~\ref{fig:cut_sphere}).
This new domain is simply connected, and in fact is diffeomorphic to a ball. Up to a bilipschitz equivalence, we will assume that~$\Omega^\prime$ \emph{is} a ball.
We define a boundary datum~$g_\varepsilon^\prime$ for~$\Omega^\prime$ by setting $g_\varepsilon^\prime := g_\varepsilon$ on~$\Omega\setminus \cup_i D_i$, and $g_\varepsilon^\prime := g_{\varepsilon, \, i}$ on~$D_i^+ \cup D_i^-$.
Then,~\eqref{hat g},~\eqref{g energy disk} and~\eqref{hp:bd data} imply
\[
 E_\varepsilon(g_\varepsilon^\prime, \, \partial\Omega) \leq C \left( E_\varepsilon(g_\varepsilon, \, \partial\Omega) + \abs{\log\varepsilon} + 1\right) \leq C \left( \abs{\log\varepsilon} + 1\right) .
\]
Then Formula~\eqref{homogeneous_ext} gives a map~$P_\varepsilon^\prime\in H^1_{g_\varepsilon^\prime}(\Omega^\prime, \, \Sz)$ which satisfies
\[
  E_\varepsilon(P^\prime_\varepsilon, \, \Omega^\prime) \leq C \left( \abs{\log\varepsilon} + 1 \right) .
\]
Since ${P^\prime_\varepsilon}_{| D_i^+} = {P^\prime_\varepsilon}_{| D_i^-}$ for every $i$, the map~$P^\prime_\varepsilon$ factorizes through~$\varphi$, and defines a new function~$P_\varepsilon\in H^1_{g_\varepsilon}(\Omega, \, \Sz)$ such that
\[
 E_\varepsilon(P_\varepsilon, \, \Omega) \leq C \left( \abs{\log\varepsilon} + 1 \right) .
\]
By comparison, we conclude that~\eqref{log energy} holds for any~$0 < \varepsilon \leq \varepsilon_0$.
Now, fix~$\varepsilon_0 < \varepsilon < 1$ and consider the ($\Sz$-valued) harmonic extension~$\tilde P_\varepsilon$ of~$g_\varepsilon$.
There holds
\[
 \|\nabla\tilde P_\varepsilon\|^2_{L^2(\Omega)} \leq C \norm{\nabla g_\varepsilon}^2_{L^2(\partial\Omega)} \stackrel{\eqref{hp:bd data}}{\leq} C \left(\abs{\log\varepsilon} + 1\right),
 \qquad \|\tilde P_\varepsilon\|^2_{L^\infty(\Omega)} \leq C \norm{g_\varepsilon}^2_{L^\infty(\partial\Omega)} \stackrel{\eqref{hp:bd data}}{\leq} C
\]
and so
\[
 E_\varepsilon(P_\varepsilon) \leq \frac12 \|\nabla\tilde P_\varepsilon\|^2_{L^2(\Omega)} + C\varepsilon_0^{-2} \leq C \left(1 + \varepsilon_0^2\right) \left(\abs{\log\varepsilon} + 1\right).
\]
Also in this case, the lemma follows by comparison.
\end{proof}

\begin{figure} \vspace{.5cm}
 \centering
 \includegraphics[height=.20\textheight]{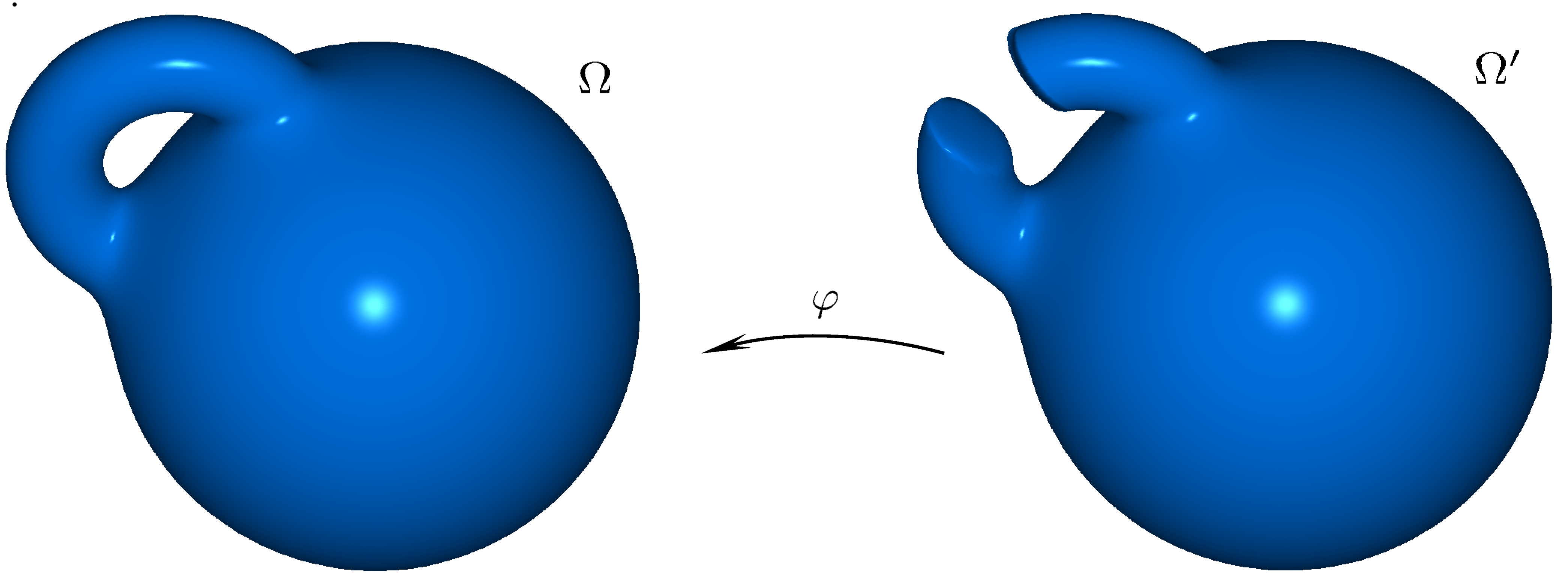}
  \caption{On the left, a ball with one handle. On the right, the corresponding domain~$\Omega^\prime$: the handle has been cut along a disk.
 The map~$\varphi\colon\Omega^\prime\to\Omega$ identifies the opposite disks in the handle cut.}
 \label{fig:cut_sphere}
\end{figure}

We turn now to the proof of Proposition~\ref{prop:intro-log}.
The boundary data we construct are smooth approximations of a map $\partial\Omega \to \NN$ with at least one point singularity.
Then, the lower bound for the energy follows by the estimates of Subsection~\ref{subsect:JerrardSandier}.

\begin{proof}[Proof of Proposition~\ref{prop:intro-log}]
 Up to rotations and translations, we can assume that the $x_3$-axis~$\{ x_1 = x_2 = 0 \}$ crosses transversely~$\partial\Omega$ at one point~$x_0$ at least.
 Let $\eta_\varepsilon\in C^\infty(\R^+, \, \R)$ be a cut-off function satisfying
 \[
  \eta_\varepsilon(0) = \eta_\varepsilon^\prime(0) = 0, \qquad \eta_\varepsilon(r) = s_* \textrm{ for } r\geq \varepsilon, \qquad
  0 \leq \eta_\varepsilon \leq s_*, \qquad \abs{\eta_\varepsilon^\prime} \leq C\varepsilon^{-1} .
 \]
 Set
 \[
  g_\varepsilon(x) := \eta_\varepsilon(|x^\prime|)\left\{ \left(\frac{x^\prime}{|x^\prime|}\right)^{\otimes 2} -  \frac13 \Id\right\} \qquad \textrm{for } x\in\partial\Omega,
 \]
 where $x^\prime := (x_1, \, x_2, \, 0)$.
 Computing as in Lemma~\ref{lemma:extension3}, one sees that $\{g_\varepsilon\}$ satisfies~\eqref{hp:bd data}.
 It remains to prove that the energy of minimizers~$Q_\varepsilon$ satisfies a logarithmic lower bound.
 Take a ball~$B_r(x_0)$. If the radius~$r$ is small enough, the set~$\Omega\cap B_r(x_0)$ can be mapped diffeomorphically onto the half-ball
 \[
  U := \left\{x\in \R^3\colon \abs{x} \leq 1, \ x_3 \geq 0 \right\} ,
 \]
 so we can assume WLOG that~$\Omega\cap B_r(x_0) = U$.
 Let $U_s := \{x\in U\colon x_3 = s \}$, for $r/2 \leq s \leq r$. 
 The map ${\Qe}_{|\partial U_s}\colon \partial U_s \to \NN$ is a homotopically non-trivial loop, which satisfies~$E_\varepsilon(Q_\varepsilon, \, \partial U_s)\leq C$.
 Then, by applying Corollary~\ref{cor:lower bound} we deduce
 \[
  E_\varepsilon(\Qe, \, U_s) \geq \kappa_* \log\frac{s}{\varepsilon} - C
 \]
 for a constant~$C$ depending on~$r$,~$\Omega$. By integrating this bound for~$s\in(r/2, \, r)$, the proposition follows.
\end{proof}

\section{Coexistence of line and point singularities: an example}
\label{sect:SX}

In this section, we show through an example that both the set of line singularities~$\Sl$ and the set of point singularities~$\Spt$ may be non-empty.
We consider the following domain. For two fixed positive numbers
\begin{equation} \label{Lr}
 L > 0 \qquad \textrm{and} \qquad 0 < r < \frac{1}{2} ,
\end{equation}
define
\begin{gather*}
 p_{\pm} := (\pm(L + 1), \, 0, \, 0), \qquad \Omega_{\pm} := B_1(p_{\pm}) , \qquad
 \Omega_0 := \left([-L - 1, \, L + 1] \times B^2_r(0)\right) \setminus\left(\Omega_- \cup \Omega_+\right)
\end{gather*}
and~$\Omega := \Omega_- \cup \Omega_0 \cup \Omega_+$.
In other words, the domain consists of two balls joined by a cylinder about the~$x_1$-axis (see Figure~\ref{fig:ballcylinder}).
This is a Lipschitz domain; however, one could consider a smooth domain~$\Omega^\prime$, obtained from~$\Omega$ by ``smoothing the corners'', 
and the results of this section could be easily adapted to~$\Omega^\prime$.


We write~$\partial\Omega = \Gamma_- \cup \Gamma_0 \cup \Gamma_+$,
where~$\Gamma_{\pm} := \partial\Omega_{\pm}\setminus\Omega_0$ and~$\Gamma_0 := \partial \Omega_0 \setminus (\overline\Omega_+ \cup \overline\Omega_-)$.
We define the auxiliary functions~$\chi\in H^1(0, \, 2\pi)$, $\eta_\varepsilon\in H^1(0, \, \pi)$ and~$\xi_r\in H^1(0, \, \pi)$ by
\begin{gather*}
 \chi(\theta) := \begin{cases}
                  {\pi}/{2} - {3\theta}/{5} & \textrm{for } 0 \leq\theta\leq 5\pi/6 \\
                  0 & \textrm{for } 5\pi/6 < \theta < 7\pi/6 \\
                  {7\pi}/{10} - {3\theta}/{5} & \textrm{for } 7\pi/6 \leq \theta \leq 2\pi ,
                 \end{cases} \qquad 
 \eta_\varepsilon(\varphi) := \begin{cases}
                               \varepsilon^{-1}\varphi & \textrm{for } 0 \leq \varphi \leq \varepsilon \\
                               1 & \textrm{for } \varepsilon < \varphi < \pi - \varepsilon \\
                               \varepsilon^{-1}(\pi - \varphi) & \textrm{for } \pi - \varepsilon \leq \varphi \leq \pi 
                              \end{cases}
\end{gather*}
and
\[
 \xi_r(\tilde\varphi) := \begin{cases} 
                          0 & \textrm{if } 0 \leq \tilde\varphi \leq \arcsin r \\
                          \dfrac{\arcsin 2r \left(\tilde\varphi - \arcsin r \right)}{\arcsin 2r - \arcsin r}  & \textrm{if } \arcsin 2r < \tilde\varphi < \arcsin 2r \\
                          \tilde\varphi & \textrm{if } \arcsin 2r \leq \tilde\varphi \leq \pi .
                         \end{cases}
\]
Notice that
\begin{equation} \label{aux_funct}
 \abs{\eta_\varepsilon^\prime} \leq \varepsilon^{-1} \qquad \textrm{and} \qquad \abs{\xi_r^\prime} \leq 2 .
\end{equation}

The boundary datum~$g_\varepsilon$ is defined as follows.
We parametrize~$\Gamma_+$ using spherical coordinates~$(\theta, \, \varphi)\in [0, \, 2\pi]\times[0, \, \pi]$ centered at~$p_+$:
\[
 x_1 = L + 1 + \sin\varphi\cos\theta, \qquad x_2 = \sin\varphi\sin\theta, \qquad x_3 = \cos\varphi
\]
and define~$g_\varepsilon$ on~$\Gamma_+$ by
\begin{equation} \label{g+}
 g_\varepsilon\colon (x_1, \, x_2, \, x_3)\in\Gamma_+ \mapsto 
 s_* \eta_\varepsilon(\varphi) \left\{  \bigg(\mathbf{e}_1\cos\chi(\theta) + \mathbf{e}_2\sin\chi(\theta)\bigg)^{\otimes 2} 
 - \frac13\Id\right\}. 
\end{equation}
On~$\Gamma_-$, we use spherical coordinates about the~$x_1$-axis, i.e.~$(\tilde\theta, \, \tilde\varphi)\in [0, \, 2\pi]\times[0, \, \pi]$ given by
\[
 x_1 = - L - 1 + \cos\tilde\varphi, \qquad x_2 = \sin\tilde\varphi\cos\tilde\theta, \qquad x_3 = \sin\tilde\varphi\sin\tilde\theta
\]
and set
\begin{equation} \label{g-}
 g_\varepsilon\colon (x_1, \, x_2, \, x_3)\in\Gamma_- \mapsto
 s_* \left\{ \left(\mathbf{e}_1\cos\xi_r(\tilde\varphi) + \mathbf{e}_2 \sin\xi_r(\tilde\varphi)\cos\tilde\theta 
 + \mathbf{e}_3 \sin\xi_r(\tilde\varphi)\sin\tilde\theta\right)^{\otimes 2}- \frac13\Id \right\} .
\end{equation}
Finally, we set
\[
 g_\varepsilon := s_* \left(\mathbf{e}_1^{\otimes 2} - \frac13\Id \right) \qquad \textrm{on } \Gamma_0 .
\]
This defines a map~$g_\varepsilon\in H^1(\partial\Omega, \, \Sz)$ which is non-orientable with two point singularities on~$\Gamma_+$,
is constant on~$\Gamma_0$ and has a hedgehog-type behaviour on~$\Gamma_-$.
In Figure~\ref{fig:ballcylinder}, we represent the direction of the eigenvector associated with the leading eigenvalue of~$g_\varepsilon(x)$, for~$x\in\partial\Omega$.
Of course, one could regularize the functions~$\chi$, $\eta_\varepsilon$ and~$\xi_r$ so that the map~$g_\varepsilon$ is smooth; this would not affect our arguments.
As usual, pick a subsequence~$\varepsilon_n$ so that the measures~$\mu_{\varepsilon_n}$ defined by~\eqref{mueps} converge weakly$^\star$ in~$C(\overline\Omega)^\prime$ to a measure~$\mu_0$,
and let~$\Sl\subseteq\overline\Omega$ be the support of~$\mu_0$.
Let~$\Spt\subseteq\Omega\setminus\Sl$ be a set such that the sequence~$\{Q_{\varepsilon_n}\}_{n\in\N}$ is compact in~$C^0(\Omega\setminus(\Sl\cup\Spt), \, \Sz)$.
By Theorem~\ref{th:convergence}, such a set exists and is locally finite in~$\Omega\setminus\Sl$.
We will show the following result, which implies Proposition~\ref{prop:intro-SX}.

\begin{prop} \label{prop:SX}
 There exists~$L^*$ such that, if
 \[
  L \geq L^* 
 \]
 then~$\emptyset\neq\Sl\subseteq\left\{x_1 \geq 0\right\}$ and~$\Spt \cap\{x_1 \leq -L/2\} \neq \emptyset$.
\end{prop}

\begin{remark} \label{remark:X empty}
The presence of a point defect is \emph{not} forced by a topological obstruction. 
In other words, there exists maps~$P_\varepsilon\colon\Omega\to\Sz$ which satisfy~$P_\varepsilon = g_\varepsilon$ on~$\partial\Omega$ and converge to a map with a line singularity but no point singularity.
Indeed, let~$\varphi\colon \overline\Omega\to\overline B_1$ be a bilipschitz homeomorphism such that $\varphi(L, \, 0, \, \pm 1) = (0, \, 0, \, \pm 1)$. Then, the functions
\[
 P_\varepsilon(x) := g_\varepsilon\circ \varphi^{-1}\left(\frac{\varphi(x)}{\abs{\varphi(x)}}\right)
\]
converge a.e. to a map with a line singularity~$\Sl := \varphi^{-1}\left\{x_1 = x_2 = 0 \right\}$, but no point singularities.
The convergence also holds in $H^1_{\mathrm{loc}}(\Omega\setminus\Sl, \, \Sz)$.
\end{remark}
 

We split the proof of Proposition~\ref{prop:SX} into some lemmas.
Throughout the section, we use the symbol~$C$ to denote a generic constant, which does not depend on~$\varepsilon$, $L$ and $r$.

\begin{lemma} \label{lemma:SXbound}
 There exists a constant~$M$, \emph{independent of~$L$ and~$r$}, such that
 \[
  E_\varepsilon(Q_\varepsilon, \, \Omega) \leq M \left(\abs{\log\varepsilon} + 1 \right) \qquad \textrm{and} \qquad \norm{Q_\varepsilon}_{L^\infty(\Omega)} \leq M 
 \]
 for any~$0 < \varepsilon < 1$.
\end{lemma}
\proof
The~$L^\infty$-bound follows by Lemma~\ref{lemma:Linfty g}, since~$|g_\varepsilon(x)| \leq (2/3)^{1/2}s_*$ for a.e.~$x\in\partial\Omega$ and any~$0 < \varepsilon < 1$.
The energy bound follows by a comparison argument.
We define a map~$P_\varepsilon$ on~$\Omega_+$ and~$\Omega_-$ by homogeneous extension:
\[
 P_\varepsilon(x) := g_\varepsilon\left(p_+ + \frac{x - p_+}{\abs{x - p_+}}\right) \ \textrm{ if } x\in \Omega_+ , \qquad
 P_\varepsilon(x) := g_\varepsilon\left(p_- + \frac{x - p_-}{\abs{x - p_-}}\right) \ \textrm{ if } x\in \Omega_- ,
\]
whereas we set
\[
 P_\varepsilon(x) :=  s_* \left(\mathbf{e}_1^{\otimes 2} - \frac13\Id \right) \ \textrm{ if } x\in \Omega_0 . 
\]
(Here we assume that~$g_\varepsilon$ is defined also on~$\partial\Omega_\pm \setminus \Gamma_\pm$, by the same formulae~\eqref{g+}--\eqref{g-}.)
Then, the map~$P_\varepsilon$ is continuous and belongs to~$H^1(\Omega, \, \Sz)$.
Moreover, $E_\varepsilon(P_\varepsilon, \, \Omega_0) = 0$ since~${P_\varepsilon}_{|\Omega_0}$ is constant and takes values in~$\NN$.
A simple computation, based on~\eqref{aux_funct}, concludes the proof.
\endproof

For any~$s\in\R$, let~$U_s := \Omega\cap\{x_1 < s\}$ and~$G_\varepsilon(s) := E_\varepsilon(Q_\varepsilon, \, U_s)$.
Fubini-Tonelli theorem entails that~$G_\varepsilon^\prime(s) = E_\varepsilon(Q_\varepsilon, \, \Omega\cap\{x_1 = s\})$ for a.e.~$s$.

\begin{lemma} \label{lemma:SXbounded}
 There exist positive constants~$L_*$ and~$M$ such that, for any~$L \geq L_*$ and any~$0< \varepsilon \leq 1/2$, there holds
 \[
  G_\varepsilon(0) \leq M .
 \]
 In particular, if~$L \geq L_*$ then~$\Sl\subseteq\overline\Omega\setminus U_0$.
\end{lemma}
\proof
This proof is based on the same arguments as Proposition~\ref{prop:desiderata}.
Define the set
\[
 D^\varepsilon := \left\{s\in (0, \, L)\colon G^\prime_\varepsilon(s) \leq \frac{2M}{L} \left(\abs{\log\varepsilon} + 1\right) \right\} .
\]
By Lemma~\ref{lemma:SXbound} and an average argument, we know that
\begin{equation} \label{SXbounded2}
 \H^1(D^\varepsilon) \geq \frac{L}{2} .
\end{equation}
Moreover, there exists~$L_* >0$ such that, for any~$L \geq L_*$, any~$0 < \varepsilon \leq 1/2$ and any~$s\in D^\varepsilon$, there holds
\[
 G^\prime_\varepsilon(s) \leq \eta_0 \abs{\log\varepsilon}
\]
where~$\eta_0$ is the constant given by Proposition~\ref{prop:interpolation2}.
Therefore, for a fixed~$s\in D^\varepsilon$ we can apply Proposition~\ref{prop:interpolation2} to the map~$u_\varepsilon := {Q_\varepsilon}_{|\{s\}\times B^2_r}$.
Notice that~$u_\varepsilon$ is defined on a disk, whereas the maps we consider in Proposition~\ref{prop:interpolation2} are defined over a sphere.
However, since~$u_\varepsilon$ takes a constant value on the boundary, it can be identified with a map defined on a sphere by collapsing~$\{s\}\times \partial B^2_r$ into a point.
Setting~$h(\varepsilon) := \varepsilon^{1/2}|\log\varepsilon|$ and~$A_\varepsilon := (s - h(\varepsilon), \, s) \times B^2_r$, we find maps~$v_\varepsilon\colon \{s - h(\varepsilon)\}\times B^2_r\to\NN$
and~$\varphi_\varepsilon\colon A_\varepsilon\to\Sz$ such that
\begin{gather}
 \frac12 \int_{\{s - h(\varepsilon)\}\times B_r^2} \abs{\nabla v_\varepsilon}^2 \, \d\H^2 \leq G^\prime_\varepsilon(s), \qquad
 E_\varepsilon(\varphi_\varepsilon) \leq C h(\varepsilon) \abs{\log\varepsilon} . \label{SXbounded3}
\end{gather}

Now, consider the set~$V_s := [s - h(\varepsilon) - r, \, s - h(\varepsilon)]\times B^2_r$ (we assume that $L_* > 2$, so that $s - h(\varepsilon) - r > -L$ for~$\varepsilon \leq 1/2$)
and the map~$\tilde v_\varepsilon\in H^1(\partial V_s, \, \NN)$ given by~$\tilde v_\varepsilon := v_\varepsilon$ on $\{s\} \times B^2_r$,
\[
 \tilde v_\varepsilon := s_* \left(\mathbf{e}_1^{\otimes 2} - \frac13 \Id \right) \qquad \textrm{on } \partial V_s \setminus \left(\{s\} \times B^2_r\right) .
\]
Thanks to~\eqref{SXbounded3}, we have
\[
  \frac12 \int_{\partial V_s} \abs{\nabla \tilde v_\varepsilon}^2 \, \d\H^2 = \frac12 \int_{\{s\}\times B^2_r} \abs{\nabla v_\varepsilon}^2 \, \d\H^2 \leq G^\prime_\varepsilon(s) .
\]
Then, by applying Lemma~\ref{lemma:extension1} (which is possible because~$V_s$ is bilipschitz equivalent to a ball), we find a map~$w_\varepsilon\in H^1(V_s, \, \NN)$ 
such that $w_\varepsilon = \tilde v_\varepsilon$ on~$\partial V_s$ and
\begin{equation} \label{SXbounded4}
 \frac12 \int_{V_s} \abs{\nabla w_\varepsilon}^2 \leq C G^\prime_\varepsilon(s)^{1/2} ,
\end{equation}
for a constant~$C$ independent of~$\varepsilon$, $L$, $r$. (Here we have used that~$r < 1$.)
Finally, we define a map~$\tilde w_\varepsilon \in H^1(U_s, \, \Sz)$ as follows.
We set $\tilde w_\varepsilon := \varphi_\varepsilon$ on~$[s - h(\varepsilon), \, s]\times B^2_r$ and $\tilde w_\varepsilon := w_\varepsilon$ on $V_s$,
\[
 \tilde w_\varepsilon := s_* \left(\mathbf{e}_1^{\otimes 2} - \frac13 \Id \right) \qquad \textrm{on } \Omega_0 \setminus \left( [s - h(\varepsilon) - r, \, s] \times B^2_r\right)
\]
and use an homogeneous extension to construct $\tilde w_\varepsilon$ on~$\Omega_-$:
\[
 \tilde w_\varepsilon(x) := g_\varepsilon\left(p_- + \frac{x - p_-}{\abs{x - p_-}}\right) \qquad \textrm{for } x\in\Omega_- .
\]
The map~$\tilde w_\varepsilon$ is continuous, satisfies $E_\varepsilon(\tilde w_\varepsilon) = 0$ on~$\Omega_0 \setminus ([s - h(\varepsilon) - r, \, s] \times B^2_r)$
and $E_\varepsilon(\tilde w_\varepsilon, \, \Omega_-) \leq C$ because of~\eqref{aux_funct}. Thus, from~\eqref{SXbounded3} and~\eqref{SXbounded4} we infer
\[
 E_\varepsilon(\tilde w_\varepsilon, \, U_s) \leq C G^\prime_\varepsilon(s)^{1/2} + C .
\]
Moreover, $\tilde w_\varepsilon$ is an admissible competitor for~$Q_\varepsilon$, because $\tilde w_\varepsilon = Q_\varepsilon$ on~$\partial U_s$.
Then, the minimality of~$Q_\varepsilon$ yields
\begin{equation} \label{SXbounded5}
 G_\varepsilon(s) \leq C G^\prime_\varepsilon(s)^{1/2} + C \qquad \textrm{for a.e. } s\in D^\varepsilon \textrm{ and every } 0 < \varepsilon \leq \frac12 . 
\end{equation}
Thanks to~\eqref{SXbounded2} and~\eqref{SXbounded5}, we apply Lemma~\ref{lemma:ODE} to~$y := G_\varepsilon$ and obtain
\[
 G_\varepsilon(0) \leq \left(1 + \frac{2}{L}\right)C \leq \left(1 + \frac{2}{L_*}\right)C =: M
\]
for every~$0 < \varepsilon \leq 1/2$.
Therefore, $\mu_{\varepsilon}\mres U_0 \to 0$ in~$\mathscr{M}_{\mathrm{b}}(U_0) := C_0(U_0)^\prime$ and hence~$\Sl\subseteq\overline\Omega\setminus U_0$.
\endproof
%
%

Before concluding the proof of Proposition~\ref{prop:SX}, we recall a well-known fact. 
Given~$r > 0$ and a continuous maps~$\n\colon B^2_r\to\S^2$ which take a constant value on~$\partial B^2_r$, it is possible to define the topological degree of~$\n$.
Indeed, the topological space which is obtained by collapsing the boundary of~$\partial B^2_r$ into a point is homeomorphic to a sphere.
Then, since~$\n_{|\partial B^2_r}$ is a constant,~$\n$ induces a continuous map~$\S^2\to\S^2$ whose homotopy class is characterized by an integer number~$d$ called the degree of~$\n$.
We will write~$d =: \deg(\n, \, B^2_r)$.
In case~$\n\in H^1(B^2_r, \, \S^2)$ takes a constant value at the boundary, the degree of~$\n$ can still be defined
(for instance, one can apply the VMO-theory by Brezis and Nirenberg~\cite{BN1, BN2}).

\begin{lemma} \label{lemma:area}
 For any~$r >0$ and any~$\n\in H^1(B_r^2, \, \S^2)$ with constant value at the boundary, if 
 \[
  \frac12 \int_{B^2_r} \abs{\nabla \n}^2 \, \d\H^2 < 4\pi 
 \]
 then~$\deg(\n, \, B^2_r) = 0$.
\end{lemma}
\proof
By applying the area formula, we obtain
\[
 \int_{B^2_r} \abs{\partial_{x_1}\n\times \partial_{x_2}\n} \, \d\H^2 = \int_{\n(B^2_r)} \H^0\left(\n^{-1}(y)\right) \, \d\H^2(y) \geq \H^2\left(\n(B^2_r)\right) .
\]
On the other hand, we have $|\partial_{x_1}\n\times \partial_{x_2}\n| \leq |\partial_{x_1}\n| |\partial_{x_2}\n|\leq |\nabla \n|^2/2$. Therefore, there holds
\[
 \frac12 \int_{B^2_r} \abs{\nabla \n}^2 \, \d\H^2 \geq \H^2\left(\n(B^2_r)\right) .
\]
If the left-hand side is~$< 4\pi$, then~$\n$ is not surjective and so~$\deg(\n, \, B^2_r) = 0$ (see e.g.~\cite[Property~1]{BN1}).
\endproof

\begin{proof}[Proof of Proposition~\ref{prop:SX}]
 Arguing as in the proof of Proposition~\ref{prop:intro-log}, and using that the boundary conditions~${g_\varepsilon}_{|U_+}$ are non orientable, one shows that
 \[
  E_\varepsilon(Q_\varepsilon) \geq C \left( \abs{\log\varepsilon} - 1\right)
 \]
for any~$\varepsilon$, $L$ and~$r$, so~$\Sl\neq\emptyset$. 
By Lemma~\ref{lemma:SXbounded}, there exists~$L_*$ such that~$\Sl\subseteq\Omega\setminus U_0$ if~$L\geq L_*$. Set
\begin{equation} \label{SX1}
 L^* := \max\left\{L_*, \, \frac{M}{\pi s_*} \right\}
\end{equation}
where~$M$ is given by Lemma~\ref{lemma:SXbounded}, and let~$L \geq L^*$.
The proposition follows once we show that~$\Spt\cap U_{-L/2}\neq \emptyset$.

By applying Lemma~\ref{lemma:SXbounded}, Theorem~\ref{th:convergence} and Corollary~\ref{cor:convergence_bd}, we deduce that~$Q_{\varepsilon_n}\to Q_0$ in~$H^1(U_{-\delta}, \, \Sz)$ for every~$\delta > 0$,
where~$Q_0\colon U_0\to\NN$ is a locally minimizing harmonic map. 
Passing to the limit as~$\varepsilon\to 0$ in Lemma~\ref{lemma:SXbounded}, we see that
\begin{equation} \label{SX2}
 \frac12 \int_{U_0} \abs{\nabla Q_0}^2 \leq M
\end{equation}
for an~$L$-independent constant~$M$. In particular,~$Q_0\in H^1(U_0, \, \NN)$.
An average argument, combined with Lemma~\ref{SX2}, shows that there exists~$-L < s < -L/2$ such that~$Q_0\in H^1(\{s\}\times B^2_r, \, \NN)$ and
\[
 \frac12 \int_{\{s\}\times B^2_r} \abs{\nabla Q_0}^2 \, \d \H^2 \leq \frac{4M}{L} .
\]
Due to Lemma~\ref{lemma:lifting}, we find a lifting of~${Q_0}_{|\partial U_s}$, 
i.e. a map~$\n_0 \in H^1(\partial U_s, \, \S^2)$ which satisfies~\eqref{lifting} and~$|\nabla Q_0|^2 = 2s_* |\nabla\n_0|^2$ a.e.
Then, we have
\begin{equation} \label{SX3}
 \frac12 \int_{\{s\}\times B^2_r} \abs{\nabla \n_0}^2 \, \d \H^2 \leq \frac{2M}{s_* L} .
\end{equation}
Combining~\eqref{SX3} with~\eqref{SX1}, we deduce
\[
 \frac12 \int_{\{s\}\times B^2_r} \abs{\nabla \n_0}^2 \, \d \H^2 \leq 2\pi .
\]
Moreover, $\n_0$ takes a constant value on the boundary of~$\{s\}\times B^2_r$, since~$Q_0$ does.
Then, by Lemma~\ref{lemma:area}, $\deg(\n_0, \{s\}\times B^2_r) = 0$.
On the other hand, $\deg(\n_0, \, \partial U_s \cap \Gamma_0) = 0$ since~$\n_0$ takes a constant value on~$\partial U_s \cap \Gamma_0$, 
and~$\deg(\n_0, \, \Gamma_-)$ can be computed explicitly thanks to~\eqref{g-}.
This yields
\[
 \deg(\n_0, \, \partial U_s) = \deg(\n_0, \, \Gamma_-) = \pm 1 ,
\]
so the map~${Q_0}_{|\partial U_s}$ is homotopically non-trivial and~$\Spt\cap U_s \neq \emptyset$.
\end{proof}

\medskip
\textbf{Acknowledgements.}
Most of this work has been carried out while the author was a PhD student at Sorbonne Universit\'es, UPMC --- Universit\'e Paris 6, CNRS, UMR 7598, Laboratoire Jacques-Louis Lions.
Part of the research leading to these results has received funding from the European Research Council under the European Union's Seventh Framework Programme (FP7/2007-2013) / ERC grant agreement n° 291053.
The author would like to thank his Ph.D. advisor, professor Fabrice Bethuel, as well as professor Sir John Ball FRS, for constant support and helpful advice.
He is also very grateful to professor Giandomenico Orlandi, doctor Arghir Zarnescu, Mark Wilkinson and Xavier Lamy for inspiring conversations and interesting remarks. 

\medskip
\textbf{Conflict of Interest.} The authors declare that they have no conflict of interest.
 
\bibliographystyle{plain}
\bibliography{3d}
 
\end{document}